\documentclass[final,onefignum,onetabnum]{siamart171218}

\definecolor{darkorchid}{rgb}{0.6, 0.2, 0.8}
\usepackage{amsfonts}
\usepackage{graphicx}
\usepackage{epstopdf}
\usepackage{algorithmic}

\usepackage{subfig}
\usepackage{mathabx}
\usepackage{todonotes}

\usepackage{mathrsfs}     % selects Times Roman as basic font
\usepackage{helvet}         % selects Helvetica as sans-serif font
\usepackage{courier}        % selects Courier as typewriter font
\usepackage{type1cm}      % activate if the above 3 fonts are
\usepackage{color}
\usepackage{mathbbol}
\usepackage{float}

\usepackage{geometry,calc,color}
\usepackage{amsmath,amssymb}
\usepackage{enumerate}

\usepackage{soul}%for strikethrough 

\ifpdf
  \DeclareGraphicsExtensions{.eps,.pdf,.png,.jpg}
\else
  \DeclareGraphicsExtensions{.eps}
\fi

\newsiamremark{remark}{Remark}
\newsiamremark{hypothesis}{Hypothesis}
\crefname{hypothesis}{Hypothesis}{Hypotheses}
\newsiamthm{claim}{Claim}

\headers{An a posteriori error estimate of the outer normal derivative
  using dual weights}{S. Bertoluzza, E. Burman, and C. He}

\title{An a posteriori error estimate of the outer normal derivative
  using dual weights\thanks{Submitted to the editors of SIAM Journal of Numerical Analysis.
\funding{EB and CH were funded by the EPSRC grant EP/P01576X/1.}}}
\author{Silvia Bertoluzza\thanks{Istituto di Matematica Applicata e Tecnologie Informatiche, CNR, Italy,
  (\email{silvia.bertoluzza@imati.cnr.it}).}
\and Erik Burman\thanks{Department of  Mathematics, University College London, UK, 
  (\email{e.burman@ucl.ac.uk}).}
\and Cuiyu He \thanks{School of Mathematical and Statistical Sciences, University of Texas Rio Grand Valley, USA, (\email{cuiyu.he@utrgv.edu}).}}

\usepackage{amsopn}

\renewcommand{\O}{\Omega}
\newcommand{\G}{\Gamma}
\renewcommand{\l}{\lambda}
\renewcommand{\d}{\delta}

\newcommand{\m}{\mu}
\renewcommand{\L}{\Lambda}
\newcommand{\e}{\eta}
\newcommand{\A}{A}
\newcommand{\V}{{\cal V}}

\renewcommand{\a}{{a\,}}

\newcommand{\divagrad}{{ \grad\cdot a \,\grad}}

\newcommand{\dimension}{d}
\newcommand{\COmega}{C(\Omega)}
\newcommand{\tBi}{{\widetilde B_i}}

\newcommand{\z}{\zeta}
\newcommand{\Th}{{\mathcal T}_h}

\newcommand{\bE}{{{\mathcal F}_h^b}}

\renewcommand{\r}{\rho}

\newcommand{\cunoT}{C_1}
\newcommand{\cdueT}{C_2}
\newcommand{\resuno}{{\bold {r}}} 
\renewcommand{\theta}{\vartheta}

\newcommand{\grad}{\nabla}
\newcommand{\RR}{\mathbb{R}}

\newcommand{\ClemProj}{\widehat{\Pi}_h}

\newcommand{\patchT}{{\Delta_T}}
\newcommand{\bVertex}{\mathcal{N}^b_h}
\newcommand{\patchP}{{\Delta_P}}
\newcommand{\rescinque}{\mathbf{r}}

\newcommand{\ProjLtwo}{\pi_h}

\newcommand{\face}{{{F}}}

\newcommand{\iEdges}{{{\mathcal{F}_h^i}}}
\newcommand{\bEdges}{{{\mathcal{F}_h^b}}}
\newcommand{\jump}[1]{\Lbrack #1 \Rbrack}

\ifpdf
\hypersetup{
  pdftitle={An a posteriori error estimate of the outer normal derivative
  using dual weights},
  pdfauthor={S. Bertoluzza, E. Burman, and C. He}
}
\fi

\newtheorem{example}{Example}

\begin{document}

\maketitle

% REQUIRED
\begin{abstract}
We derive a residual based a-posteriori error estimate for the outer
  normal flux of approximations to {the diffusion problem with variable coefficient}. By
  analyzing the solution of the adjoint problem, we show that error
  indicators in the bulk may be defined to be of higher order than
  those close to the boundary, which lead to more economic meshes. 
  The theory is illustrated with some numerical examples.
\end{abstract}

% REQUIRED


Let  {${\O} \subset \RR^{\dimension}$, % $d = 2,3$
	$\dimension = 2,3$,
	} be a polygonal/polyhedral domain, let $\G = \partial
{\O}$ denote its boundary and $\nu$ the outer unit normal.
We consider the following diffusion problem
\[
-\nabla \cdot  {a} \nabla u = f, \mbox{ in } \Omega,
\]
with non homogeneous Dirichlet boundary conditions, $u=g$ on $\Gamma$.  The
outer normal flux $\nu \cdot ({a} \nabla u)$ is an important quantity in many
applications. It is of importance for instance when a heat flux or an electric field on the
boundary of the domain needs to be approximated, or in fluid mechanics
for the fluid forces \cite{akira1986,gresho1987,nerg2001,dennis2001}. 
For boundary control problems, an accurate approximation of the
normal flux on the boundary also plays a critical role \cite{apel2015,apel2016}.
Recently there has been a number of works estimating the
error for the outer normal flux in the a priori sense. We refer
to \cite{HWM13,LM14}.

From the computational perspective it is appealing to apply adaptive
methods that concentrate degrees of freedom where they are most
needed to achieve a certain accuracy. In particular, for the normal
flux on the boundary, we expect perturbations in the bulk of the
domain to be less significant than those close to the boundary. 
This is proved in \cite{silvia} where local a priori error estimates were given for the
error in the outer normal flux. In particular, the error on the flux
quantity was shown to depend on the $H^1$-error in a tubular
neighborhood of the boundary and a global term that measures the
global error in a weak norm. Similar results using boundary
concentrated meshes were obtained more
recently in \cite{PW19}, where the application to a Dirichlet boundary control problem
was studied. A consequence of the localization property
underlying the above a priori error estimates is that a standard energy norm
estimate is unlikely to have optimal performance when approximating
the normal flux, since it does not
account for the relative independence of the goal quantity on 
perturbations in the bulk. It is however not
  straightforward to ensure accuracy of the boundary flux using a
  priori refinement in the boundary region alone, since geometric
  singularities or rough data nevertheless have to be taken into account.

The objective of the present work is to derive a residual based a posteriori error
estimate for the outer normal flux that exploits the
localization property. In particular, we add some mesh dependent weight in front of the classical residual based error estimator, and the weights greatly depend on the distance to the boundary.
More precisely, the domain is {implicitly} divided into two zones, a tubular neighborhood around the
boundary and an interior, bulk zone. For elements in the latter, the residual estimator is
multiplied with the mesh diameter to a higher power than in the
boundary region, hence giving it relative smaller weight. { To get a
precise quantification of the size of the weight 
we
consider an adjoint problem. Thanks to suitable weighted estimates we
 determine the rate of the
decrease of the adjoint solution and its derivatives 
with increasing
distance to the boundary.} This then helps provide bounds on the dual
weights in the a posteriori error estimate that allow us to decompose
the domain in a bulk and a boundary subdomain with associated error
indicators. 

The use of adjoint equations for the derivation of a
posteriori error estimates in weak norms was first proposed by
Eriksson and Johnson in \cite{EJ91}, in the case of $L^2$-norm bounds.
These ideas were generalized to the approximation of
fluxes and fluid forces using the Dual Weighted Residual a posteriori error estimation
approach (see for instance
\cite{BR96,becker1996, Giles1997AdaptiveEC, BKR00, BR01,richter2015}). In these approaches,
the dual solution was approximated, typically focussing on linear
functionals of the error. There has recently been an increased
interest in the convergence and optimality of goal
oriented adaptive methods \cite{BET11,HPZ15, FPZ16, HP16, BIP20, IP21, becker2021goal}. 
 With this work we show that when the target quantity of the
 computation is the outward normal flux, a detailed analysis of the adjoint equation can lead to a
posteriori bounds that perform better than the standard energy
estimate, but without the need of solving the dual problem,
numerically. Recall in this context that, when the target quantity driving the adaptive procedure is
 a norm of the error, the computation of the solution of the adjoint problem is complicated by the fact that the right hand side  depends on the 
error itself, and is therefore not directly available, contrary to what happens for instance when the target quantity is a given known functional of the solution, 
such as the value at a point or the integral over a line.

 Herein we only consider
the standard finite element setting where the domain is meshed with a
conforming triangulation. However, the arguments generalize in a
straightforward manner to a posteriori error estimates for fictitious domain
methods where elements are cut
\cite{BHL20}. To extend the method to adaptive standard fictitious domain
methods \cite{GG95} in the spirit of \cite{BBSV19}, or domain
decomposition methods, some more subtle arguments are needed. Indeed in
such situations, the boundary divides the computational domain
in two (or more) subdomains, thus requiring an analysis of the adjoint
solution,
similar to the one in this paper, for each subdomain and accounting for all
boundaries and interfaces of the problem. This is the topic of a
forthcoming paper.

An outline of the paper is as follows. First we introduce the weak
formulation of our model problem and the associated finite element
method
in section \ref{sec:local}. In section \ref{sec:apost} we derive the a
posteriori error estimate. Then we show in section \ref{sec:stab} how to
apply the results to some known stabilized methods, such as the
Barbosa-Hughes methods and Nitsche's method. Finally, we illustrate the
theory with some numerical examples in section \ref{sec:numerics}.

\section{The Lagrange multiplier formulation of the Dirichlet Problem}
\label{sec:local}
For $g \in H^{1/2}(\G)$ and $f \in L^2(\Omega)$ given, we consider the
problem of finding $u \in H^1({\O})$, ${\l} \in H^{-1/2}(\G)$ such that for all 
 $v \in H^1({\O})$, $\m \in H^{-1/2}(\G)$ 
 \begin{gather}
   \label{pbcont}
\int_{\O} a\, \grad u\cdot \grad v -  \int_\G {\l} v  =  \int_{\O} f v, \qquad\qquad
\int_\G u \m  = \int_\G g \m.
 \end{gather}
  { where $a\in C^\infty(\bar \Omega)$ is the diffusion coefficient, which for the sake of simplicity we assume to be scalar, satisfying
  $0 < \alpha \le a \le M$ for some constants $\alpha$ and $M$}.
We consider a Galerkin discretization of such problem. More precisely, letting
 $V_h
\subset H^1({\O})$, ${\L}_h \subset H^{-1/2}(\G)$ be finite element spaces defined on a shape regular triangulation \(\Th\). We look for  $u_h \in V_h$, ${\l}_h \in {\L}_h$ such that for all 
 $v_h \in V_h$, $\m_h \in {\L}_h$ 
 \begin{equation}
   \label{pbdisc}
\int_{\O} a \grad u_h\cdot \grad v_h -  \int_\G {\l}_h v_h  =  \int_{\O} f v_h, \qquad\qquad
\int_\G u_h \m_h = \int_\G g \m_h.
 \end{equation}

We assume that $V_h$ contains the space of continuous piecewise polynomials of order $k$ $(k \ge 0)$ on $\Th$, which we denote by $\widecheck V_h$, and that $\Lambda_h$ contains a subspace $\widecheck \L_h$ which is either the space of piecewise constants, or the space of continuous piecewise linears on the mesh induced on $\Gamma$ by $\Th$.

Restricting the test functions in \cref{pbcont} to the discrete
spaces and taking the difference of \cref{pbcont} and \cref{pbdisc} we see that the
following Galerkin orthogonality holds: for all 
 $v_h \in V_h$, $\m_h \in {\L}_h$ 
\begin{equation}\label{eq:gal_ortho}
\int_{\O} a\, \grad (u- u_h)\cdot \grad v_h -  \int_\G (\l-{\l}_h) v_h  = 0, \qquad\qquad
\int_\G (u - u_h) \m_h = 0.
\end{equation}

Observe that in the above we  are as general as possible in the definition of
the two spaces. We do not even need to assume that the spaces satisfy
the inf-sup condition required for the stability of \cref{pbdisc}. This of course does not mean that the method is
stable without it, only that the a posteriori error estimate will
measure the computational error independently of the stability
properties of the pair $V_h \times \Lambda_h$. An example of spaces that may be used in the
framework are
\begin{equation}\label{defVh}
V_h = \{
u \in H^1(\Omega): \ u|_T \in \mathbb{P}_k(T), \ \forall T \in \Th
\},
\end{equation}
and, for $k' \geq 0$
\begin{equation}\label{defLh1}
\Lambda_h = \{
\lambda \in L^2(\Gamma): \ u|_F \in \mathbb{P}_{k'}(F), \ \forall F \in \Th|_\Gamma \},
\end{equation}
or, for $k' \geq 1$ 
\begin{equation}\label{defLh2}
\Lambda_h = \{
\lambda \in C^0(\Gamma): \ u|_{F} \in \mathbb{P}_{k'}(F), \ \forall F \in \Th|_\Gamma \}.
\end{equation}
Also variants of the spaces \cref{defLh1} and \cref{defLh2} with local conforming enrichment on the boundary
to satisfy the inf-sup condition are valid \cite{brezzi1997}. 

{ \begin{remark}\label{rem1.1}  We point out that, for $k'= k$, the choice \cref{defLh2} for the multiplier space, coupled with the choice \cref{defVh}  for the approximation of the primal unknown (i.e., choosing $\Lambda_h = V_h|_{\Gamma}$)  yields a stable discretization of Problem \cref{pbcont}, equivalent to strongly imposing the Dirichlet boundary condition $u_h = \ProjLtwo g$, where $\ProjLtwo: L^2(\Gamma) \to V_h|_\Gamma$ is the $L^2(\Gamma)$ orthogonal projection. 
	Then, using $\lambda_h$ as an approximation to the normal flux is equivalent to compute the latter by post-processing with a variational approach as proposed, for instance, in \cite{PW19}.
Remark that, when the domain has corners, this method will not have, in general, optimal approximation for the multiplier, and it should be modified following the strategy used in the mortar method (see \cite{bernardi1993domain}), where discontinuity is allowed at the corners, with  $k'=k-1$  for those elements on the boundary mesh $\Th|_\Gamma$ which are adjacent to the corners, and  $k'=k$ for the remaining elements.  Observe,  however, that also for the suboptimal choice \cref{defLh2}, the estimator we are going to present, remains valid.
\end{remark}}

\section{A posteriori error estimates}
\label{sec:apost}
The a posteriori error estimate is derived in three steps.
We first derive an error representation using the adjoint problem. We then derive the local bounds for the adjoint solution and, finally, we obtain the weighted residual estimates. In what follows we will use the notation $A \lesssim B$ to indicate that $A \leq c B$ for some positive constant $c$ independent of mesh size parameters such as element diameters and/or face diameters or edge lengths. $A \simeq B$ will stand for $A \lesssim B \lesssim A$. 

\subsection{Error representation using duality}
\label{sec:Apost}
We let $$\A: (H^1(\O) \times H^{-1/2}(\G)) \times  (H^1(\O) \times H^{-1/2}(\G))  \to \mathbb{R}$$ be defined by
{
\begin{equation}\label{adjoint}
\A ( w,\e ; v,\z) = \int_{\O}\a \nabla w \cdot \nabla v -  \int_\G \e v + \int_\G w \z .
\end{equation}}
Let $(u,\lambda)\in H^1(\Omega)\times H^{-1/2}(\Gamma)$ be the solution of \cref{pbcont} and let $(u_h,\lambda_h)\in
V_h \times \Lambda_h$ satisfy \cref{pbdisc}. 
Set $e = u - u_h$ and $\d = \l - \l_h$. We define $L:H^{-1/2}(\Gamma) \to \mathbb{R}$ as
\[
L(\xi) := \| \d \|^{-1}_{-1/2,\G} ( \d , \xi )_{-1/2,\G}, \qquad \text{so that } \qquad  L(\d) = \| \d \|_{-1/2,\Gamma} ,
\]
where $(\cdot,\cdot)_{-1/2,\Gamma}$ is the scalar product for the space $H^{-1/2}(\Gamma)$, whose precise expression is provided  later in \cref{negative-half-norm-representation}, and where $\| \cdot \|_{-1/2,\Gamma}$ is the corresponding norm.
Define $(z,\z) \in \V = H^1(\O) \times H^{-1/2}(\G)$ as the solution of 
\begin{equation}\label{eq:weak_adjoint}
\A ( w,\e ; z,\z) = L(\e), \qquad \forall\ (w,\e) \in \V. 
\end{equation}
	Remark that the right hand side functional $L$ depends on the unknown error $\delta$, so that it is not possible to compute $z, \zeta$, even only approximately.
It is, however, easy to see that $|L(\xi)| \leq \| \xi \|_{-1/2,\Gamma}$, and then the operator $L$ has unitary norm. Therefore, by the stability of \cref{eq:weak_adjoint}, we have
\begin{equation}\label{dual-estimate}
\| z \|_{1,\Omega} \lesssim 1, \qquad \| \z \|_{-1/2,\G} \lesssim
1.
\end{equation}

Let $\iEdges$ and $\bEdges$ respectively denote the set of interior and boundary   $(d-1)$-dimensional facets of the triangulation $\Th$
and,  for an
element $T \in \Th$, let $\nu_T$ denote the outer unit normal to $\partial T$. On a $(d-1)$-dimensional facet  $F = \partial T^+ \cap \partial T^-$
we define the jump of the normal 
flux by $\jump{\a \partial_\nu u_h} =\a \nabla u_h^+ \cdot \nu_{T^+} + \a \nabla u_h^- \cdot \nu_{T^-}$. 

\begin{proposition}\label{prop:error_rep}(Error representation)
Let $\d = \l - \l_h$ and let $z, \z$ be the solution of \cref{eq:weak_adjoint}.
Then it holds that for any $z_h \in V_h$ and $\z_h \in \Lambda_h$
\begin{equation}\label{error-representation}
\begin{split}
\| \d \|_{-1/2,\G}  =& \sum_{T \in \Th}\int_T (f + \divagrad u_h) (z - z_h) - \sum_{F \in \iEdges} \int_{F}
\jump{\a \partial_\nu u_h} (z-z_h)\\ 
&+ \sum_{F\in\bEdges}\int_{F}(\l_h - \a \partial_\nu u_h)(z-z_h) +  \int_\Gamma (g - u_h)(\z -\z_h).
\end{split}
\end{equation}

\end{proposition}
\begin{proof}
Taking $w = e$ and $\e = \d$ in \cref{eq:weak_adjoint} we have
\[
\| \d \|_{-1/2,\G} = L(\d) = \A(e, \d; z, \z).
\]
Now, for $z_h \in V_h$, $\z_h \in \L_h$ arbitrary, thanks to Galerkin
orthogonality \cref{eq:gal_ortho} we can write:
\[
\| \d \|_{-1/2,\G}=
\A ( e , \d ; z - z_h, \z - \z_h) = I + II
\]
with
\[
I = \int_\O a \nabla e \cdot\nabla (z - z_h) , \;
II = - \int_\Gamma (a \partial_\nu u - \l_h) (z-z_h) +
 \int_\Gamma (g - u_h) (\z - \z_h) .
\]
For the term $I$ we obtain using Green's theorem
\begin{equation*}
\begin{split}
I =& \int_{\O}a \nabla e \cdot \nabla (z - z_h) 
= \sum_{T \in \Th}\int_T aa \nabla e  \cdot \nabla (z- z_h) \\ 
=& \sum_{T \in \Th} \left(  \int_T (f + \divagrad u_h) (z - z_h) + \int_{\partial
  T} a \nabla (u -  u_h)\cdot \nu_T (z - z_h) \right) \\
= &\sum_{T \in \Th} \int_T (f + \divagrad u_h) (z - z_h) -  \sum_{F \in \iEdges} \int_{F}
\jump{a \partial_\nu u_h} (z-z_h)  \\
&+ \sum_{F\in\bEdges}\int_{F} (a \partial_\nu u - a \partial_\nu u_h)(z-z_h) .
\end{split}
\end{equation*}

Combining all yields \cref{error-representation}. This completes the proof of the proposition. \end{proof}

\newcommand{\HL}[1]{#1^{\mathcal{H}}}
\newcommand{\Riesz}{\mathfrak{R}}
\newcommand{\Humo}{H^{1/2}_{\circ}(\Gamma)}
\newcommand{\dist}[1]{d_\Gamma(#1)}
\newcommand{\weightx}{\omega_x}
\newcommand{\Phix}{\Phi_x}

\subsubsection{Some observations on the operator $L$}
We start by observing that taking $v_h=1$ in \cref{eq:gal_ortho}
implies $\int_\Gamma \d  = 0$. Then we have
\[
\| \d \|_{-1/2,\Gamma} = \sup_{\phi\in H^{1/2}(\Gamma)} \frac{ \int_{\Gamma}\d \phi  }
{\| \phi \|_{1/2,\Gamma}}
\simeq \sup_{{\phi\in H^{1/2}(\Gamma)}\atop{\int_\Gamma\phi  = 0}} \frac{\int_{\Gamma}\d \phi  }{| \phi |_{1/2,\Gamma}}.
\]
On the space $\Humo = \{ \phi \in H^{1/2}: \ \int_\Gamma \phi = 0\}$ of zero average functions in $H^{1/2}(\Gamma)$, we can  define a scalar product and a norm, equivalent to the standard $H^{1/2}$ scalar product and norm,  as 
\[(\phi,\psi)_{1/2,\Gamma} = \int_{\Omega} \nabla \HL{\phi}\cdot \nabla \HL{\psi} , \qquad
| \phi |_{1/2,\Gamma}:= | \HL{\phi} |_{1,\Omega}, 
\]  
where $\HL{\phi} \in H^1(\Omega)$ denotes the harmonic lifting of $\phi$. We then let $\| \cdot \|_{-1/2,\Gamma}$ be defined by duality with respect to the above norm. We now let $\Riesz:  (\Humo)' \to \Humo$ denote the Riesz isomorphism, which, we recall, is defined as the solution of 
\[
(
\Riesz\lambda , \phi )_{1/2,\Gamma} = \int_\Gamma \lambda  \phi 
 \quad \forall \,\phi \in \Humo.
\]
We recall that, as $\Riesz$ is an isomorphism, we also have that
\begin{equation}\label{negative-half-norm-representation}
{
(\lambda,\mu)_{-1/2,\Gamma} = (\Riesz\lambda,\Riesz\mu)_{1/2,\Gamma}} = 
\int_{\Omega} \nabla \HL{(\Riesz \lambda)} \cdot \nabla\HL{(\Riesz \mu)} .
\end{equation}
It is now easy to check that, if $\mu \in L^ 2(\Gamma)$ satisfies $\int_\Gamma \mu = 0$, then $\HL{(\Riesz \mu)}$ is the unique solution \textcolor{red}{of}
\begin{equation}\label{Riesz_lam}
- \Delta \HL{(\Riesz \mu)} = 0 \text{ in }\Omega,\qquad \int_\Gamma \HL{(\Riesz \mu)} =0, \qquad \partial \HL{(\Riesz\mu)}/\partial\nu = \mu.
\end{equation}

Indeed for any function $v \in H^1(\Omega)$, there is a unique decomposition $v = \bar v + v_1+ v_0$ such that
$\bar v = |\Gamma|^{-1} \int_{\Gamma} v$, $v_0 \in \Humo$ is the harmonic extension of $v - \bar v$, and $v_1 \in H_0^1(\O)$ satisfies  $\triangle v_1 = \triangle v$.
Then we have that for any $v \in H^1(\Omega)$
\begin{equation}
	\begin{split}
		&\int_{\Omega} \nabla \HL{(\Riesz \mu)} \cdot \nabla v =  
		\int_{\Omega} \nabla \HL{(\Riesz \mu)} \cdot \nabla v_0  =
		(\Riesz \mu, v_0)_{1/2} = \int_{\Gamma }\mu v_0 = \int_{\Gamma}\mu v,	\end{split}
\end{equation}
which is the  weak form of equation \cref{Riesz_lam}.

\subsection{Local estimates for the adjoint solution $z$}
We observe that $z$ is the solution of the following problem.
\[
\int_\Omega a \nabla w \cdot \nabla z + \int_\Gamma w \zeta = 0, \qquad - \int_\Gamma \eta z = \| \delta \|_{-1/2,\Gamma}^{-1} (\delta,\eta)_{-1/2,\Gamma} = 
| \Riesz \delta |^{-1}_{1/2,\Gamma} \int_{\Gamma}\eta \,  \Riesz \delta.
\]
This rewrites as
\[
- \divagrad z = 0 \text{ in }\Omega, \qquad z = -  |\Riesz\delta
|^{-1}_{1/2,\Gamma}{\Riesz \delta}{} \text{ on }\Gamma.
\]

{ The following Lemma, whose proof we include for the sake of completeness, was proven in \cite{KhoromskijMelenk04}.
	
	\begin{lemma}\label{lem:z-bound}
	  Let $\dist{x}$ denote the distance of $x$ from $\Gamma$ and {let $w \in H^1(\Omega)$ satisfy  $\divagrad w =0$ in 
	  $\Omega$.}
	  Then, for all $p \geq 0$ it holds that
		\begin{equation}\label{lem:dual_bounds}
		\| d_\Gamma^{p+1} \nabla^{p+2} w \|_{0,\Omega} \lesssim | w |_{1,\Omega}.
		\end{equation}		
		\end{lemma}
		
		\begin{proof} 
	We start by proving a local bound. Let $B_R$ and $B_{cR}$, $ 0<c<1$  
	be two concentric balls of radius respectively $R$ and $cR$, and assume that $w \in H^1(B_R)$ satisfies $\divagrad w = 0$ in $B_R$. Then, we claim that for all $p \geq 0$ it holds that
		\begin{equation}\label{elliptic-bound}
		\| \nabla^{p+2} w \|_{0,B_{cR}} \lesssim R^{-p-1}\| \nabla w \|_{0,B_R} + R^{-p-2}\| w \|_{0,B_R}, 
		\end{equation}
where the implicit constant in the inequality depends on $c$. 	
We start by proving\cref{elliptic-bound} for $R = 1$. We prove it by induction on $p$. For $p = 0$, this is a consequence of \cite[Theorem 8.8]{GilTru01}. Let us now assume that the result is true for all $p \leq n-1$ and  prove it for $p = n$. We let $c' = 1-(1-c)/2 = c/2+1/2$, and let $\omega_c \in C^\infty_0(B_{c'})$, $\omega_c \geq 0$, $\omega_c = 1$ in $B_c$. We have
	\[
	\divagrad(\omega_c w) =  2 a \grad w \cdot \grad \omega_c + a w \Delta \omega_c + w \grad a \cdot \grad \omega_c,\quad \text{ in }B_{c'} \qquad \omega_c w = 0,\quad\text{ on }\partial B_{c'}.
	\]
	Using standard results on the smoothness of the solution of elliptic equations (see \cite{GilTru01}), by the induction assumption
	we have that 
	\begin{equation*}
	\begin{split}
	\| \nabla^{n+2} w \|_{0,B_c} &\leq  \| \nabla^{n+2}( \omega_c w  ) \|_{0,B_{c'}} 
	\lesssim
	 \| 
	  2 a \grad w \cdot \grad \omega_c + a w \Delta \omega_c + w \grad a \cdot \grad \omega_c
	 \|_{n,B_{c'}}\\ &\lesssim  \| \nabla w \|_{n,B_{c'}} +  \| w \|_{n,B_{c'}}\lesssim 
	\| w \|_{1,B_1},
	\end{split}
	\end{equation*}
which proves our claim for $R=1$. By rescaling we immediately obtain \cref{elliptic-bound}.

Let us now prove \cref{lem:dual_bounds}. We consider a covering of $\Omega$, consisting of a countable collection of balls $B_i = B_{r_i}(x_i)\subset \Omega$, of center $x_i$ and radius $r_i$, with  $r_i = \tilde c \dist{x_i}$ for some fixed $0 <\tilde c <1$, 
such that
\begin{enumerate}
	\item there exist $N \in \mathbb{N}$ such that all $x \in \Omega$ belong to at most $N$ balls $B_i$; \label{covering1}
	\item for some $ 0 < c <1$ independent of $i$, letting $\tBi \subset\subset B_i$ denote the ball of center $x_i$ and radius $c r_i$, it holds that $\Omega \subseteq \cup_{i} \tBi$,\label{covering2}
\end{enumerate}
(by the Besicovitch covering Theorem, such a collection exists). 
We observe that the relation between the radius of the balls in our covering and the distance of the centers from the boundary of the domain implies that for all $i$, $x \in \tBi$ implies $\dist{x} \simeq r_i$.
Then, letting $w_i = |B_i|^{-1} \int_{B_i} w$ denote the average of $w$ in $B_i$, using \cref{elliptic-bound}  and a Poincar\'e inequality, we can write 
\begin{equation*}
\begin{split}
\| d_\Gamma^{p+1} \nabla^{p+2} w \|_{0,\Omega}^2 &\leq \sum_i \| d_\Gamma^{p+1} \nabla^{p+2} w \|_{0,\tBi}^2 
\lesssim \sum_i r_i^{2(p+1)}\|  \nabla^{p+2} w \|_{0,\tBi}^2 \\
 & \lesssim \sum_i r_i^{2(p+1)}\|  \nabla^{p+2} (w - w_i) \|_{0,\tBi}^2 \lesssim  \sum_i (|w - w_i |^2_{1,B_i} + r_i^{-2} \| w - w_i \|_{0,B_i}^2)\\
& \lesssim \sum_i |  w |_{1,B_i}^ 2 \lesssim | w |_{1,\Omega}^2,
\end{split}
\end{equation*}
which concludes the proof.		
\end{proof}
}

\subsection{The a posteriori error estimator}
Using the error representation of \cref{prop:error_rep} and
the local bounds for the adjoint solution stated in 
\cref{lem:dual_bounds}, we will now derive the a posteriori error
estimation.
Comparing to the classical residual based error indicator,
our local error indicators for each element/{facet} are additionally multiplied by local
dual weights depending on the distance from the element/{facet} to the boundary. Let us
first introduce some notations that will be useful for the bounds.

We let $h_T$ (resp. $h_{F}$) denote the diameter  of an element $T$ (resp. of a $(d-1)$-dimensional facet $F$) in $\Th$. For a given element $T \in \Th$, $\patchT$ denotes the patch of
elements that have at least a vertex in common with $T$.
 The distance of an element $T$ to the boundary
will be measured using $\r_T = \underset{x \in \Delta_T}{\min} \dist{x}$. That
is the shortest distance from the associated patch to the boundary.

We now let  $\ClemProj : H^1(\Omega) \to \widecheck V_h$ denote the Scott-Zhang projector, introduced in \cite{SZ90}. We recall that, for $1 \leq m \leq k+1$ it holds that
{\begin{gather}\label{boundclement}
	\| z - \ClemProj z \|_{0,T} + h_T | z - \ClemProj z |_{1,T} \lesssim h^{m}_T | z |_{m,\patchT}.
	\end{gather}
Using this bound for $m = 1$ and $m = k+1$ we have the following local interpolation bounds for the
adjoint solution.
\begin{lemma}\label{lem:clem}
Let $z_h = \ClemProj z$, then we have the following two bounds 
\begin{gather*}
\| z - z_h \|_{0,T} + h_T | z - z_h |_{1,T} \leq  \cunoT h_T | z |_{1,\patchT}, \\
{ \| z - z_h \|_{0,T} + h_T | z - z_h |_{1,T} \leq  C_2  h_T^{k+1} \r_T^{-k} \| d_\Gamma^{k}\, \nabla^{k+1} z \|_{0,\patchT}}
\end{gather*}
The constants $\cunoT$ and $C_2$ depend on the shape regularity of the mesh 
\end{lemma}
{ \begin{proof}
The first inequality is immediate by \cref{boundclement} with
$m=1$. The second inequality trivially holds for $T$ with $\Delta_T$ adjacent to the boundary (for which $\rho_T^{-k} = \infty$). For the elements for which $\Delta_T$ is interior to $\Omega$, it follows by first applying \cref{boundclement}
with $m=k+1$, then multiplying and dividing by $d_\Gamma^k$, and finally bounding $d_\Gamma^{-k} \leq \rho_T^{-k}$:
% which, for the elements for which $\Delta_T$ is interior to $\Omega$, it follows by first applying \cref{boundclement}
%with $m=k+1$, then multiplying and dividing by $d_\Gamma^k$, and finally bounding $d_\Gamma^{-k} \leq \rho_T^{-k}$:
\[
\| z - z_h \|_{0,T} + h_T | z - z_h |_{1,T} \lesssim h^{k+1}_T | z |_{k+1,\patchT}
\lesssim h^{k+1}_T \rho_T^{-k} |d_\Gamma^{k} z |_{k+1,\patchT}.
\]
\end{proof}}

\newcommand{\resdue}{\mathbf{r}_0}
\newcommand{\restre}{\mathbf{r}_{1}}
\newcommand{\resquattro}{\mathbf{r}_{2}}

{Let us at first assume that we have $g \in H^1(\Gamma)$. Under such an assumption we have the following theorem.}
\begin{theorem}\label{thm:apost} 
	Define the  following local residuals:
	\begin{equation}\label{defresquattro}
	\begin{split}
		&\resuno(T) =  h_T\| f + \divagrad u_h \|_{0,T}, \quad \forall T \in \mathcal{T}_h,\\
		&\resdue(F) =  h_{F}^{1/2} \| \jump{a \partial_\nu u_h} \|_{0,F}, \quad \forall F \in \iEdges,\\
		&\restre(F) = h_{F}^{1/2}   \| \l_h - a \partial_\nu u_h \|_{0,F},\quad \forall F \in \bEdges, \\
		&\resquattro(F) = h_{F}^{1/2} | g - u_h |_{1,F},\quad \forall F \in \bEdges.
	\end{split}
	\end{equation}
	Then we have
\begin{equation}\label{reliability}
\| \lambda - \lambda_h \|_{-1/2,\G} \lesssim \sqrt{  \sum_{T\in \Th}
  \varsigma_T^2 |
  \resuno(T) |^2 + \sum_{F \in \iEdges}  \varsigma_{F}^2 | \resdue(F) |^2  + 
\sum_{F \in \bEdges} \left(
| \restre(F) |^2 + | \resquattro(F) |^2
\right)},
\end{equation}
where the element and facet weights $\varsigma_T$ and $\varsigma_{F}$ are defined by
\begin{equation}\label{eq:def_weight}
\varsigma_T=\min\{
\cunoT ,\cdueT  h^{k}_T { \r_T^{{-k}}} \},\quad 
{\varsigma_{F}=\min\{\varsigma_T,\varsigma_{T'}\}}, \mbox{ with } F = T \cap T'.
\end{equation}
\end{theorem}
\begin{proof}
Let us start by splitting $\Th$ as the union of two disjoint sets

\[
\Th^1 = \Big\{ T \in \Th: { \cunoT | z |_{1,\patchT} \leq C_2  h_T^{k} \r_T^{-k}  \| d_\Gamma^{k} \nabla^{k+1} z \|_{0,\patchT}}\Big\}, \qquad \Th^2 =
\Th \setminus \Th^1.
\]
Setting $z_h = \ClemProj z$ and $\z_h=0$ in the  error representation of  \cref{prop:error_rep}, we have
\begin{equation*}
\begin{split}
\| \d \|_{-1/2,\G}  =& \sum_{T \in \Th}\int_T (f + \divagrad u_h) (z - \ClemProj z) - \sum_{F \in \iEdges} \int_{F}
\jump{a \partial_\nu u_h} (z-\ClemProj z)\\
&+ \sum_{F\in\bEdges}\int_{F}(\l_h - a \partial_\nu u_h)(z-\ClemProj z)
+  \int_\Gamma (g - u_h)\z.
\end{split}
\end{equation*}

Observe that  \cref{lem:clem} gives us two error estimates for $\| z - \ClemProj z
\|_{0,T}$, and, depending on whether $T \in \Th^1$ or $T \in \Th^2$,
we apply the best possible estimate. This yields
\begin{equation*}
\begin{split}
&\sum_{T\in\Th} \int_T (f + \divagrad u_h)(z - \ClemProj z) \\
\lesssim&
\sum_{T \in {\Th^1}} \| f + \divagrad u_h \|_{0,T} \cunoT h_T | z |_{1,\patchT} + 
\sum_{T \in {\Th^2}} \| f + \divagrad u_h \|_{0,T} C_2  {  h_T^{k} \r_T^{-k}  \| d_\Gamma^{k} \nabla^{k+1} z \|_{0,\patchT}}\\
= &\sum_{T \in \Th^1} \varsigma_T\resuno(T) | z |_{1,\patchT} + \sum_{T \in \Th^2} \varsigma_T\resuno(T)  
{  \| d_\Gamma^{k} \nabla^{k+1} z \|_{0,\patchT}} \\
\leq& \sqrt{ \sum_{T\in \Th} \varsigma_T^2| \resuno(T) |^2 } \sqrt{
	\sum_{T \in \Th^1} | z |_{1,\patchT}^2 +  { \sum_{T \in \Th^2}   \| d_\Gamma^{k} \nabla^{k+1} z \|_{0,\patchT}^2}  }.
\end{split}
\end{equation*}
Applying
 \cref{lem:z-bound} 
we have
\[
	\sum_{T \in \Th^1} | z |_{1,\patchT}^2 +   { \sum_{T \in \Th^2}   \| d_\Gamma^{k} \nabla^{k+1} z \|_{0,\patchT}^2 }\lesssim
 \| z \|_{1,\O}^2 + {    \| d_\Gamma^{k} \nabla^{k+1} z \|_{0,\Omega}^2\lesssim  \| z \|_{1,\Omega} } \lesssim 1,
	\]
	so that
	\begin{equation}\label{A}
	\sum_{T\in\Th} \int_T (f + \divagrad u_h)(z - \ClemProj z) \lesssim   \sqrt{ \sum_{T\in \Th} \varsigma_T^2 | \resuno(T) |^2 },
	\end{equation}
	where the constant in the inequality depends on $\Omega$ and $a$.

A similar argument can be applied for interior facets. Letting $F \in \iEdges$, $F \subset \partial T$, the  standard bound holds
\begin{equation}
\begin{split}
&\int_{F}  \jump{a \partial_\nu u_h} (z-\ClemProj z) 
\leq \| \jump{a \partial_\nu u_h} \|_{0,F} \| z - \ClemProj z \|_{0,F}\\  
\lesssim  &\|\jump{a \partial_\nu u_h}\|_{0,F} \left( h_T^{-1/2} \| z - \ClemProj z \|_{0,T} + h_T^{1/2} |z - \ClemProj z|_{1,T}
\right) 
\leq   \| \jump{a \partial_\nu u_h} \|_{0,F}  h_T^{1/2} \cunoT | z |_{1,\patchT},
\end{split}
\end{equation}
as well as the enhanced bound
\begin{equation}
\int_{F} \jump{a \partial_\nu u_h} (z-\ClemProj z)
\leq   \| \jump{a \partial_\nu u_h} \|_{0,F} 
\cdueT  h_T^{k+1/2}  { \r_T^{-k}  \| d_\Gamma^{k} \nabla^{k+1} z \|_{0,\patchT}}.
\end{equation}
As for the cell contribution to the a posteriori estimate, we can
retain, for each facet, the more favorable estimator depending on whether
the facet $F$ belongs to an element  in $\Th^1$ or in $\Th^2$.
By similar argument to the ones used for the element residual term, we have
\begin{equation}\label{B}
- \sum_{F \in \iEdges} \int_{F}
\jump{a \partial_\nu u_h} (z-\ClemProj z) \lesssim 
 \COmega \sqrt{
 \sum_{e\in \iEdges} \varsigma_{F}^2
 |\resdue(F)|^2 }.
\end{equation}
The boundary terms are treated in the standard way for any $F \in \bEdges$ and $F \subset \partial T$,
\begin{equation}
\begin{split}
 &\int_{F}(\l_h - a \partial_\nu u_h)(z-\ClemProj z) \leq \| \l_h -  a \partial_\nu u_h \|_{0,F} \| z - \ClemProj z \|_{0,F}  
 \lesssim  \| \l_h -  a \partial_\nu u_h \|_{0,F} h_{F}^ {1/2} | z |_{1,\patchT}.
\end{split}
\end{equation}
Therefore,
\begin{equation}\label{C}
\sum_{F\in \bEdges }\int_{F}(\l_h -a \partial_\nu u_h)(z-\ClemProj z) \lesssim \left(\sum_{F\in \bEdges}h_{F}   \| \l_h - a \partial_\nu u_h \|^2_{0,F} \right)^{1/2} \|z\|_{1,\Omega} \le \left(
\sum_{F\in \bEdges}| \restre(F) |^2
\right)^{1/2}.
\end{equation}
 By \cref{dual-estimate}, the last term can be bounded as 
\[
\int_\Gamma (g - u_h)\z  \leq \| g - u_h \|_{1/2,\Gamma} \| \z  \|_{-1/2,\Gamma} \lesssim \|g - u_h \|_{1/2,\Gamma}.
\]
Finally, since $ g - u_h $ is orthogonal to $\widecheck \L_h
\subseteq \Lambda_h$, we can use Lemma 3 of \cite{bertoluzza2016}  to bound
\begin{equation}\label{boundfurbo}
\|g - u_h \|^2_{1/2,\Gamma} \lesssim 
\sum_{F\in\bEdges} h_{F} | g - u_h |_{1,F}^2
 =
\sum_{F \in \bEdges} | \resquattro(F) |^2.
\end{equation}
Combining all gives \cref{reliability}. This completes the proof of the theorem.
\end{proof}

{ If $g \in H^1(\Gamma)$,
	it is, therefore, natural, 	for the Lagrangian multiplier method, to define the following error indicator $\eta_T$ for each element $T \in \Th$, and estimator $\eta$, by
	\begin{equation}\label{eta}
	\begin{split}
	\eta_T &=  \sqrt{  
		\varsigma_T^2 |
		\resuno(T) |^2 + \sum_{F\in \iEdges \cap \partial T}  \varsigma_{F}^2 | \resdue(F) |^2  + 
		\sum_{F \in \bEdges \cap \partial T} \left(
		| \restre(F) |^2 + | \resquattro(F) |^2
		\right)},\\
	\eta &= 
	\sqrt{ \sum_{T \in \Th} \eta_T^2 }.
	\end{split}
	\end{equation}}

\newcommand{\Vertex}{\mathcal{X}_\Gamma}

{
If $g$ is not in $H^1(\Gamma)$, we can not get the full localization \cref{boundfurbo} of the residual $g - u_h$ on $\Gamma$. We can however resort, {for the two dimensional case, to \cite[Theorem 2.2]{faermann2000}, and, for the three dimensional case, to \cite[Lemma 3.1]{faermann2002},} which allow us to bound 
\[
\| g - u_h \|^2_{1/2,\Gamma} \lesssim \sum_{P \in \bVertex} | g - u_h |^2_{1/2,\patchP} 
\]
 where $\bVertex$ is the set of nodes of the mesh $\Th$ on $\Gamma$ and where for $P \in \bVertex$, $\patchP \subset \Gamma$ is the patch formed by the  boundary { facets} sharing $P$ as a vertex. 
For those patches $\patchP$ for which $g|_{\patchP} \in H^1(\patchP)$, $| g - u_h |_{1/2,\patchP}$ can be further bounded by $\underset{F \subset \patchP}{\sum} | \resquattro(F) |^2$. For the remaining patches, the $H^{1/2}(\patchP)$ semi-norm of the residual will have to be computed by evaluating the double integral involved in the definition of the fractional norm.}

{
	\begin{remark} 
			Note that, in the implementation of the method, we do not explicitly use the splitting $\mathcal{T}_h =  \mathcal{T}_h^1 \cup\mathcal{T}_h^2$, which is only needed for the theoretical analysis. 
	Remark also that, for the elements adjacent to $\Gamma$, for which $\rho^{-k}_T = \infty$, we always have $\varsigma_T = C_1$. 
	\end{remark}}

\section{Application to stabilized methods for the imposition of
  boundary conditions}\label{sec:stab}
In engineering practice it is often advantageous to use a stabilized
method instead of choosing the spaces so that the inf-sup condition is
satisfied. In this section we show how  the proposed framework can be adapted to two of the most well-known
stabilized methods, namely the Barbosa-Hughes
method \cite{BH91} and the Nitsche's method \cite{Nitsche1971berEV}.
{We assume for the sake of simplicity that $g \in H^1(\Gamma)$.}
 Both
the final results and the arguments are in the same spirit as 
\cref{thm:apost} above and therefore we only give sketches of the proofs.
\subsection{Indicators for the Barbosa--Hughes method}
The Barbosa--Hughes discrete problem reads: find $u_h \in V_h$, $\l_h \in \L_h$ such that for all $v_h \in V_h$, $\mu_h \in \L_h$ it holds that
\begin{gather}\label{BH1}
\int_\O a \nabla u_h \cdot \nabla v_h - \int_\Gamma \l_h v_h \pm \alpha \sum_{F \in \bEdges} h_{F} \int_{F} (a \partial_\nu u_h - \lambda_h)(a \partial_\nu v_h )= \int_\O f v_h,\\
\label{BH2}
\int_\Gamma u_h \mu_h - \alpha \sum_{F \in \bEdges}h_{F} \int_{F} (a \partial_\nu u_h - \lambda_h)  \mu_h = \int_\G g \mu_h.
\end{gather}
Here we use $\pm$ in front of the stabilization term in 
\cref{BH1}, to indicate that the analysis applies to both the
symmetric and antisymmetric version of the method.
The functional $L$ and $z,\z$ are defined as in the previous section. Similarly we have the following error representation by subtracting \cref{BH1} and \cref{BH2} from \cref{pbcont}: for arbitrary $z_h \in V_h$ and $\z_h \in \L_h$ it holds that 
\begin{equation*}
\begin{split}
L(\d) =& \int_\O a \nabla e \cdot \nabla z - \int_\G \d z + \int_\G e \z \\ 
=& 
\int_\O a \nabla e \cdot \nabla (z - z_h)  - \int_\G \d (z - z_h)+ \int_\G e (\z - \z_h) +
\int_\O a \nabla e \cdot \nabla z_h - \int_\G \d z_h + \int_\G e \z_h\\ 
=&
\int_\O a \nabla e \cdot \nabla (z - z_h)  - \int_\G \d (z - z_h)+ \int_\G e (\z - \z_h) -
\alpha \sum_{F\in \bEdges}h_{F} \int_{F} (a \partial_\nu u_h - \l_h) (\z_h \mp a \partial_\nu z_h).
\end{split}
\end{equation*}
\newcommand{\gPh}{g^P_{h}}
From \cref{prop:error_rep}, we have 
\begin{equation}\label{bound_BH}
\begin{split}
L(\d) =& \sum_{T \in \Th}\int_T (f + \divagrad u_h) (z - z_h) -
\sum_{F \in \iEdges} \int_{F}\jump{ a \partial_\nu u_h} (z-z_h)+ \sum_{F\in\bEdges}\int_{F}(\l_h - a \partial_\nu u_h)(z-z_h) \\
&+  \int_\Gamma (g - u_h)(\z -\z_h)  -
\alpha \sum_{F\in \bEdges}h_{F} \int_{F} (a \partial_\nu u_h - \l_h) (\z_h \mp a \partial_\nu z_h).
\end{split}
\end{equation}
We again set $z_h = \ClemProj z$, $\z_h = 0$.
The first three terms in \cref{bound_BH} can be bounded using 
\cref{A}, \cref{B} and \cref{C}.
However, for the fourth term in \cref{bound_BH}, contrary to the previous case,
we do not have that $u_h - g$ is orthogonal to the multiplier space,
therefore  \cref{boundfurbo} no longer holds. Instead we only have the following
weaker bound  \cite{faermann2000,faermann2002}.
Recall that $\bVertex$ denote the set of boundary vertices of the triangulation, and for each $P\in \bVertex$  denote by $\patchP \subset \Gamma$ the patch formed by the  boundary faces sharing $P$ as a vertex. We have 
 \begin{equation}\label{locH12}
\|  g - u_h \|_{1/2,\G}^2 \lesssim \sum_{P \in \bVertex} | u_h - g |^2_{1/2,\patchP} + \sum_{F \in \bEdges} h_{F}^{-1} \| u_h - g \|^2_{0,F}.
\end{equation}

We can further localize the term $| u_h - g |^2_{1/2,\patchP}$. In order to do so we add and subtract $\gPh \in \widecheck V_h|_\patchP$, where $\gPh$ is the $L^2(\patchP)$ projection onto the local space of continuous piecewise linears {on the $(d-1)$  dimensional local mesh $\Th|_\patchP$}, yielding
{
\[
|  g - u_h |^2_{1/2,\patchP} \lesssim  | u_h - \gPh |^2_{1/2,\patchP} + | \gPh - g |^2_{1/2,\patchP},
\]
}
which, combining with the inverse inequality, gives 
{
\[
| u_h -  \gPh  |^2_{1/2,\patchP} \lesssim { h_P^{-1} }\| u_h - \gPh  \|^2_{0,\patchP} \simeq \sum_{F \subseteq \patchP} h_{F}^{-1}\| u_h -  \gPh  \|^2_{0,F},
\] 
{where $h_P = \max_{F \in \patchP} h_F$ (remark that the shape regularity of the mesh implies that for all $F \subseteq \patchP$ we have $h_F \simeq h_P$).}
}
 Thanks to the fact that $g - g_h^P$ is orthogonal to the continuous piecewise linear functions, we have 
\[
 | \gPh - g |^2_{1/2,\patchP} \lesssim
 \sum_{F \subseteq \patchP} h_{F} | \gPh - g |^2_{1,F}.
\]
{
We also observe that, for $P$ a vertex of $F$
\[
\| g - u_h \|^2_{0,F} \lesssim \| g - \gPh \|^2_{0,F} + \| \gPh - u_h \|^2_{0,F}
\lesssim  \sum_{F \subseteq \patchP} h_{F} | \gPh - g |^2_{1,F} +  \| u_h - \gPh \|^2_{0,F}.
\]} 
Combining these bounds we easily obtain
\begin{equation}\label{D}
\begin{split}
\int_\Gamma (g - u_h)\z  &\leq \| g - u_h \|_{1/2,\Gamma}
\| \z \|_{-1/2,\Gamma} \lesssim \|g - u_h \|_{1/2,\Gamma}\\
&\lesssim \sqrt{\sum_{P \in \bVertex} \sum_{F \subseteq \patchP} (
h_{F}^{-1}  \| u_h - \gPh \|^2_{0,F} + h_{F} | \gPh - g |_{1,F}^2)}= \sqrt{\sum_{P \in \bVertex} \sum_{F \subseteq \patchP} |\rescinque(F,P)|^2},
\end{split}
\end{equation}
where, for $P$ a vertex of $F \subset\bEdges$ we define
\[
\rescinque(F,P) = \sqrt{
h_{F}^{-1}  \| u_h - \gPh \|^2_{0,F} + h_{F} | \gPh - g |_{1,F}^2} \,.
\]

Finally, we bound the additional term resulting from the stabilization, namely
\begin{equation}\label{bound_BH2}
\begin{split} 
\sum_{F\in \bE} h_{F} \int_{F} (a \partial_\nu u_h - \l_h) (a \partial_\nu (\ClemProj z)) 
& \leq \sqrt{
\sum_{F\in \bE}  h_{F} \| a \nabla u_h \cdot \nu - \l_h \|_{0,F}^2}
\sqrt{\sum_{F\in \bE} h_{F} \| a  \partial_\nu (\ClemProj z)  \|_{0,F}^2}\\
 \lesssim& \sqrt{
\sum_{F\in \bE}  h_{F} \|a \partial_\nu u_h - \l_h \|_{0,F}^2}.
\end{split}
\end{equation}
The last bound derives from 
 a standard trace inequality on the element $T$ associated to the
boundary face $F$ followed by an inverse inequality and an $H^1$ stability bound for $\ClemProj$
\begin{equation}
\begin{split}
\|a \partial_\nu (\ClemProj z ) \|_{0,F} &\lesssim \| \nabla(\ClemProj z) \|_{0,F} \lesssim h_T^{-1/2} \| \nabla (\ClemProj z) \|_{0,T} + h_T^{1/2} | \nabla (\ClemProj z) |_{1,T} \\ &\lesssim h_T^{-1/2} \| \nabla (\ClemProj z) \|_{0,T} \lesssim h_T^{-1/2}  | z |_{1,\patchT},
\end{split}
\end{equation}
which, together with (\ref{dual-estimate}), yields {
\[
\sum_{F\in\bEdges} h_{F} \|a \partial_\nu (\ClemProj z)  \|_{0,F}^2 \lesssim  \| z \|^2_{1,\Omega} \lesssim 1.
\] }
We then have
\begin{equation}\label{E}
\alpha \sum_{F\in\bEdges}
 h_{F} \int_{F} (a \nabla u_h \cdot \nu - \l_h) (a  \partial_\nu (\ClemProj z) )
 \lesssim \alpha \sqrt{
\sum_{F\in\bEdges}| \restre(F) |^2
	}
\end{equation}
where, we recall, $\restre(F) = h^{1/2}_F \|a \partial_\nu u_h - \l_h \|_{0,F}$.

Collecting the above bounds we obtain the a posteriori error estimate
for the Barbosa--Hughes formulation \cref{BH1}--\cref{BH2}:
\begin{equation}
\| \lambda - \lambda_h \|^2_{-1/2,\G} \lesssim  
\sum_{T\in \Th} \varsigma_T^2| \resuno(T) |^2
+
\sum_{F\in \iEdges}
\varsigma_{F}^2 |\resdue(F)|^2 
+(1+\alpha^2)
\sum_{F\in \bEdges}| \restre(F) |^2
+
\sum_{P \in \bVertex} \sum_{F \subseteq \patchP} |\rescinque(F,P)|^2.
\end{equation}

\subsection{Indicators for Nitsche's method}
Let us now consider Nitsche's method, which reads: find $u_h \in V_h$ such that for all $v_h \in V_h$, there holds
{
\begin{equation}\label{nitsche}
\begin{split}
&\int_\O a \nabla u_h \cdot \nabla v_h -\int_\G v_h (a \partial_\nu u_h ) \pm \int_\G u_h (a \partial_\nu v_h )+ \gamma \sum_{F\in \bEdges}h_{F}^{-1} \int_{F} u_h v_h \\
=& \int_\O f v_h  \pm \int_\G g (a \partial_\nu v_h) +   \gamma \sum_{F\in \bEdges}h_{F}^{-1} \int_{F} g v_h. 
\end{split}
\end{equation}
}

\newcommand{\resduebis}{\mathbf{r}_3}

Following the work of Stenberg \cite{Sten95}, { which focuses on the Poisson equation but which is easily adapted to \cref{pbcont},} Nitsche's method is equivalent to a Barbosa-Hughes method with the choice {$\Lambda_h = L^2(\G)$}. The solution $u_h,\l_h$ of \cref{BH1}-\cref{BH2} with $\L_h = L^2(\G)$ verifies that $u_h$ solves \cref{nitsche} with $\gamma = \alpha^{-1}$, and we have that, on $e \subset \G$, 
\begin{equation} \label{lambda-h-Nitsche}
\l_h =a \partial_\nu u_h + \gamma h_{F}^{-1} (g - u_h).
\end{equation}
Due to the equivalence, $L(\delta)$ has the same representation \cref{bound_BH} with $\alpha$ replaced by $\gamma^{-1}$.
With the same choice of $z_h$ and $\z_h$, the first and second terms can be bounded using \cref{A}  and \cref{B} respectively. The fourth term $\int_\Gamma (g - u_h)\z$ can be  bounded using \cref{D}. 
 For the remaining terms, observing that
\[
\restre(F) = h_{F}^{1/2}   \| \l_h - a \partial_\nu u_h \|_{0,F} = 
h_{F}^{1/2}  \| \gamma h_{F}^{-1} (g - u_h)  \|_{0,F} := \gamma \resduebis(F),
\]
with 
\[
\resduebis(F) = h_{F}^{-1/2} \| u_h - g \|_{0,F},
\]
which, combining with \cref{C} and \cref{E}, yields
{
\begin{equation}\label{CNitsche1}
	\sum_{F\in\bEdges}\int_{F}(\l_h -a \partial_\nu u_h)(z-\ClemProj z)   \mp
\gamma^{-1} \sum_{F \in \bEdges}  h_{F} \int_{F} (a \partial_\nu u_h - \l_h) (a \nabla (\ClemProj z) \cdot \nu)
\lesssim \sqrt{\sum_{F\in \bEdges} (1+\gamma^2)|  \resduebis(F) |^2}.
\end{equation}}
Collecting all, we obtain the following a posteriori
error bound for the normal flux computed using Nitsche's method.
{
\begin{equation}
\|a \partial_\nu u - \lambda_h\|_{-1/2,\G} \lesssim \sqrt{
\varsigma_T^2\sum_{T\in \Th} | \resuno(T) |^2 
+ \sum_{F \in \iEdges} \varsigma_{F}^2| \resdue(F) |^2  
+ \sum_{F \in \bEdges} (1+\gamma^2)  | \resduebis(F) |^2
+ \sum_{P \in \bVertex} \sum_{F \subseteq \patchP} |\rescinque(F,P)|^2},
\end{equation}
where $\lambda_h$ is given in \cref{lambda-h-Nitsche}.
}

Similarly as in \cref{eta}, we can define the corresponding error indicator $\eta_T$ and $\eta$ for the Barbosa--Hughes and Nitsche's methods.

\section{Numerical experiments}\label{sec:numerics}
{In this section we demonstrate the performance of the proposed error estimator on some simple, yet significant, two dimensional test cases.}
Firstly, we demonstrate the action of the weight $\varsigma_T$ defined in \cref{eq:def_weight}
in the adaptive mesh refinement procedure
 independently of any particular problem. For simplicity, we fix $C_1=1$.
 In the computation, we approximate $\rho_T$ by the following:
{
\[
	\rho_T \approx  \min_{x \in \mathcal{N}_{\triangle_T}} \dist{x}
\]
where $\mathcal{N}_{\triangle_T}$ is the set of all vertices on $\triangle_T$.}

We start with a $4$ by $4$ initial triangular mesh on a unit square domain, 
see \cref{fig:adaptive-w}(a). A total number of $7$ refinement steps are performed and the marking strategy identifies an element $K \in \mathcal{T}_h$ to be refined if 
\[
	\varsigma_T > 0.5  \varsigma_{T,max}, \quad  \mbox{where} \,  \varsigma_{T,max} = \max_{T \in \mathcal{T}_h}\varsigma_T .
\]
\cref{fig:adaptive-w} shows the meshes at various steps with $k=2$ and {$C_2 = 1.0$}. It is easy to observe that significantly more refinements are placed near the boundary. Further experiments also show that the refinements on the boundary become more dominant if we decrease the value of $C_2$ or increase the order of $k$. 

\newcommand{\cpoinc}{C_{\text{P},\patchT}}
\newcommand{\Ckbest}{C_{k,\text{best}}}

\begin{remark}\label{rem:4.1}
	We note that, while the precise value of the best constants $\cunoT$ and $\cdueT$ in Lemma \ref{lem:clem} is not known, it is possible to give an estimate of the ratio $\cdueT/\cunoT$, {in terms of the Poincar\'e constant for the patch $\patchT$}. Indeed, as $\ClemProj$ preserves polynomials of degree not greater than $ k$, we can  write
		\[
		\| u - \ClemProj u \|_{0,T} + h_T \| u - \ClemProj u \|_{1,T} \leq \cunoT h_T \inf_{p \in \mathbb{P}_k} | u - p |_{1,\patchT},
		\]	
and then we can choose $C_2$ such that $C_2/C_1$ is the smallest constant $\Ckbest$ for which
\[
 \inf_{p \in \mathbb{P}_k} | u - p |_{1,\patchT} \leq \Ckbest h_T^{k} | u |_{k+1,\patchT}.
\]
Such a constant may be estimated by recursively applying some upper bound for the Poincar\'e constant for the patch $\patchT$, which can be obtained, for instance, by the approach of \cite{VF11}. Its dependence on the polynomial degree $k$ can also be taken into account.
For this choice to be the most effective, we {would  however need} the upper bounds for $\Ckbest$ to be sharp. If this is not the case, we observe that the true error might present some more or less pronounced oscillations.
In our numerical tests, we tried several different values of $C_2$.
In all the cases considered, setting $C_2$ between $0.1$ to $1$ turns out to be  a reasonable choice. See also \cref{rem:4.1}.					
\end{remark}

%figures

%To add later
\begin{figure}[ht]
\centering
\begin{tabular}{cc}
\includegraphics[width=.30\textwidth]{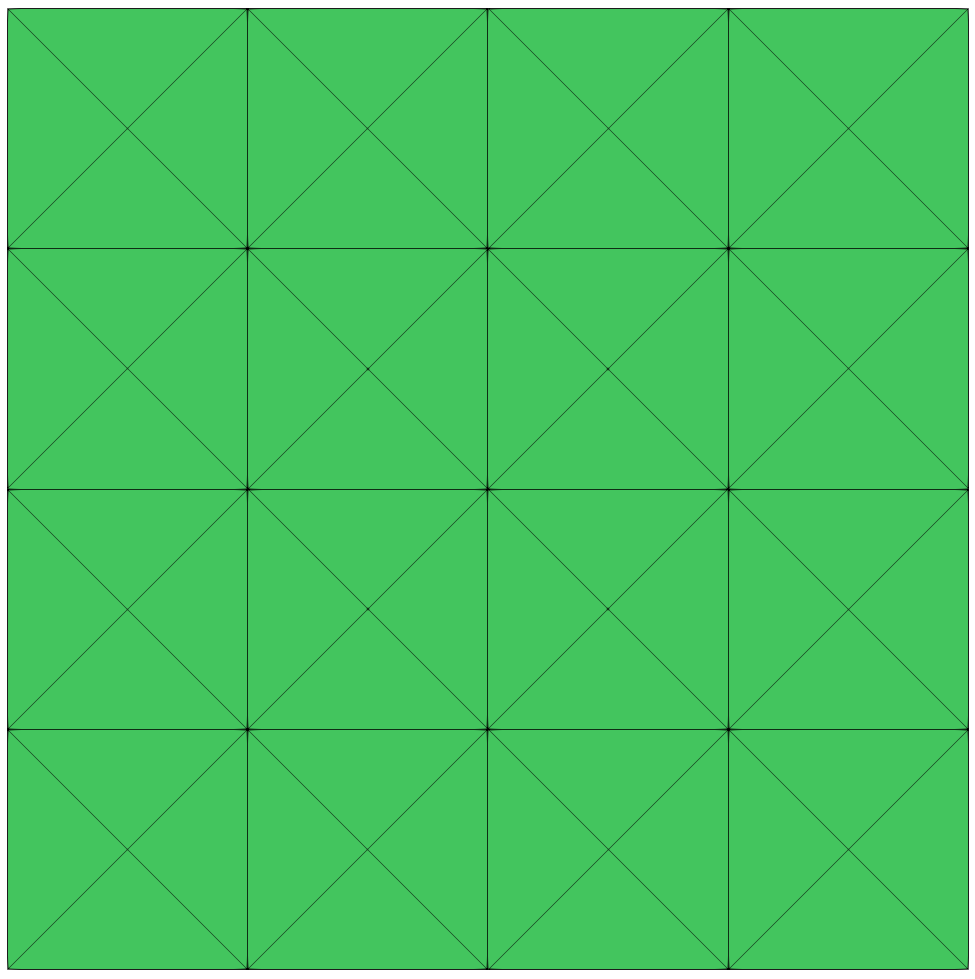}
&\includegraphics[width=.30\textwidth]{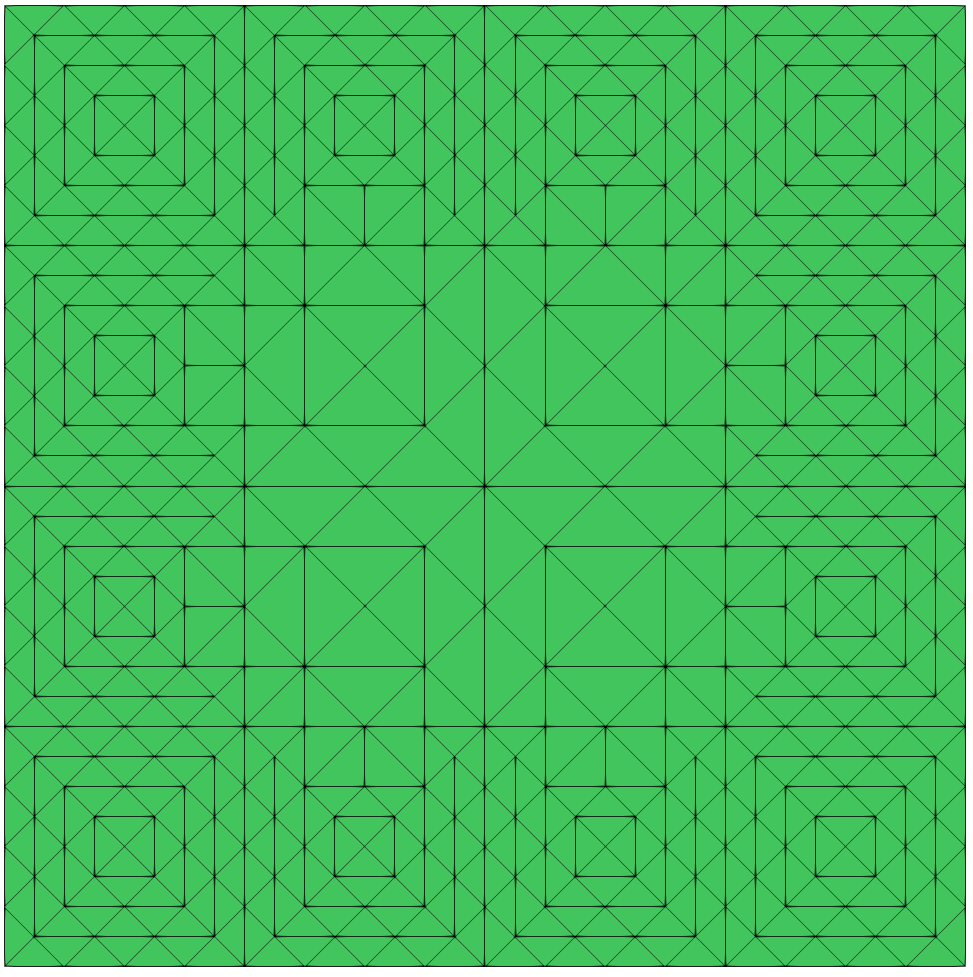}\\
(a) step 1&(b) step 3\\
\includegraphics[width=.30\textwidth]{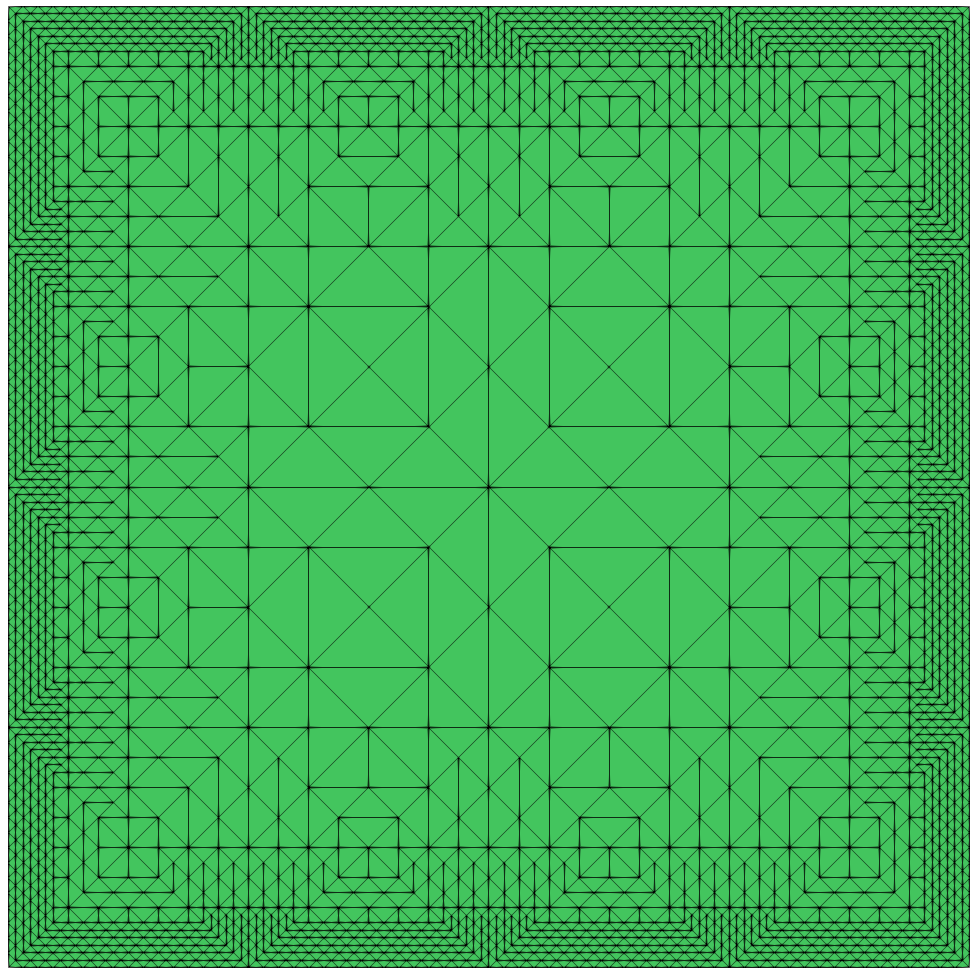}
&\includegraphics[width=.30\textwidth]{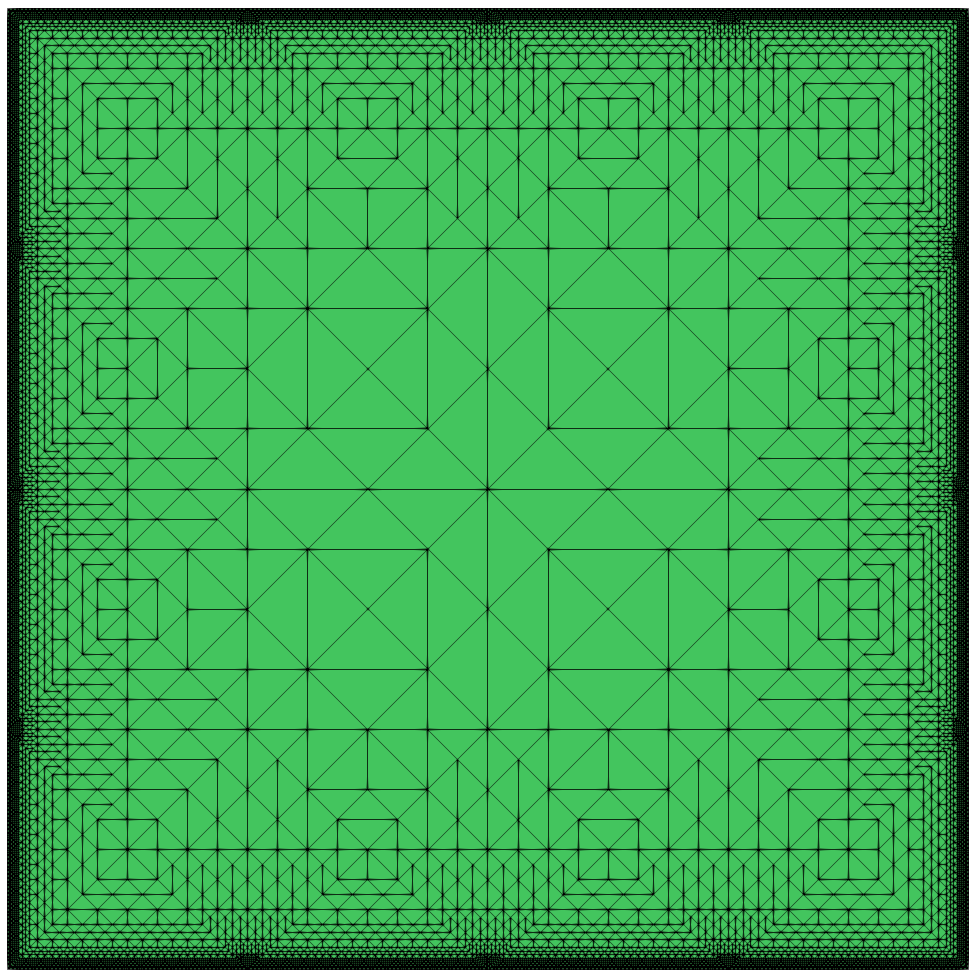}\\
(c) step 5&(d) step 7
\end{tabular}
\caption{Adaptive meshes based on $\varsigma_T$ with $k=2$ and $C_2=1.0$}
\label{fig:adaptive-w}
\end{figure}

\subsection{Computation of the true error}
In this subsection, we present two methods to compute the true error, i.e., $\|\lambda - \lambda_h\|_{-1/2, \Gamma}$, for the purpose of comparison.
From \cref{negative-half-norm-representation} and \cref{Riesz_lam}, 
\[
	\|\lambda - \lambda_h\|_{-1/2, \Gamma}^2 = 
	|\nabla w|_\Omega^2  \quad( \mbox{or } \left<\lambda - \lambda_h, w \right>_{\Gamma}),
\]
where
$w \in H^1(\Omega)$ satisfies the following variational problem:
\begin{equation}\label{dual-pro}
	\int_{\Omega}\nabla w \cdot \nabla v = 
	\int_\Gamma (\lambda - \lambda_h) v, 
	\quad  \mbox{and} \quad \int_{\Gamma} w =0 \quad \forall \, v\in H^1(\O).
\end{equation}

Note that (\ref{dual-pro}) is a pure Neumann problem. The compatibility of the 
solution is guaranteed since $\int_{\Gamma} (\lambda - \lambda_h) =0$
for all aforementioned numerical methods.  
We approximate the true error in each refinement step using a two order higher finite element method on a finer mesh (compared to the mesh used in the adaptive procedure). 
We let $w_h \in  V_h^{k+2}$ denote the Galerkin projection of $w$ on $V_h^{k+2} = \{ v \in H^1(\O): v|_T \in P^{k+2}(T) \quad \forall \,T \in \tilde{\mathcal{T}}_h \}$. Here $\tilde{\mathcal{T}}_h$  is the finer mesh.
We then approximate the error by
\begin{equation}\label{true-error}
	\|\lambda - \lambda_h\|_{-1/2, \Gamma}^2 \approx
	|\nabla w_h|_\Omega^2 \quad (\text{or }  \int_{\Gamma}(\lambda - \lambda_h) w_h).
\end{equation}

When $\lambda$ does not have enough regularity, using (\ref{true-error}) to accurately compute the true error becomes infeasible as a very fine mesh is  required to guarantee the accuracy. 
We therefore introduce another method to compute the true error by exploring  properties of  the wavelet decomposition. Indeed, it is known that, by expanding a function in $H^{-1/2}(\Gamma)$ based on a suitable wavelet basis, an equivalent $H^{-1/2}(\Gamma)$ norm can be computed by taking a  weighted $L^2$ norm of the coefficient vector. The latter can be efficiently computed by applying a wavelet transform \cite{CDF}.
This  only requires computations on $\partial \Omega$, therefore we are able to compute the true error to a satisfactory accuracy even for low regularity $\lambda$.

More precisely, given $v \in H^{-1/2} (\Gamma)$, we aim at computing $\|v\|_{-1/2, \Gamma}$. In order to do so, we consider the sequence of spaces $\{V_j\}_{j=0}^\infty$ such that $V_j\subset L^2(\Gamma)$ is the space of piecewise constant functions on the embedded uniform grid on $\Gamma$  with mesh size $|\Gamma| 2^{-j}$. We denote by {$\{x^j_k\}_{k=0}^{2^j-1}$}, the nodes of the corresponding mesh, which we assume to be ordered counter-clock wise.
For $v \in V_j$, we can compute the vector $\mathbf{v}_j$ of length $2^j$
\[
	\mathbf{v}_j :=\{v_{jk}\}_{k=0}^{2^j-1} \quad \mbox{and} \quad
	v_{jk} = \dfrac{2^{j/2}}{|\Gamma|}\int_{x^j_k}^{x^j_{k+1}} v.
\] 
$\{v_{jk}\}_{k=0}^{2^j-1}$ is regarded as the coefficients of the $L^2(\Gamma)$ orthonormal bases consisting of the normalized characteristic functions on the elements of the grid. As $V_j \subset V_{j+1}$, for all level $j$ we can decompose $v_{j+1} \in V_{j+1}$ as $v_{j+1} = v_j + d_j$, with $v_j \in V_j$ obtained by applying a suitable oblique projector $P_j$ to $v_{j+1}$. This gives us a telescopic expansion of all function in $V_M$ as $v_M = v_0 + \sum_{j=0}^{M-1} d_j$, and, passing to the limit as $M$ goes to infinity, of all functions in $L^2(\Gamma)$ as $v = v_0 + \sum_{j=0}^{\infty} d_j$. Given $\mathbf{v}_{j+1}$, we can compute $\mathbf{v}_j :=\{v_{jk}\}_{k=0}^{2^j-1} $ and 
$\mathbf{d}_j := \{d_{jk}\}_{k=0}^{2^j-1}$ (this last one being the vector of coefficients of $d_j$ with respect to a suitable basis for the space $W_j = (1-P_j)V_{j+1}$),  by applying a {\em low-pass filter} $h$ (strictly related with the projector $P_j$), and the {\em band-pass filter} $g = [1,-1]$: 
\[
	v_{jk} = \sum_{l=0}^L \dfrac{\sqrt{2}}{2} h(l)\, v_{j+1, 2k+l} \quad \mbox{and}\quad
	d_{jk} = \sum_{l=0}^1 \dfrac{\sqrt{2}}{2} g(l)\, v_{j+1, 2k+l}  = \dfrac{\sqrt{2}}{2} \left( v_{j+1,2k} - v_{j+1,2k+1} \right),
\]
where $L+1$ is the length of the low-pass filter $h$.
 In the above computation the function $v$ is considered as periodic, so that, when the index $2k+l>2^{j+1}-1$, we extend the vector $\mathbf{v}_{j+1}$ as $v_{j+1,2^{j+1}+k} =  v_{j+1,k},  k \ge 0$.
For suitable choices of the low pass filter $h$, the following norm equivalence holds for all $v \in H^{-1/2}(\O)$ (\cite{D}) 
\[
	\|v\|_{-1/2, \Gamma}^2 \simeq \|\mathbf{v}_{0}\|_2^2 + \sum_{j=0}^\infty 2^{-j} \|\mathbf{d}_j\|_2^2,
\]
where $\|\cdot\|_2$ denotes the Euclidean  norm. % on $\mathbb{R}^{2^j}$.
In our experiments we choose the so called {\em (2,2)-biorthogonal wavelet} (see \cite{CDF}), for which the low pass filter $h$ is
\[
h = \dfrac{\sqrt{2}}{2} [3/128, -3/128, -11/64, 11/64, 1, 1, 11/64, -11/64, -3/128, 3/128].
\] 
By choosing $M$ big enough and projecting $v$ onto $V_M$ (in our tests we use the $L^2$ orthogonal projection), we approximate the norm by %\noteHe{(changes made below. function symbols to vector coefficient.)}
\begin{equation}\label{wavelet-computation}
	\|v\|_{-1/2, \partial \O}^2 \approx  \|\mathbf{v}_{0}\|_2^2 + \sum_{j=0}^{M-1} 2^{-j} \|\mathbf{d}_j\|_2^2.
\end{equation}

%%%%%%%%%%%%%%%%%%%%%%%%%%%%%%%%%%%%%%%%%%%%%%%%%%%%%%

\subsection{Test results}
\textcolor{blue}{Before presenting the results of our numerical tests, let us recall what the dependence of the error on the number of degrees of freedom is expected to be for an order $k$ method on either a uniform or a boundary concentrated mesh: letting $h$ denote the mesh size on the boundary and $N$ the total number of degrees of freedom, we have $h \simeq N^{-1/2}$ for uniform meshes, and $h \simeq N^{-1} |\log(N)|$ for boundary concentrated meshes. For a smooth solution, the error on the normal flux for optimal order $k$  method  will behave like $h ^k$, that is, $N^{-k/2}$ for uniform grids and  $N^{-k} |\log(N)|^k$ for boundary concentrated meshes. }

\textcolor{blue}{To assess the performance of our estimator, we test it on the
	 Lagrangian method without stabilization and on Nitsche’s method (the Barbosa-Hughes method being equivalent to the latter). 
	  Nitsche's method with polynomial degree $k$ is optimal, i.e., it yields an order $k$ rate of convergence, on uniform meshes (see \cref{tab:ex1-Nitsche-k=1}).   
	 For the Lagrangian method, the rate of convergence depends on the choice of the multiplier. We test two choices: discontinuous piecewise polynomials of  order $k'= k-2$ and continuous polynomials of order $k' = k$. Both choices yield inf-sup stable discretizations, yet they are both suboptimal (see \cref{tab:ex1-LM}): the first choice only provides, for the normal flux, an approximation of order at most $k-1/2$, at the cost of using an order $k$ method in the bulk, while, in the presence of corners, the second only allows for an order 1 approximation of the normal flux, independently of $k$, as it involves approximating a discontinuous function (the normal flux, in the presence of corners) by means of continuous functions. 
 	  We point out that we are in no way advocating such choices as recommended methods for solving the problem considered (other choices for the multiplier, see \cref{rem1.1}, allowing for optimality, are of course to be preferred for the actual computation of the flux). However, considering such suboptimal cases allows us to put the robustness of our method to the test, and to show that the 	  refinement driven by our estimator can somehow make up for the lack of optimality.
	 }

\begin{example}\label{ex1}
In this example, we consider the Poisson equation on the unit square domain with right hand side and boundary data chosen so that the solution is the Franke function \cite{franke1979}
\begin{equation*}
\begin{split}
	u(x,y)  =&0.75 \exp{\left(-(9x-2)^2/4 - ( 9y-2)^2/4\right)} 
	+ 0.75 \exp{(-(9x+1)^2/49 - (9y+1)/10)}\\
	&+0.5\exp{(-(9x-7)^2/4 - (9y-3)^2/4)}
	-0.2 \exp{(-(9x-4)^2 - (9y-7)^2)}.
\end{split}
\end{equation*}
This function has two peaks at $(2/9, 2/9)$ and $(7/9,1/3)$ and one sink at $(4/9, 7/9)$. %(see  \cref{Fig:Ex1-Franke}).
\end{example}

We firstly test the convergence rate of the true error $\| \lambda - \lambda_h\|_{-1/2,\Gamma}$ on uniform meshes. 
The true error is computed using the aforementioned two methods. 
We denote by $E_1$ the error computed by (\ref{true-error}) and by $E_2$ the error computed using the wavelet in \cref{wavelet-computation} with $M = 20$. The problem (\ref{true-error}) is solved on a finer uniform mesh with mesh size $h=1/64$. 
Tables \ref{tab:ex1-Nitsche-k=1} and \ref{tab:ex1-LM} show the convergence rates for $E_1$.
Observe that these are in agreement with the expected convergence rates given by the standard error estimates for the two methods, that is order $1$ (resp $2$) for Nitsche's method with $k = 1$ (resp. $k=2$), and 
\textcolor{blue}{order $3/2$  (resp.  $1$) for the Lagrangian multiplier method with $k=2$, $k'= 0$, (resp. $k=2$, $k'= 2$)}.
\textcolor{blue}{
From Tables \ref{tab:ex1-Nitsche-k=1} and \cref{tab:ex1-LM}, we also observe that the ratio between $E_2$ and $E_1$ is relatively stable (the fluctuation of the ratio is likely caused by the inaccurate computation of $E_1$). In particular, for Nitsche's method, the ration $E_2/E_1$ remains close to $0.25$ for both orders.
These results, therefore, confirm that $E_2$ is equivalent to the true error for both the Nitsche and Lagrangian multiplier methods. }

\begin{table}[ht]
\caption{ \footnotesize{\cref{ex1}: Convergence rates for Nitsche's method on uniform meshes}}
\label{tab:ex1-Nitsche-k=1}
\begin{center}
{
	\begin{tabular}{||c|c c|c c|c c|c c||}
	\hline
	&\multicolumn{4}{|c|}{$k=1$}&\multicolumn{4}{|c|}{$k=2$}\\
	\hline
	h & $E_1$ & rate & $E_2$ &  $E_2/E_1$& $E_1$ & rate & $E_2$ & $E_2/E_1$\\
	\hline
	1/8   &    3.35E-1  & 0.40    &    8.53E-2     & 0.25& 2.86E-1&3.31&5.17E-2&0.18\\
	1/16 &  1.73E-1    & 0.95    &   4.51E-2      & 0.26& 3.19E-2&3.16&7.50E-3&0.23\\
	1/32 &  8.66E-2    & 1.00    &   2.26E-2      & 0.25 &4.69E-3&2.77&1.44E-3&0.30\\
	1/64 &  4.33E-2    & 1.00    &   1.13E-2      & 0.26& 2.51E-4&2.10&8.15E-5&0.32\\
	\hline
	\end{tabular}
}
\end{center}
\end{table}

\begin{table}[ht]
\caption{ \footnotesize{\cref{ex1}: Convergence rates for Lagrangian Multiplier method on uniform meshes}}
\label{tab:ex1-LM}
\begin{center}
{
	\begin{tabular}{||c|c c|c c|c c|c c||}
	\hline
	&\multicolumn{4}{|c|}{$k=2,k'=0$}&\multicolumn{4}{|c|}{$k=2,k'=2$}\\
	\hline
	h & $E_1$ & rate & $E_2$ &  $E_2/E_1$& $E_1$ & rate & $E_2$ & $E_2/E_1$\\
	\hline
	1/8   & 4.58E-2 &1.88  &1.14E-2 &0.25  & 9.83E-3 &2.67&3.13E-2&3.18\\
	1/16 &1.43E-2  &1.68  &3.97E-3 &0.28  & 2.65E-3 &1.89&7.74E-3&2.92\\
	1/32 & 4.80E-3 &1.57  &1.40E-3 &0.29  &1.18E-3  &1.15&2.60E-3&2.19\\
	1/64 & 5.62E-4 &1.54  &1.82E-4 &0.32  & 5.88E-4 &1.01&1.15E-3&1.96\\
	\hline
	\end{tabular}
}
\end{center}
\end{table}

 We now test the adaptive mesh refinement (AMR) procedure for the Lagrangian method. In the adaptive procedure, we set the stopping criteria such that the total number of degree of freedoms (DOFs) less than $20,000$. The marking strategy is set such that an element $T$ is marked to be refined if $
 	\eta_K \ge 0.5 \eta_{K,max}$.
 \textcolor{blue}{In this example, we set $C_2 = 1.0$}. The initial mesh is set to be the $4 \times 4$  mesh in \cref{fig:adaptive-w}(a). 
 For comparison, we also perform the adaptive mesh refinemet procedure using the classical residual based error estimator (AMRc) without any dual weights .
 For the Lagrangian method, it is defined as

\[
	\eta_{classical} =
 \sqrt{  \sum_{T\in \Th}
  |\resuno(T) |^2 + \sum_{F \in \iEdges}   | \resdue(F) |^2  + 
\sum_{F \in \bEdges} \left(
| \restre(F) |^2 + h_F^{-1}\|g -u_h\|_{0,F}^2 \right) }
\]
and for the Nitsche's method \cite{BHL20} is defined as

\[
	\eta_{classical} = 
	 \sqrt{  \sum_{T\in \Th}
 \textcolor{blue}{|} \resuno(T) |^2 + \sum_{F \in \iEdges}   | \resdue(F) |^2  + 
\sum_{F \in \bEdges} 
\gamma^2  h_F^{-1}\|g -u_h\|_{0,F}^2  }.
\]
It is well known that $\eta_{classical}$ is optimal in minimizing the energy norm of the error, i.e., $\| \nabla (u - u_h)\|_{0,\O}$. 
Note that comparing with \cref{eta}, the $H^1$ norm in $\resquattro(F)$ is reduced to $L^2$ norm on $\Gamma$ with adjusted weights for the Lagrangian method.

\cref{Ex1-LM-mesh} shows the final meshes for the adaptive Lagrangian Multiplier method ($k=2, k'=0$) using respective $\eta_{classical}$(left) and $\eta$(right). It can be seen that the mesh generated by $\eta_{classical}$ has dense refinements around the interior peaks and sinks while the mesh generated by $\eta$ has more dense refinements near the boundary and almost completely ignore the peaks and sinks in the interior domain.

\begin{figure}[ht]
\centering
\begin{tabular}{cc}
\includegraphics[width=.30\textwidth]{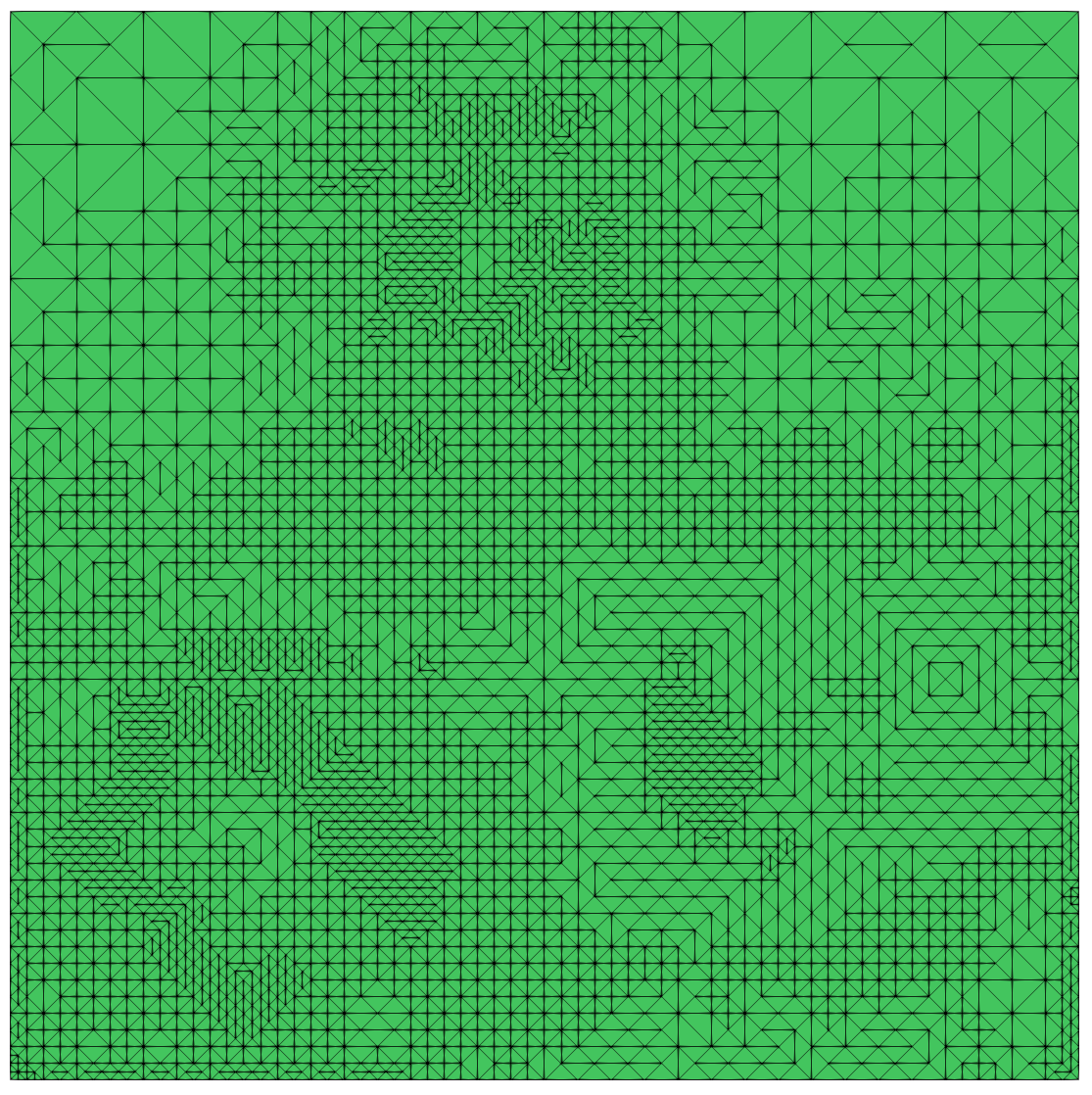} 
&
\includegraphics[width=.30\textwidth]{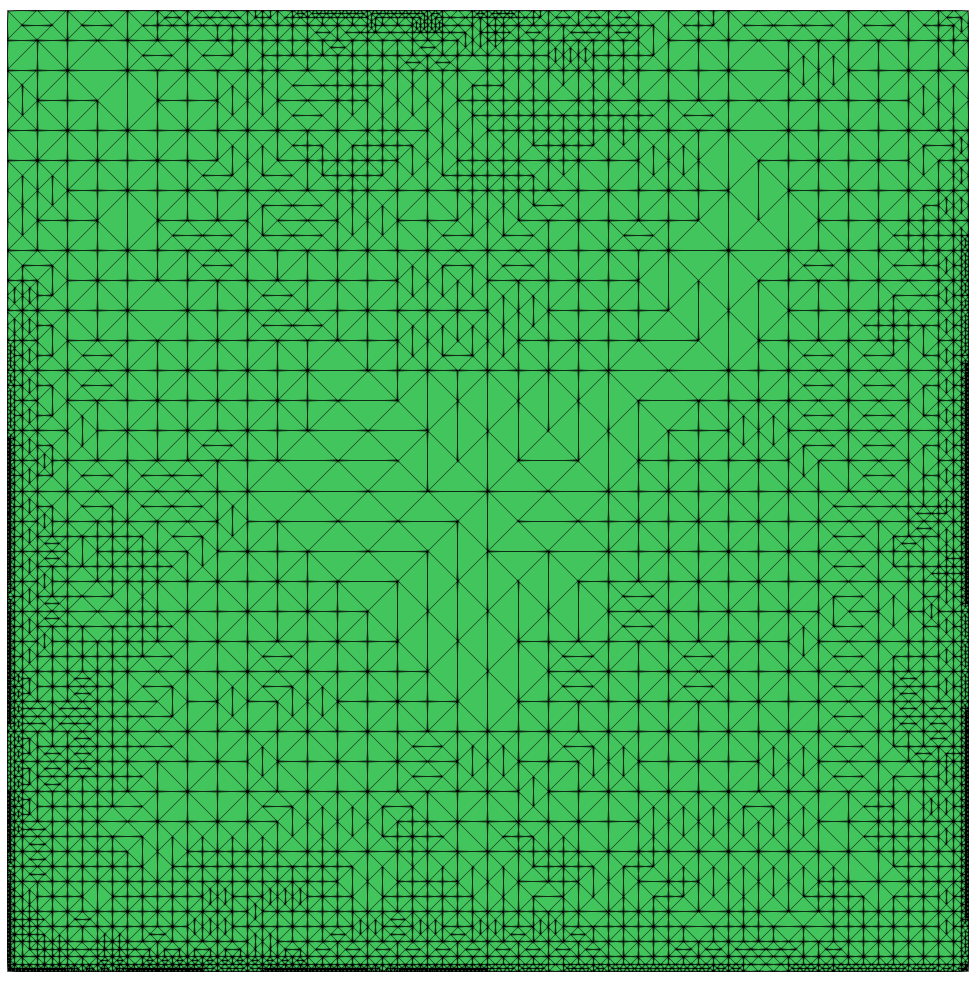}\\
{(a) by $\eta_{classical}$} &{(b) by $\eta$}
\end{tabular}
\caption{\cref{ex1}. Final meshes for the Lagrangian method ($k=2, k'=0,C_2=1.0$)} 
 \label{Ex1-LM-mesh}
\end{figure}

In the log-log plots \cref{Ex1-LM-error-a}--\cref{Ex1-LM-error-c}, we compare the convergence of  true errors and estimators. 
The purpose of the convergence figures is to compare  the two adaptive procedures using respectively the  dual-weighted and the classical non-weighted error estimators. From \cref{Ex1-LM-error-a}, we see that the error driven by $\eta$ converges faster than the one driven by 
$\eta_{classical}$, which already has the order $N^{-1}$ with $N$ being the total number of DOFs.
In comparison with rates attained by uniform refinement, that are provided in the \cref{tab:ex1-LM} and \cref{tab:ex1-Nitsche-k=1}, the relationship is that the rate obtained by $\eta_{classical}$ is higher or equal than the uniform approximation rate, and that the rate obtained by $\eta$ is higher than that obtained by $\eta_{classical}$.
\textcolor{blue}{
More in detail,	in \cref{Ex1-LM-error-a} we display two reference straight lines: the slope $-1$ of the first line
  refers to the approximation rate in the energy norm that can be attained by the best approximation with order $k$ finite elements on a quasi uniform grid with $N$ degrees of freedom, which also provides an upper bound for the corresponding error of the normal flux. The slope of the second reference line is numerically evaluated by linear regression
  of the data set $(\log(N), \log(E))$
from the AMR with the proposed estimator $\eta$. For this case its value is $\sim-1.5$.   For the figures thereafter, the same strategies will be used to present the reference slopes. }

 \cref{Ex1-LM-error-b} shows that both methods display the same rate of convergence with respect to the number of DOFs on the boundary. However, for the same total number of DOFs, much more DOFs are located on the boundary by $\eta$. More precisely, \cref{Ex1-LM-error-c} shows that the ratio between the numbers of boundary DOFs and the total DOFs gradually gets higher for the meshes generated by $\eta$ in the AMR procedure.
 
\begin{figure}[ht]
    \centering
        \subfloat[]{\includegraphics[width=.30\textwidth]{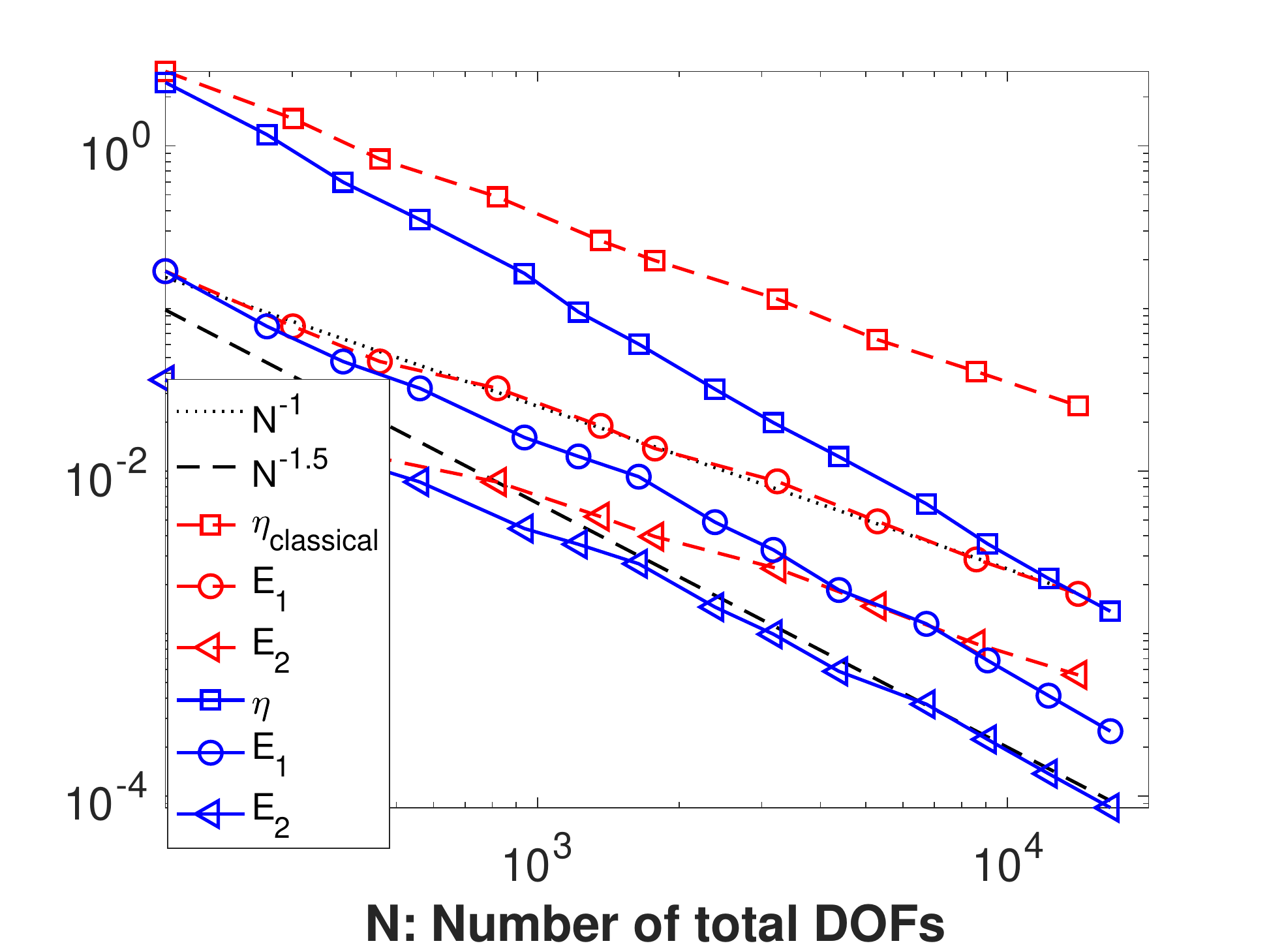}
        \label{Ex1-LM-error-a}}                                                        
    \hfill
        \subfloat[]{\includegraphics[width=.30\textwidth]{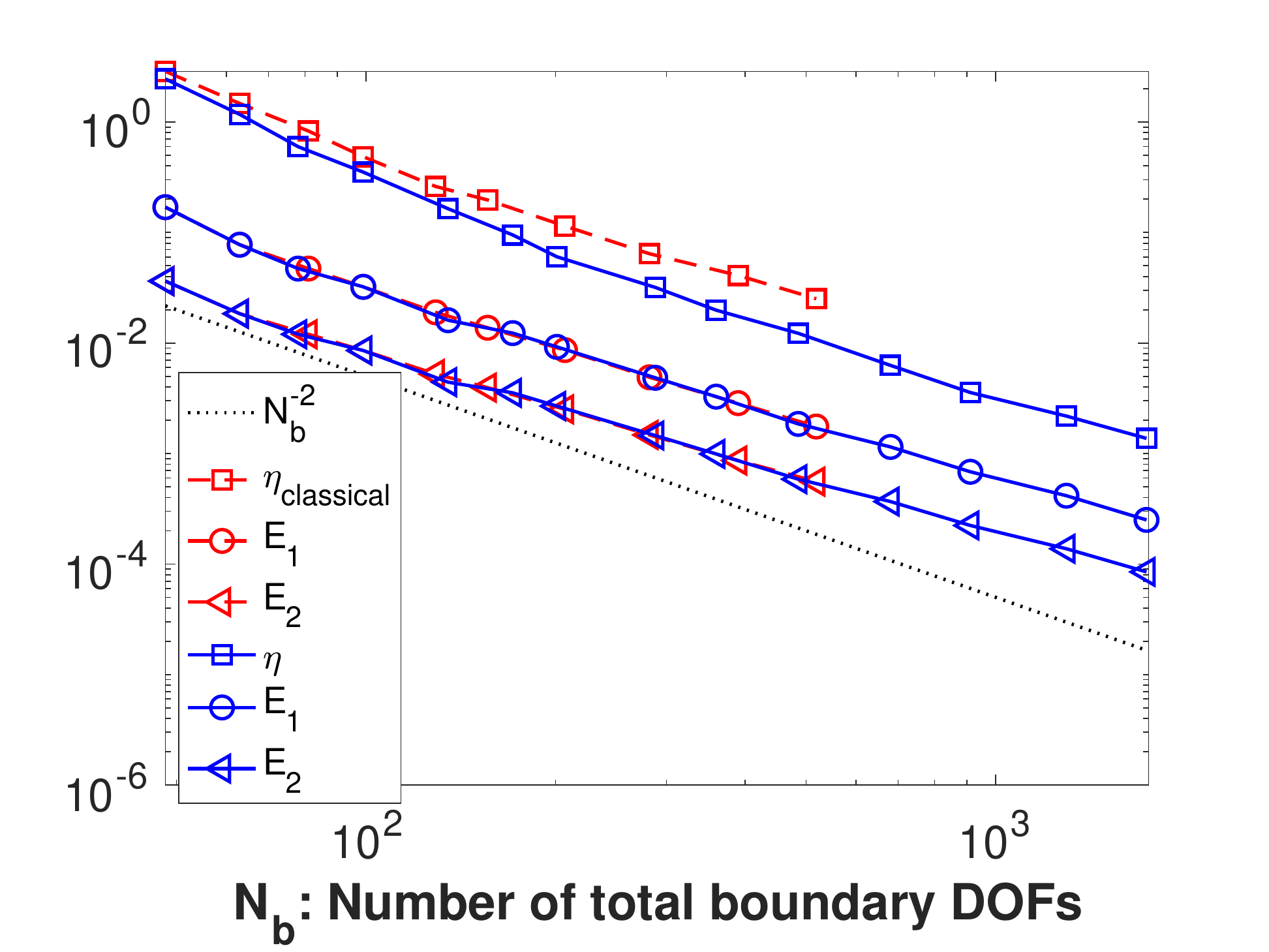}
        \label{Ex1-LM-error-b}}
         \hfill
        \subfloat[]{\includegraphics[width=.30\textwidth]{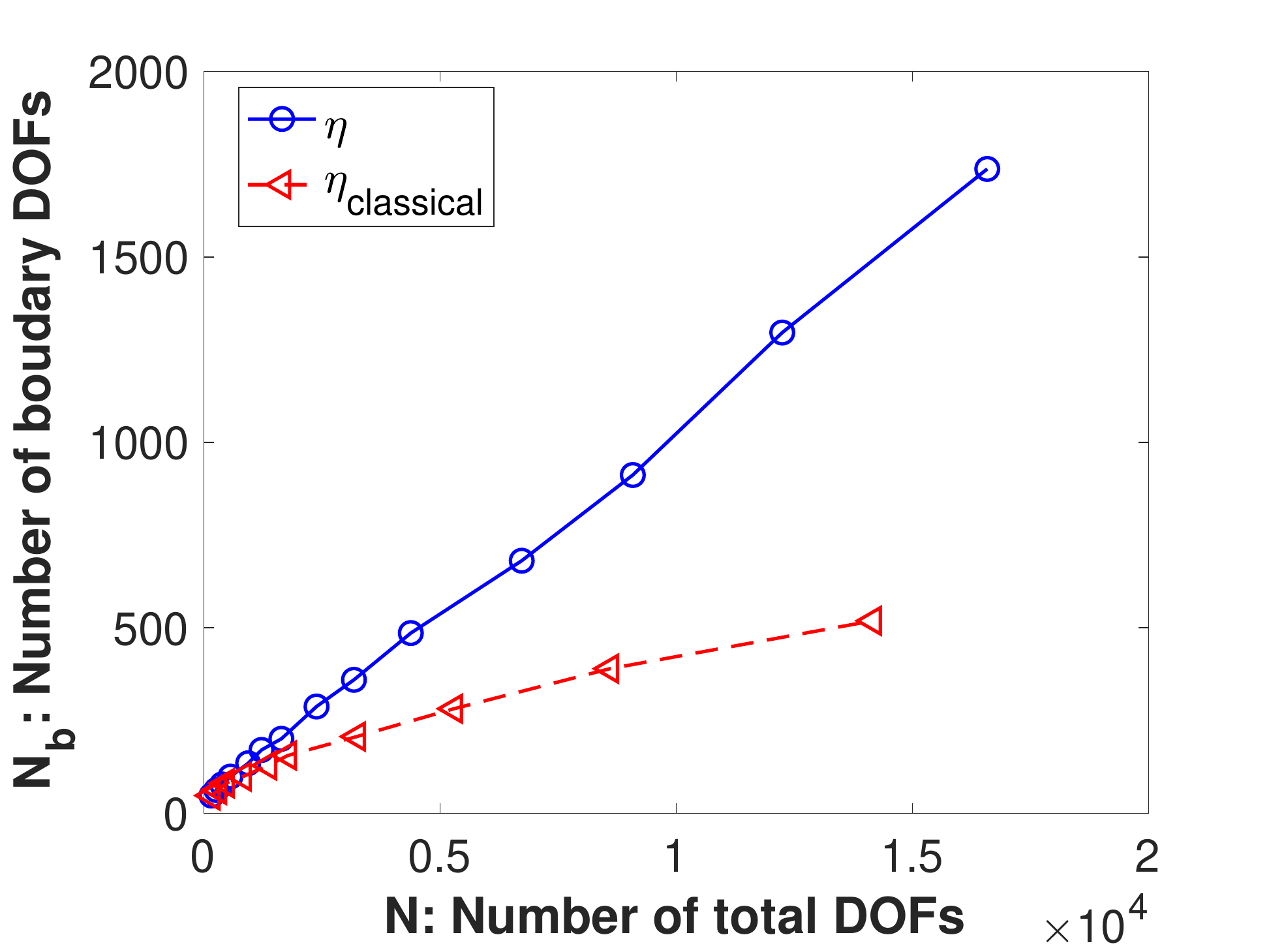}
        \label{Ex1-LM-error-c}}
\caption{ \cref{ex1}. Convergence comparison for Lagrangian method ($k=2, k'=0,C_2=1.0$)}
\end{figure} 

\begin{figure}[ht]
\centering
\begin{tabular}{cc}
\includegraphics[width=.30\textwidth]{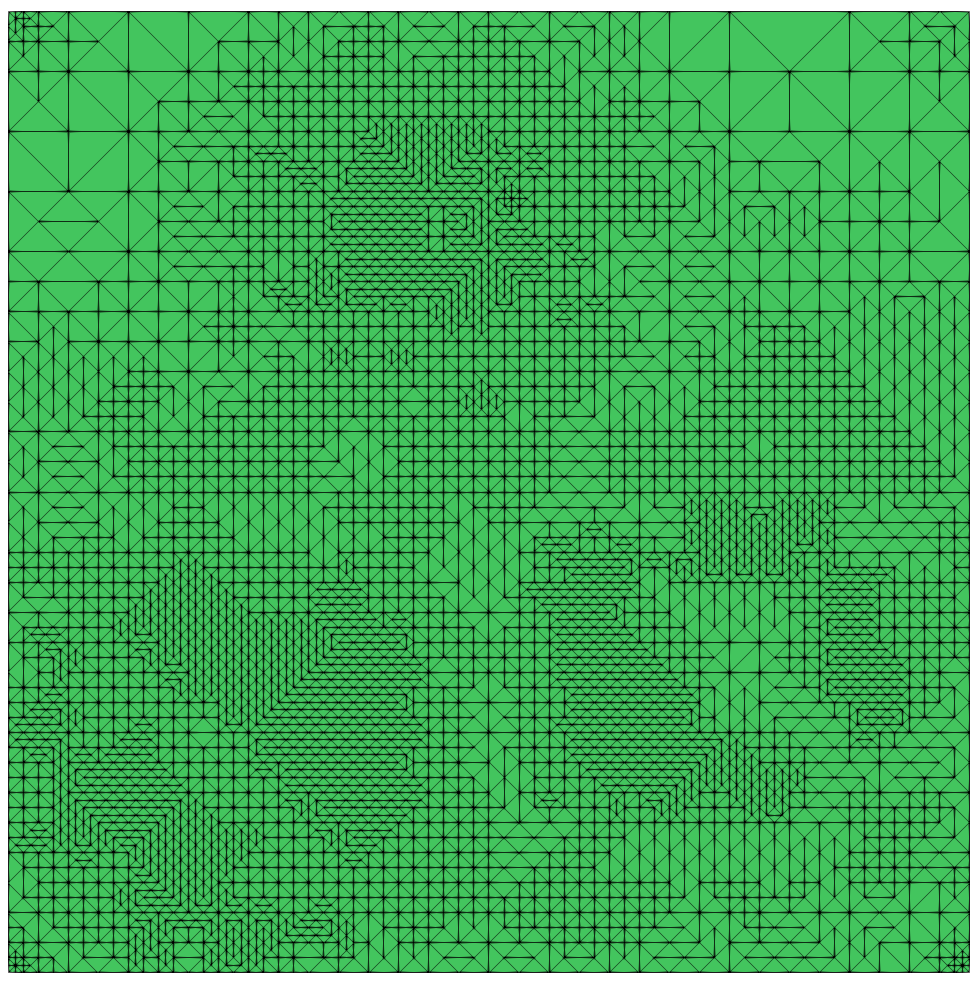} 
&
\includegraphics[width=.30\textwidth]{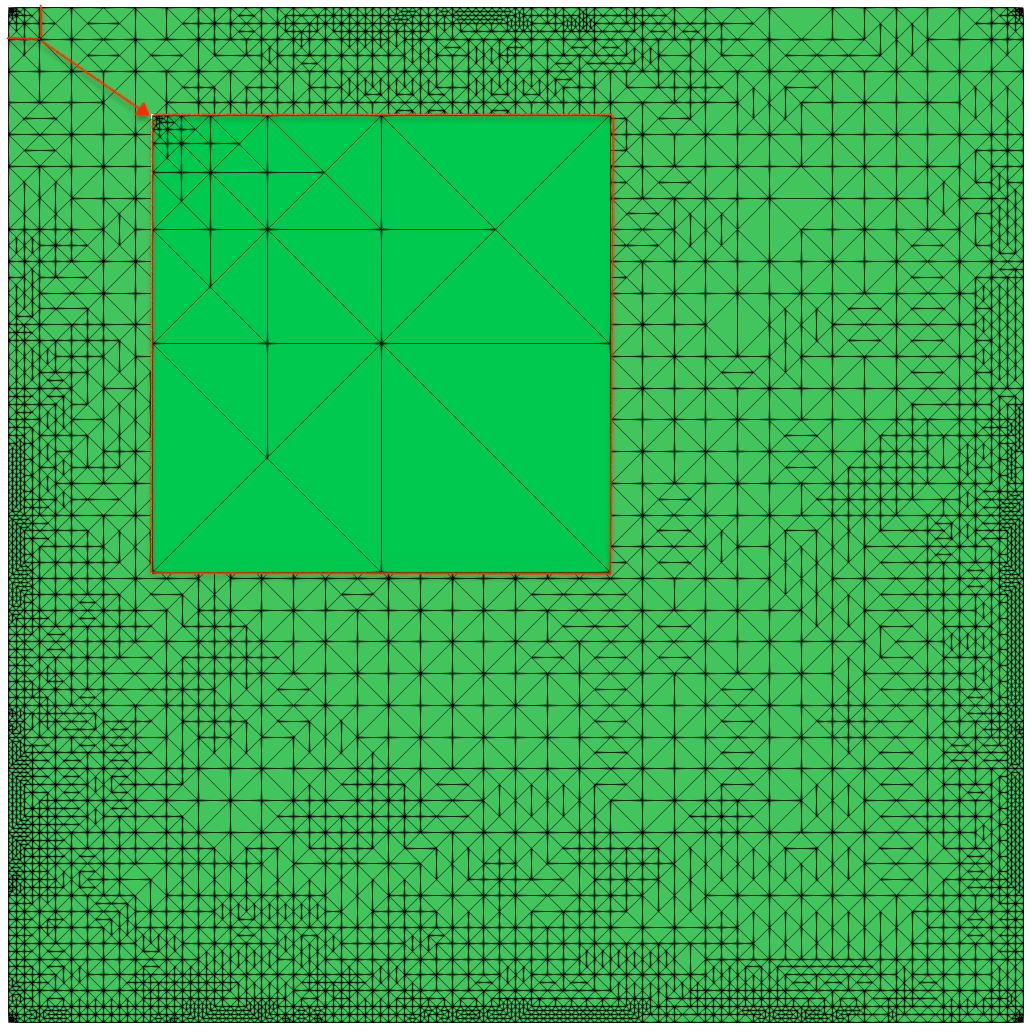}\\
{(a) by $\eta_{classical}$} &{(b) by $\eta$}
\end{tabular}
\caption{\cref{ex1}. Final meshes for the Lagrangian method ($k=2, k'=2,C_2=1$), including, on the right, a zoom on the upper left corner.} 
 \label{Ex1-LM-mesh-k'=2}
\end{figure}

\textcolor{blue}{We also test \cref{ex1} using the Lagrangian Multiplier method with $k=2$, $k'=2$ and
\begin{equation}
\Lambda_h = \{
\lambda \in C^0(\Gamma): \ u|_{\face} \in \mathbb{P}_{2}(\face), \ \forall \face \in \Th|_\Gamma \}.
\end{equation}
Since  in this test the domain has corners, and, consequently, $\lambda$ is discontinuous, optimal approximation for the multiplier can not be achieved, as $\lambda_h \in C^0(\Gamma)$. This also shows in \cref{tab:ex1-LM} for the uniform refinement. In \cref{Ex1-LM-mesh-k'=2}b, we observe that the mesh is densely refined around the corners which indicates that the error estimator $\eta$ successfully captures the error on the corners.
}

\begin{figure}[ht]
    \centering
        \subfloat[]{\includegraphics[width=.30\textwidth]{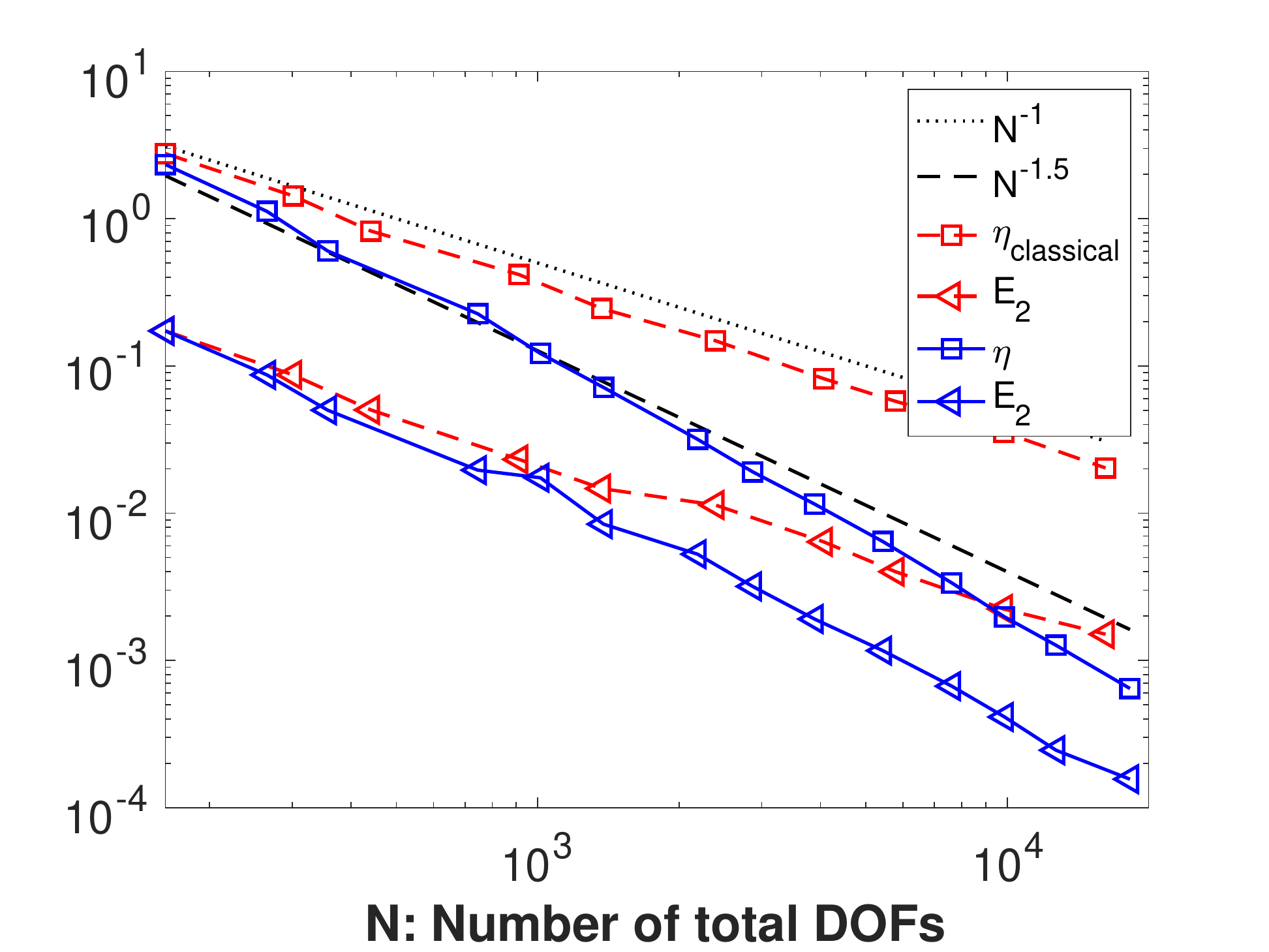}
        \label{Ex1-LM-error-k'=2-a}}                                                        
    \hfill
        \subfloat[]{\includegraphics[width=.30\textwidth]{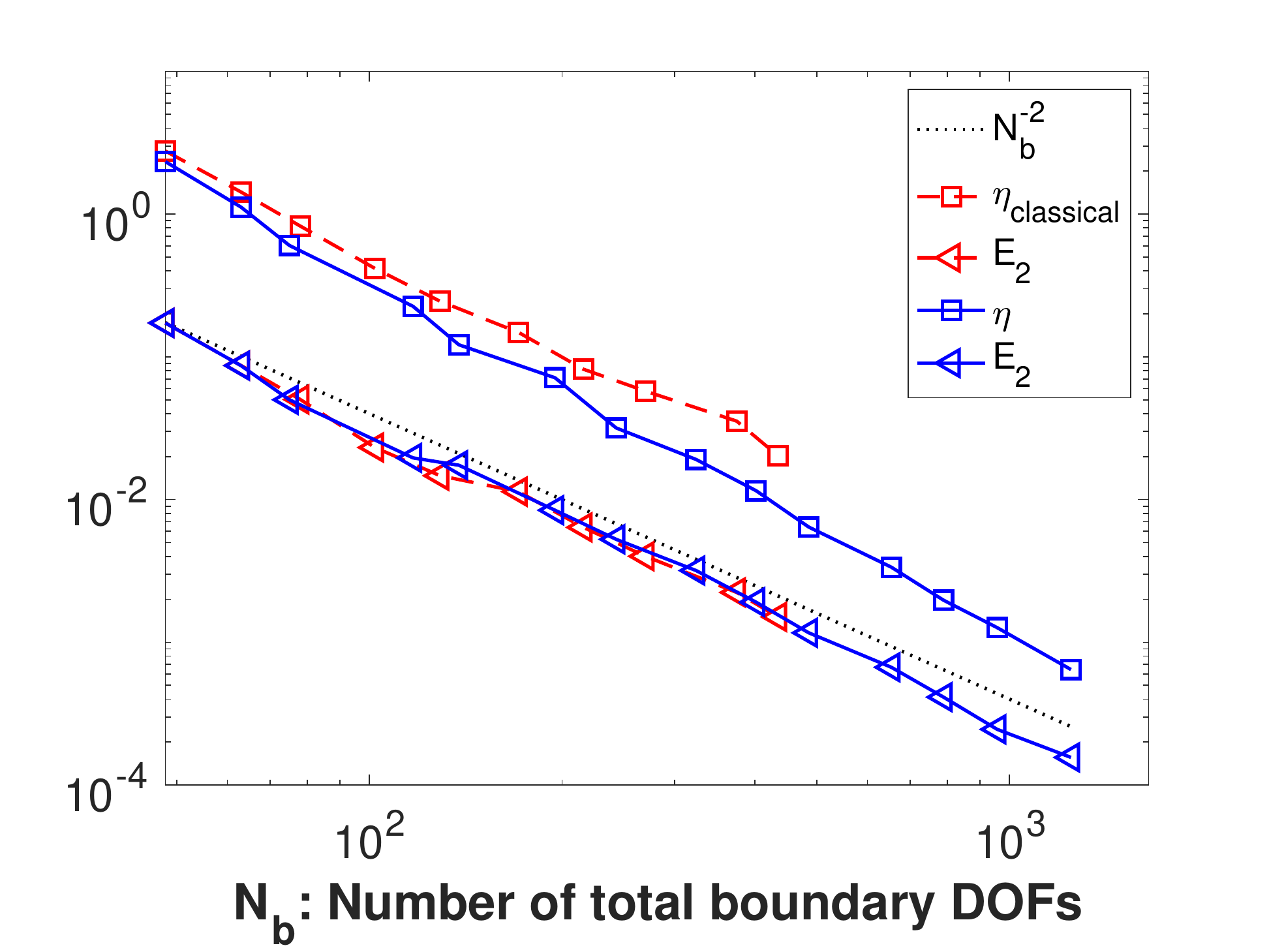}
        \label{Ex1-LM-error-k'=2-b}}
         \hfill
        \subfloat[]{\includegraphics[width=.30\textwidth]{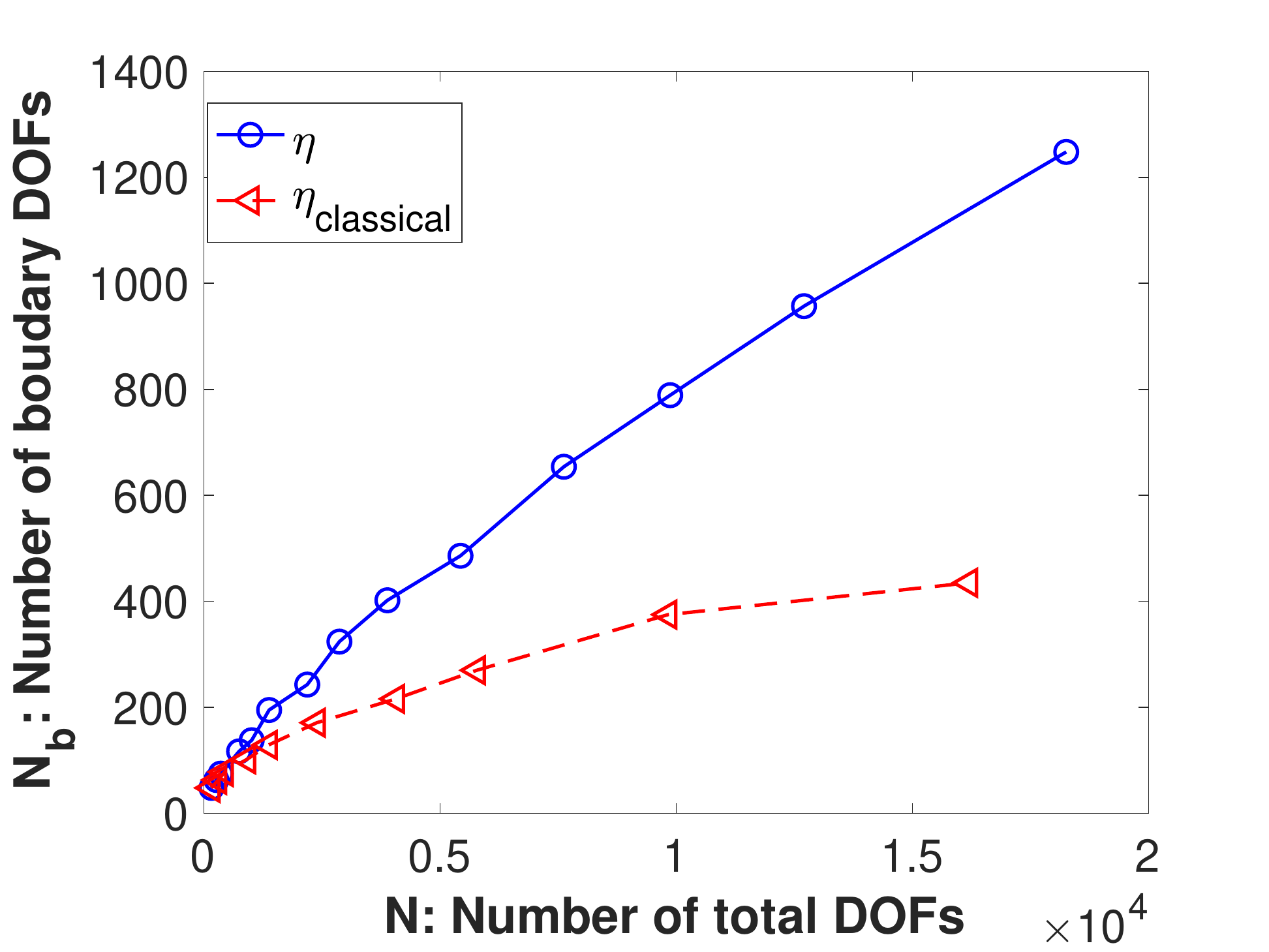}
        \label{Ex1-LM-error-k'=2-c}}
\caption{ \cref{ex1}. Convergence comparison for Lagrangian method ($k=2, k'=2, C_2=1.0$)}
\end{figure}

%To add later 09/17
 \begin{figure}[ht]
\centering
\begin{tabular}{cc}
\includegraphics[width=.30\textwidth]{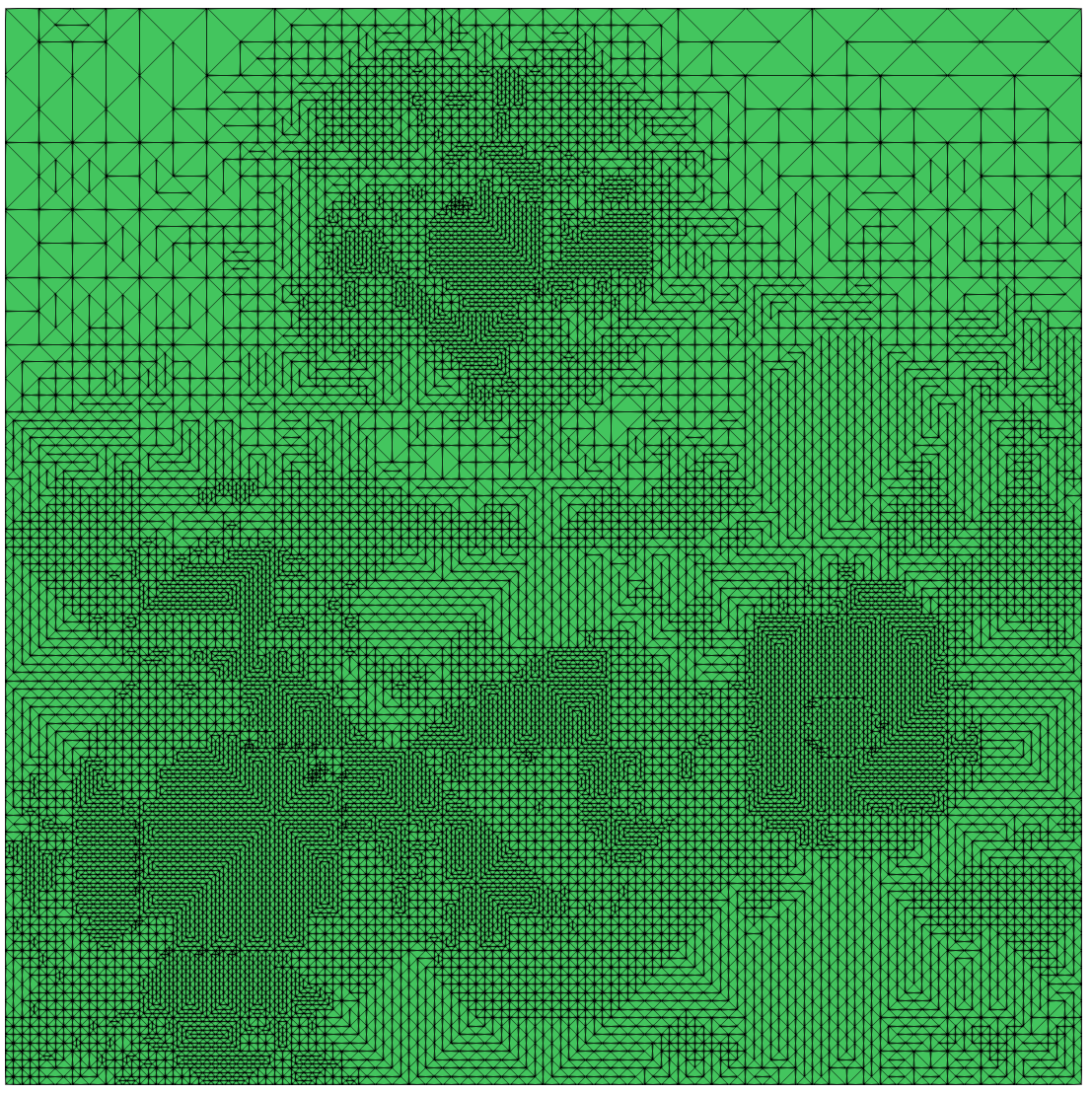} 
  &
 \includegraphics[width=.30\textwidth]{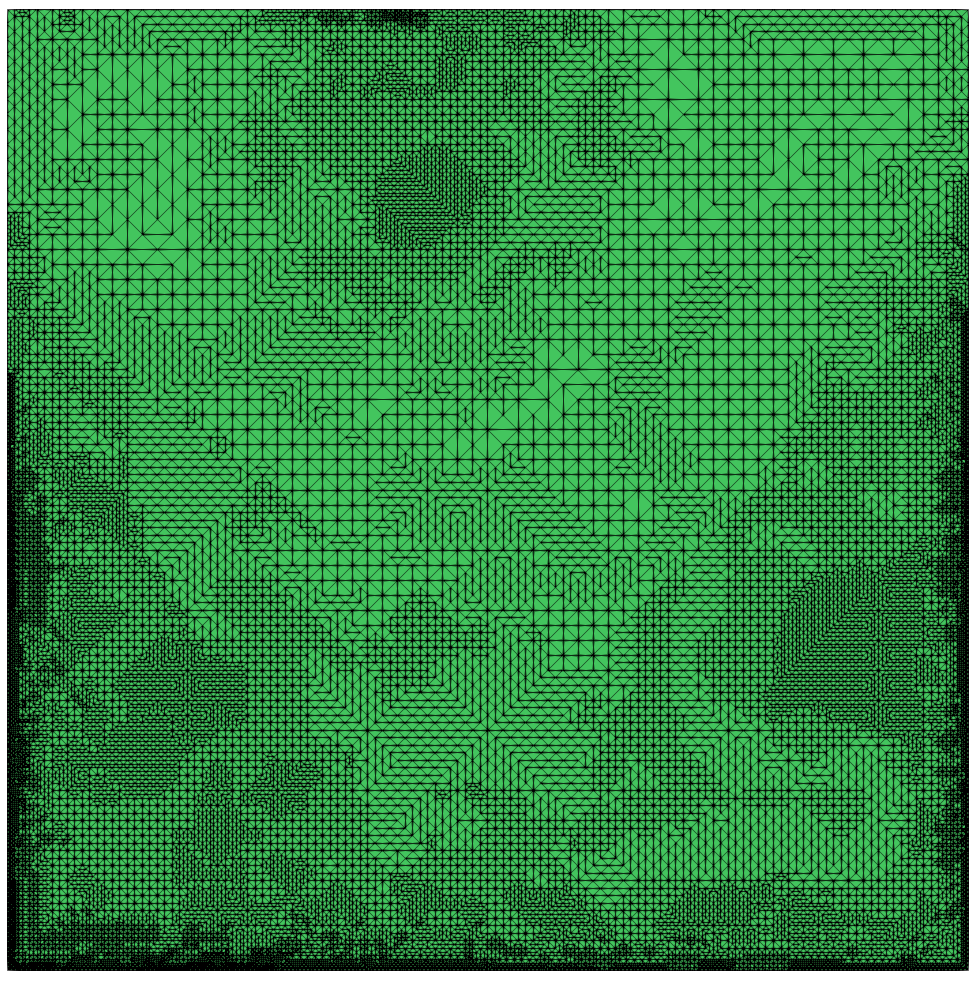}\\
(a) $k=1$ by $\eta_{classical}$ &(b) $k=1, C_2=1.0$ by $\eta$\\ 
\includegraphics[width=.30\textwidth]{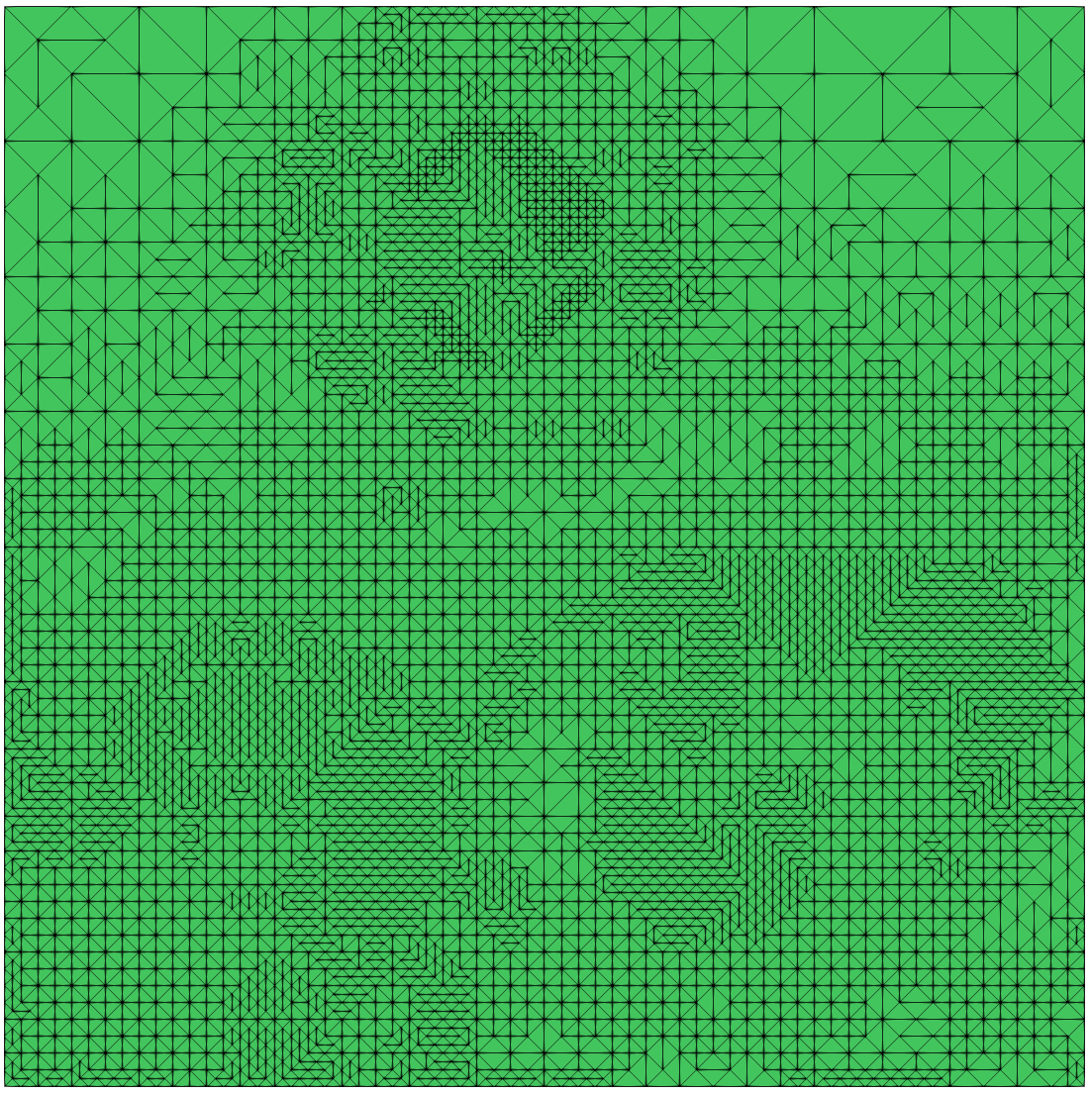} 
  &
\includegraphics[width=.30\textwidth]{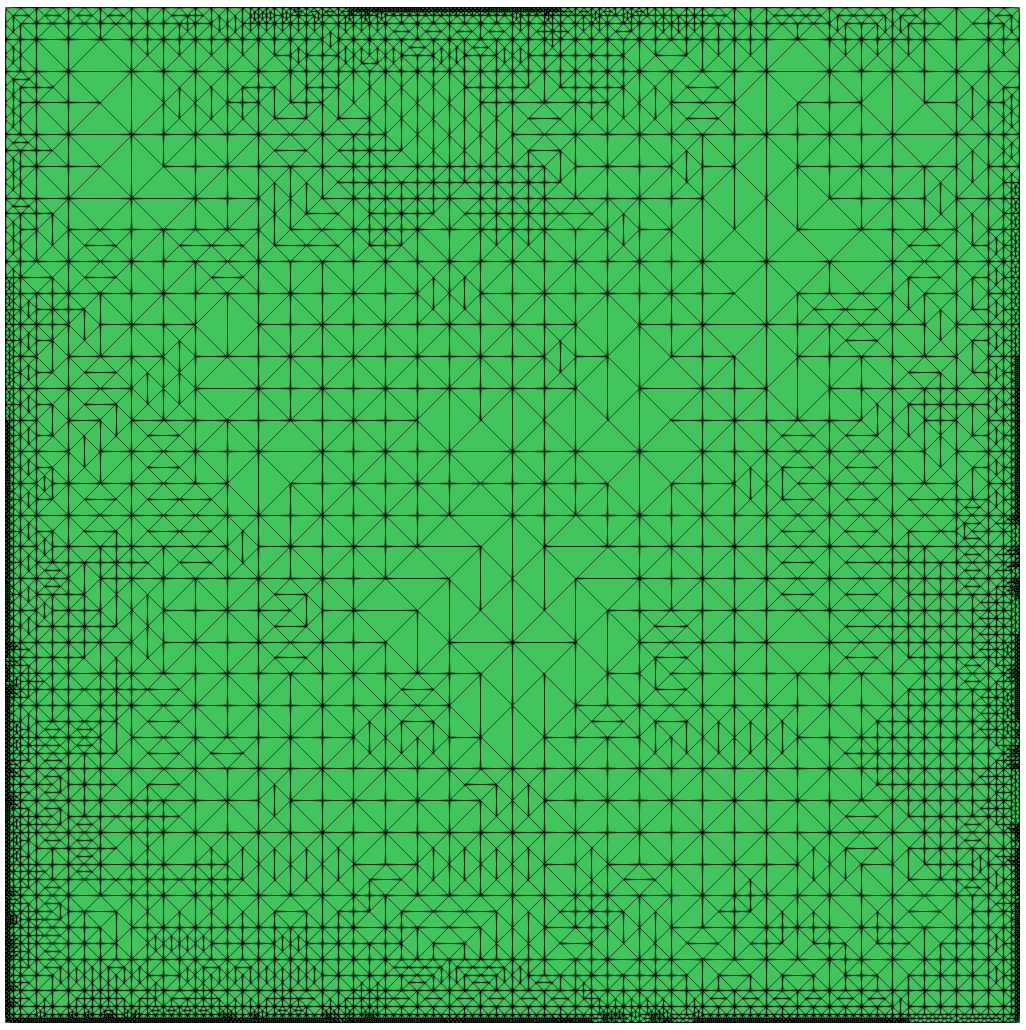}\\
(c)  $k=2$ by $\eta_{classical}$ &(d)  $k=2, C_2=0.1$ by $\eta$
\end{tabular}
\caption{\cref{ex1}. Final meshes for Nitsche's method.}
\label{fig:ex1-Nitsche}
\end{figure}

We now  test the Nitsche's method with $k=1$ and $k=2$
and set $\gamma=10$ in \cref{nitsche}.
\cref{fig:ex1-Nitsche} compares the final meshes generated  using $\eta$ and $\eta_{classical}$. 
We observe similar phenomena to that of the Lagrangian method, i.e., the mesh generated by $\eta_{classical}$ has dense refinement near the interior peaks and sinks while the mesh generated by $\eta$ has dense refinements almost all close to the boundary. The corresponding convergence rates of the true error and error estimators are plotted in 
\cref{Ex1-Nitsche-error}. 
\textcolor{blue}{Again, for both orders, we observe significant improvements of the convergence rate comparing to the classical case. }

{For the Nitsche's method of linear order, we also
  compare the performance with the 
   boundary
  concentrated meshes proposed in \cite{PW19} by Pfefferer and Winkler, which yield what is presently the best a priori error estimate for a non adaptive approximation of the normal flux. The boundary concentrated mesh has a fixed hierarchy structure, i.e., it has uniform mesh size $h^2$ on the boundary and $h\sqrt{\mbox{dist}(T,\Gamma)}$ for interior elements. 
We generate three such meshes in \cref{Ex1-Nitsche-PJ}.}

%To add later 09/17
\begin{figure}[ht]
\centering
\begin{tabular}{ccc}
\includegraphics[trim=12cm 3cm 8cm 3cm,clip=true,width=.30\textwidth]{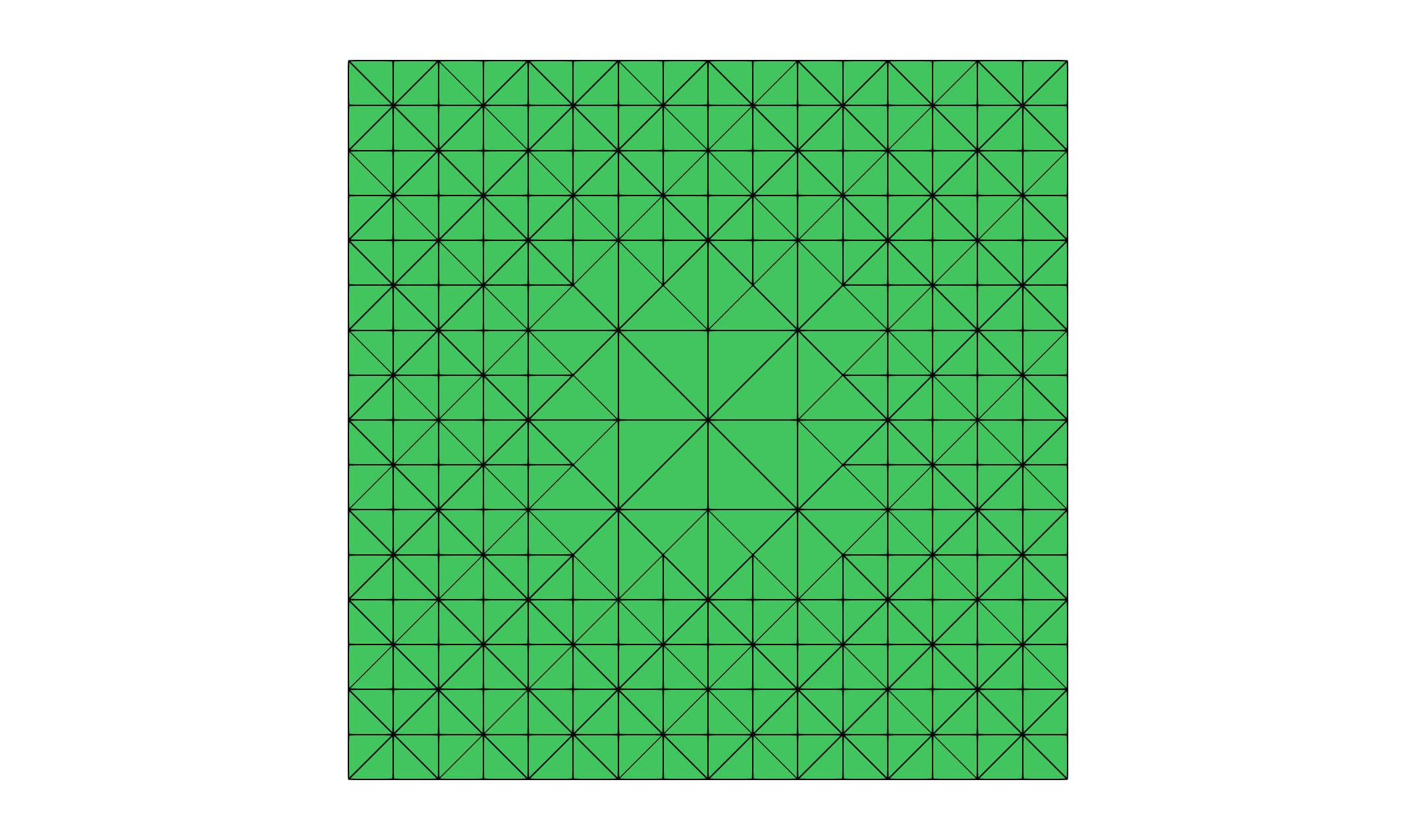} 
&
\includegraphics[trim=12cm 3cm 8cm 3cm,clip=true,width=.30\textwidth]{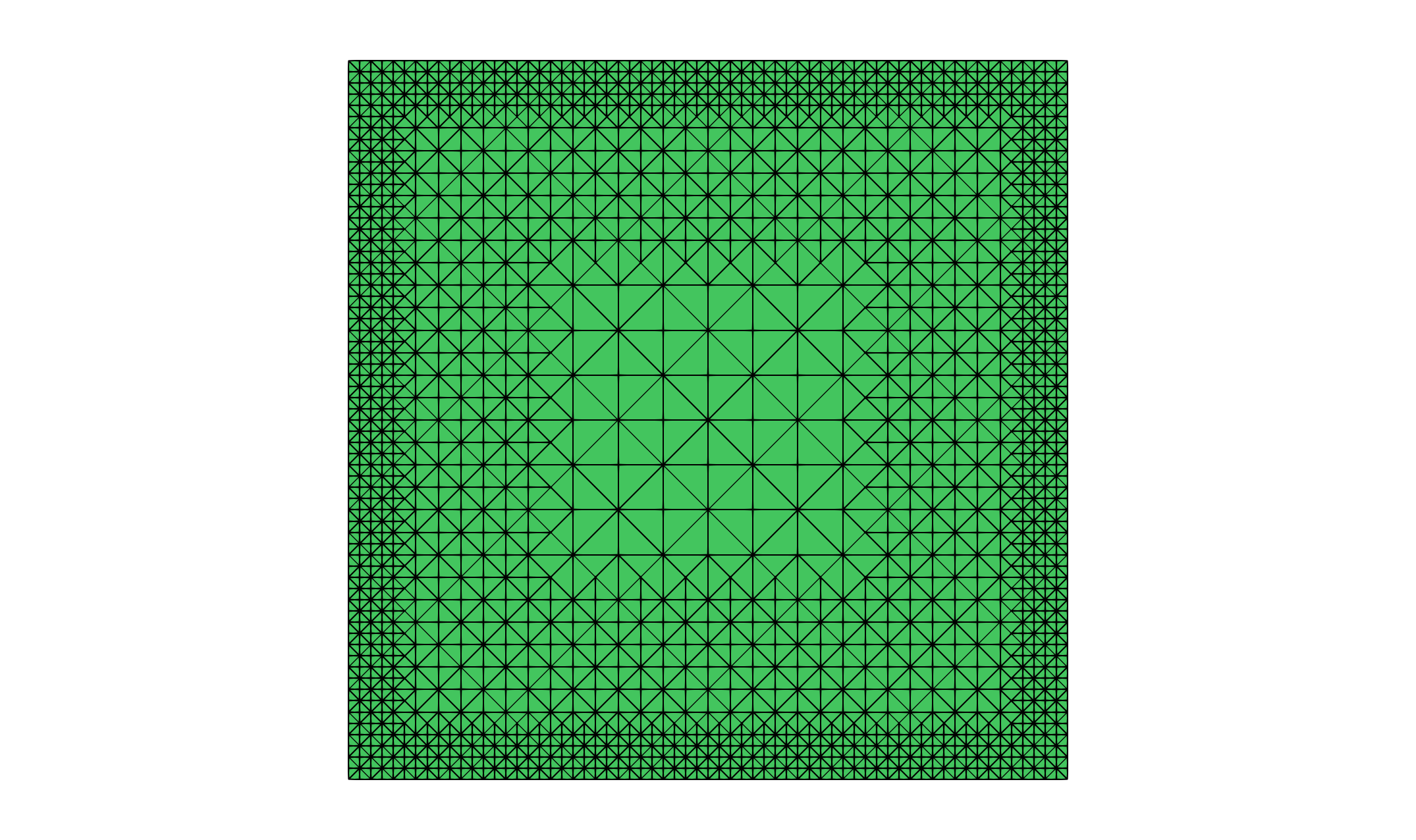}
&
\includegraphics[trim=12cm 3cm 8cm 3cm,clip=true,width=.30\textwidth]{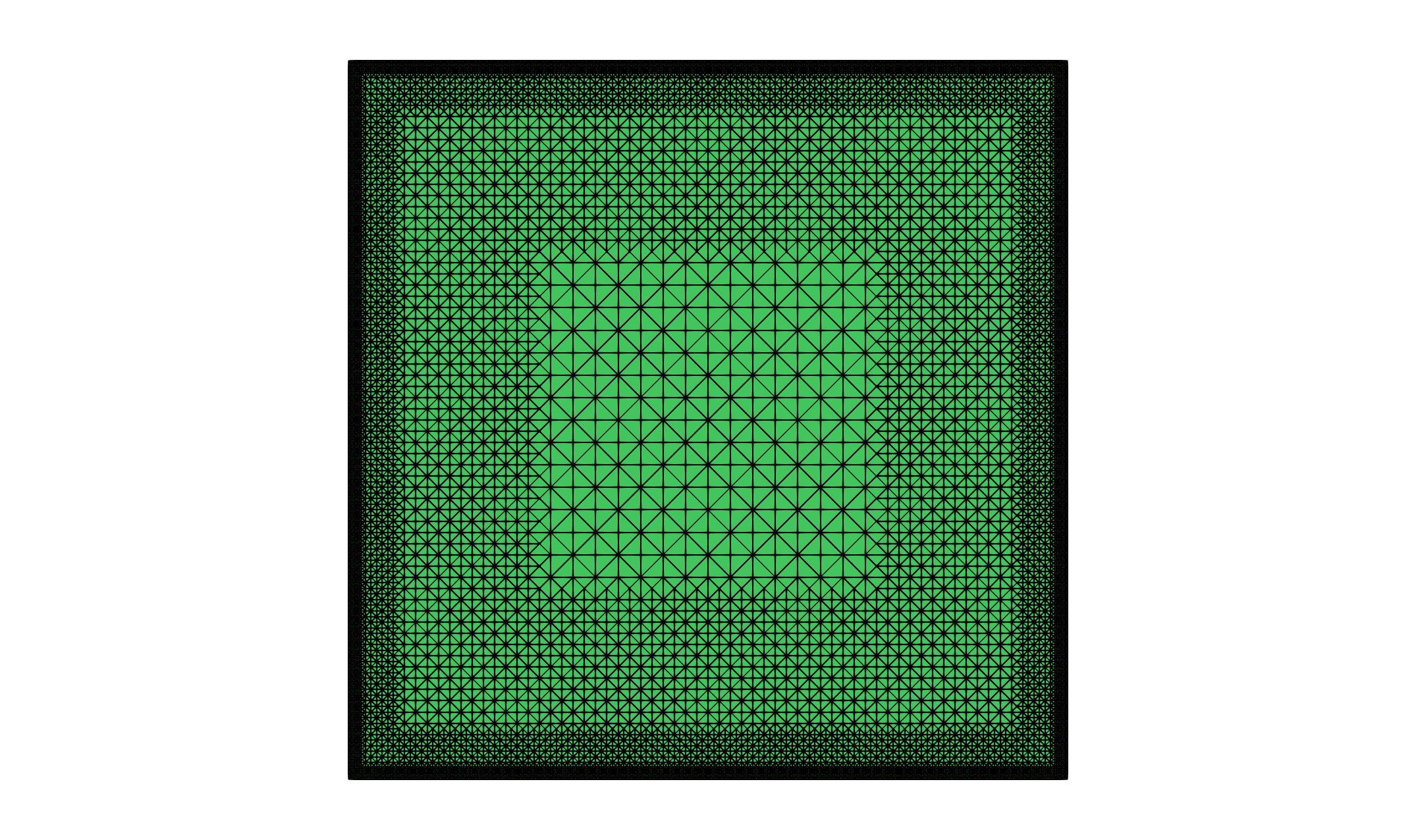}\\
\end{tabular}
\caption{\cref{ex1}. Three meshes for the PW method} %using $\eta_{classical}$ and $\eta$
 \label{Ex1-Nitsche-PJ}
\end{figure}

{
The corresponding log-log curve of $(N, E_2)$ is plotted with legend $E_{2, PW}$ in \cref{Ex1-Nitsche-error} (see the top left figure). 
We observe that the error obtained with this mesh, which is adapted ``a priori'' to a good approximation of the boundary flux, is very close to the error obtained thanks to our error estimator $\eta$. This is not surprising, as the solution of the problem is smooth.  
}

%To add later 09/17
 \begin{figure}[ht]
\centering
\begin{tabular}{ccc}
\includegraphics[width=0.30\textwidth]{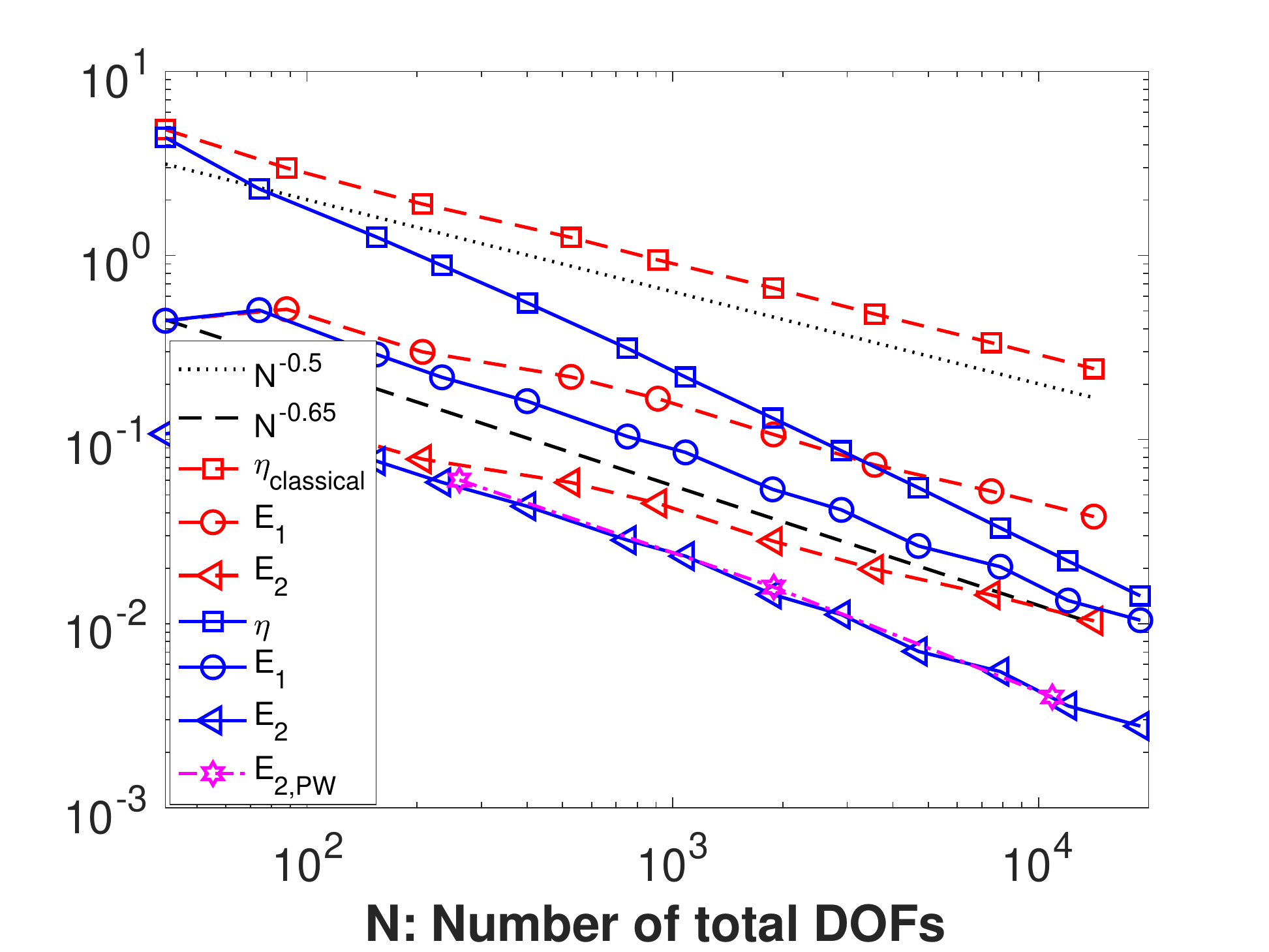}
&\includegraphics[width=0.30\textwidth]{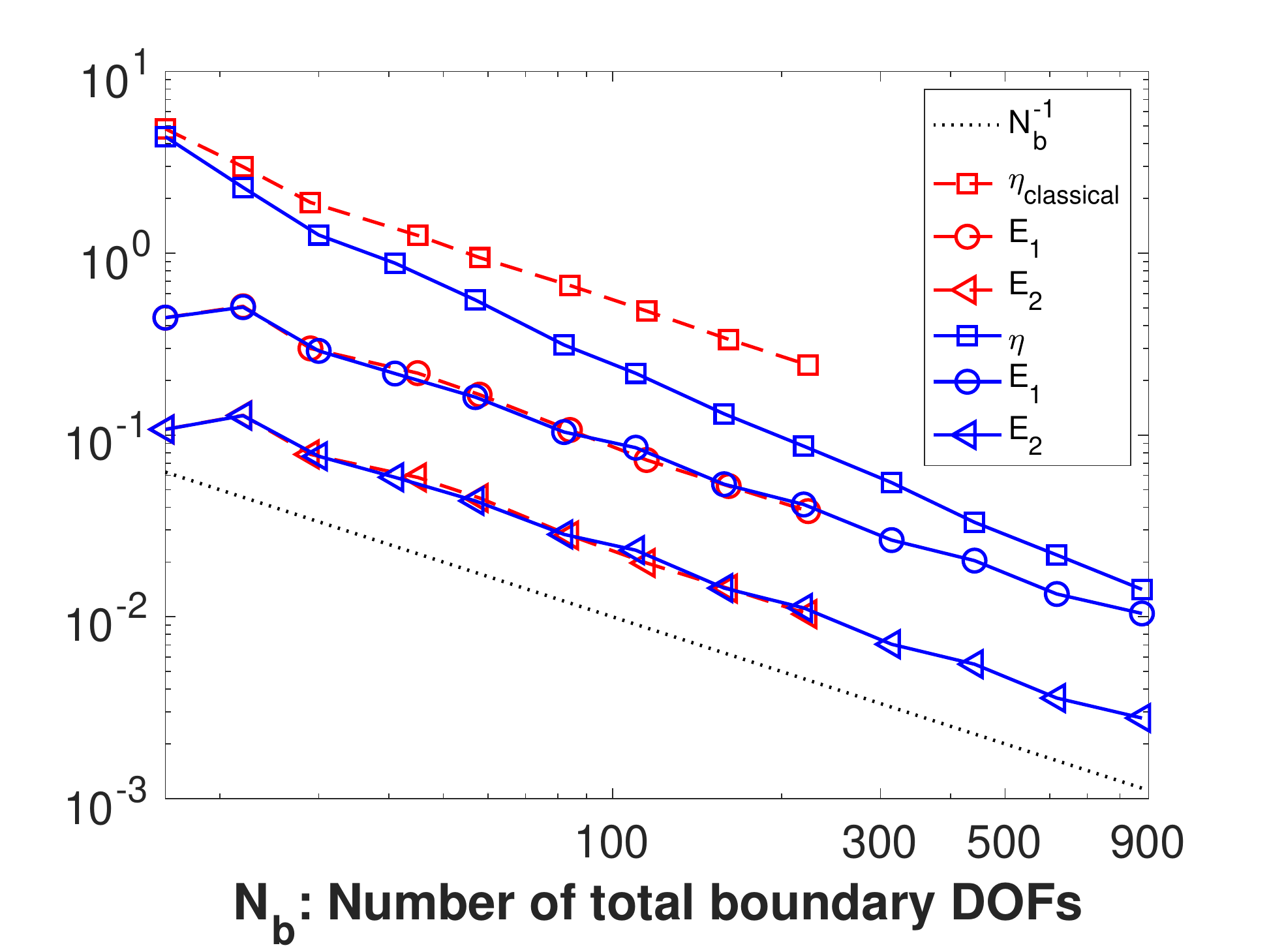}
&\includegraphics[width=0.30\textwidth]{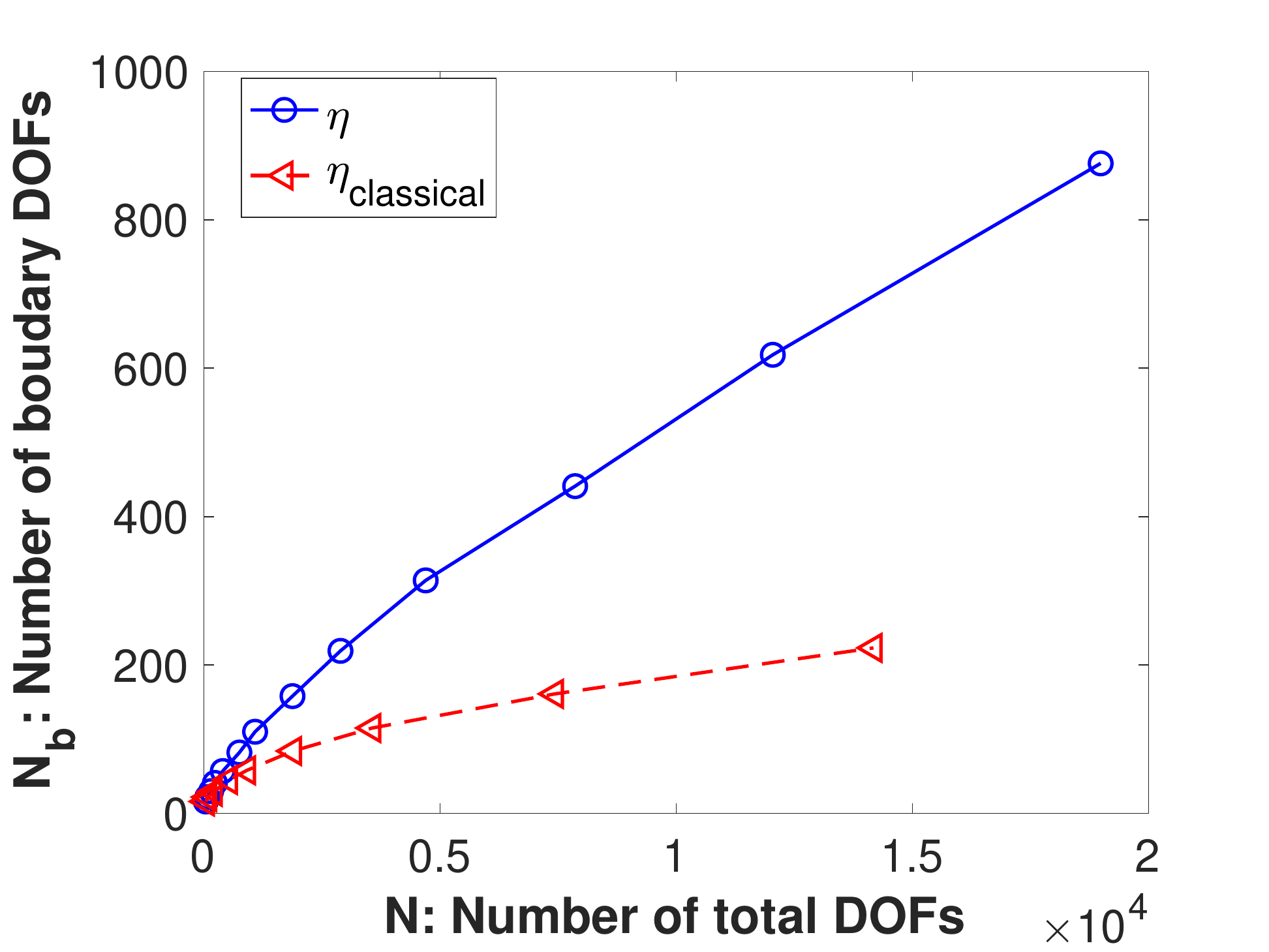}\\
& $k=1,C_2=1.0 $&  \\
  \includegraphics[width=0.30\textwidth]{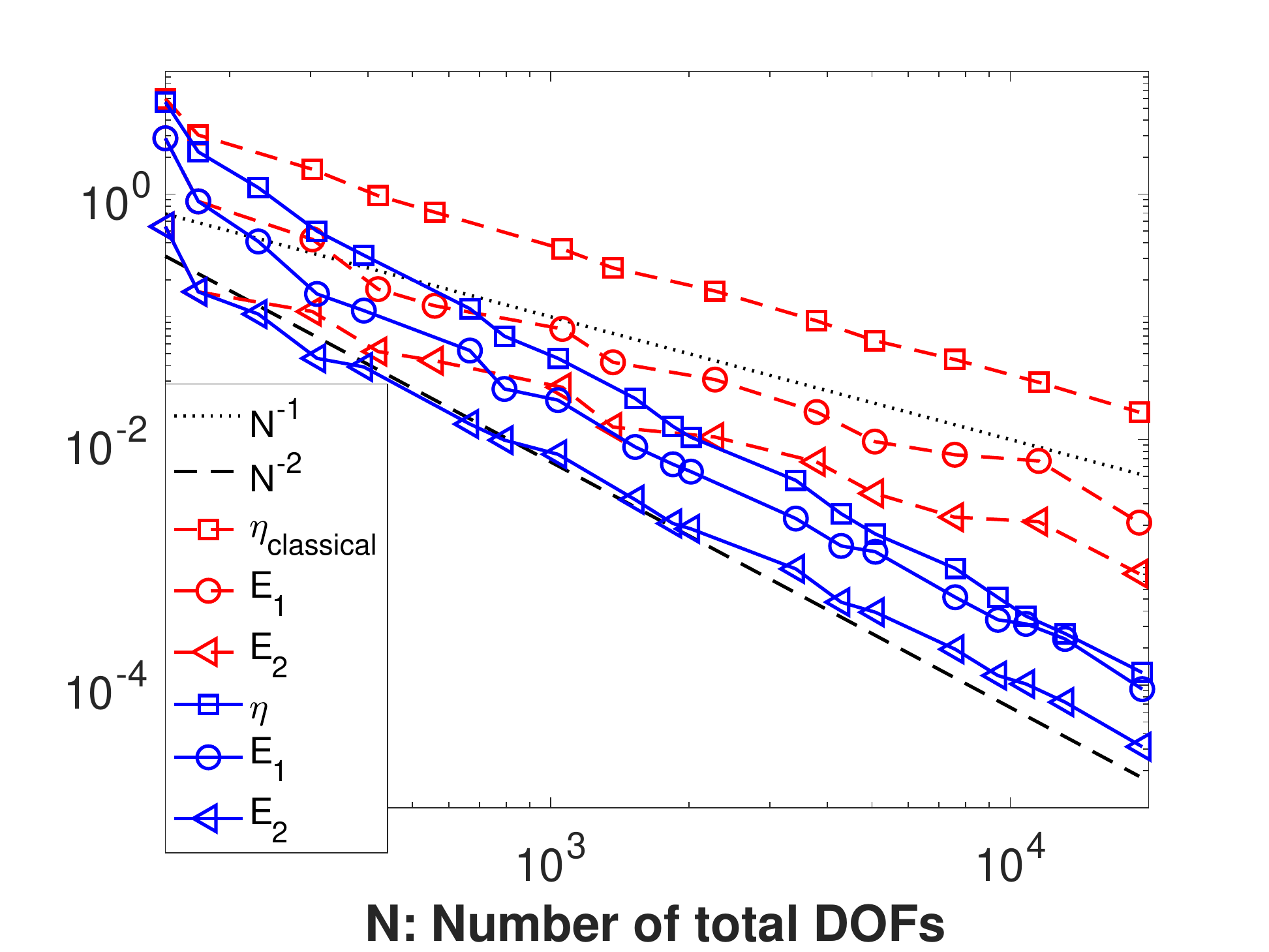}
&\includegraphics[width=0.30\textwidth]{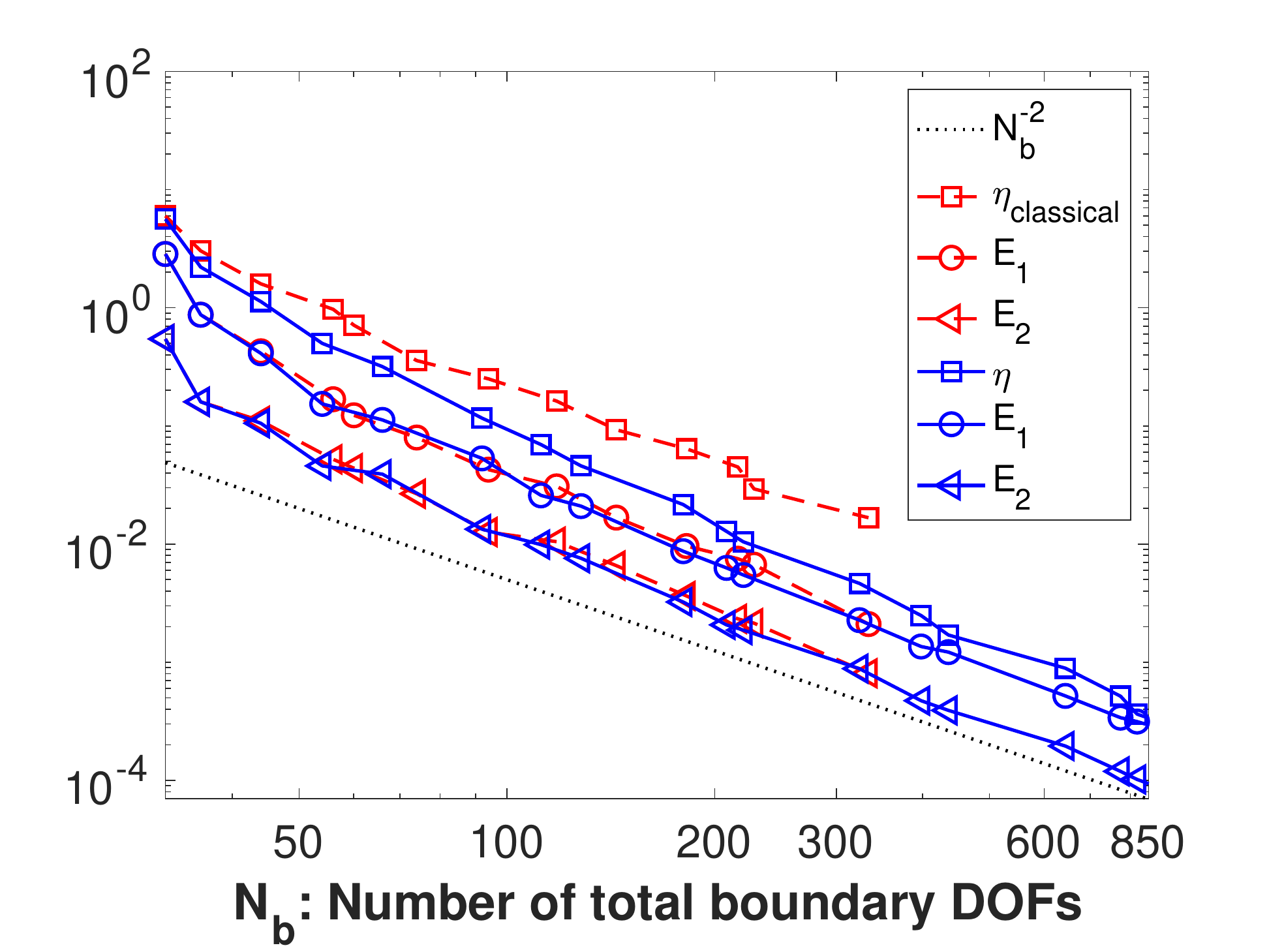}
&\includegraphics[width=0.30\textwidth]{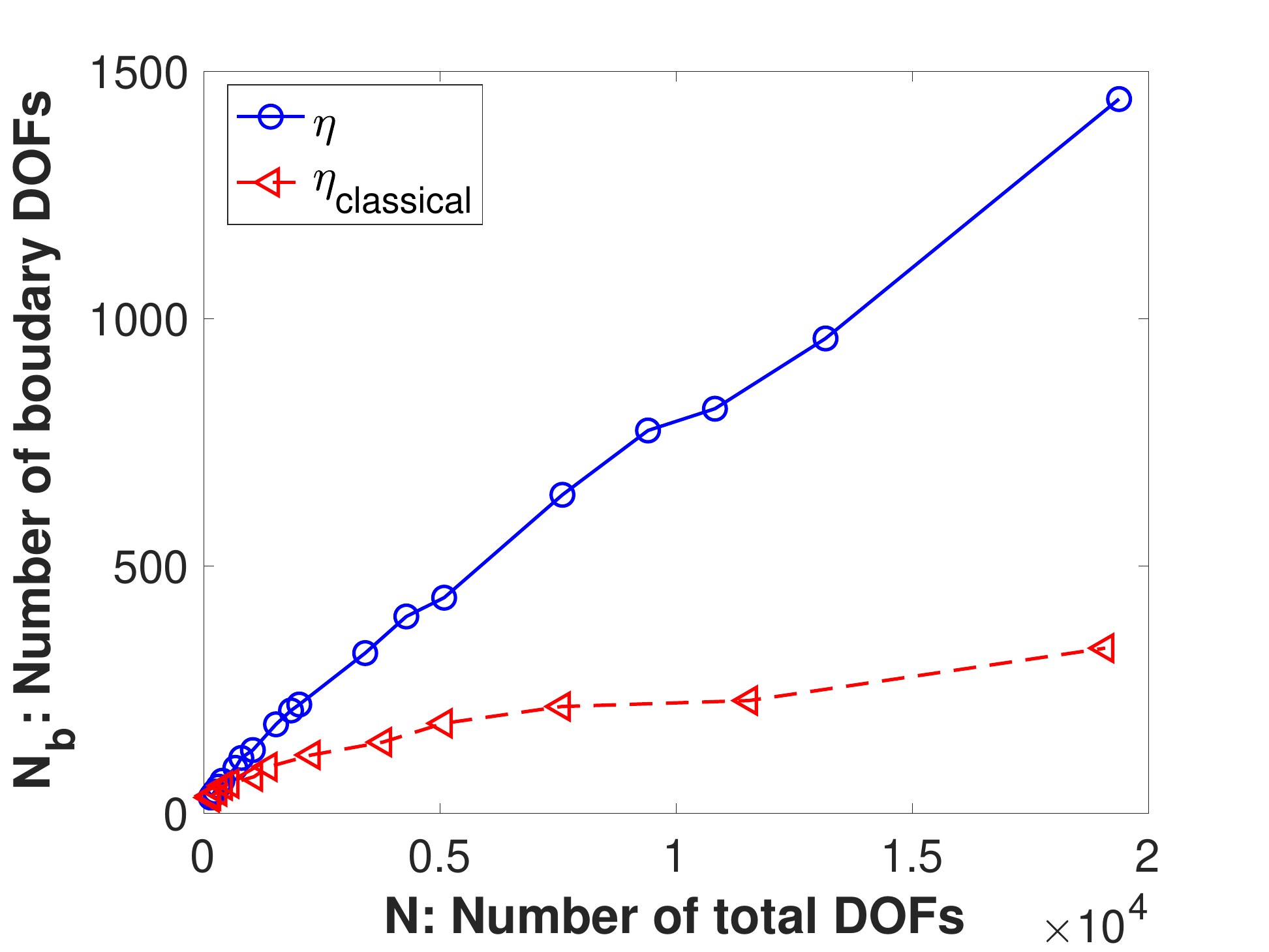}\\
& $k=2, C_2=0.1$& \\
\end{tabular}
\caption{\cref{ex1}.  Convergence comparison for Nitsche's method}
\label{Ex1-Nitsche-error}
\end{figure}

 %To add later 09/17
 \begin{figure}[ht]
\centering
\begin{tabular}{cc}
\includegraphics[width=0.35\textwidth]{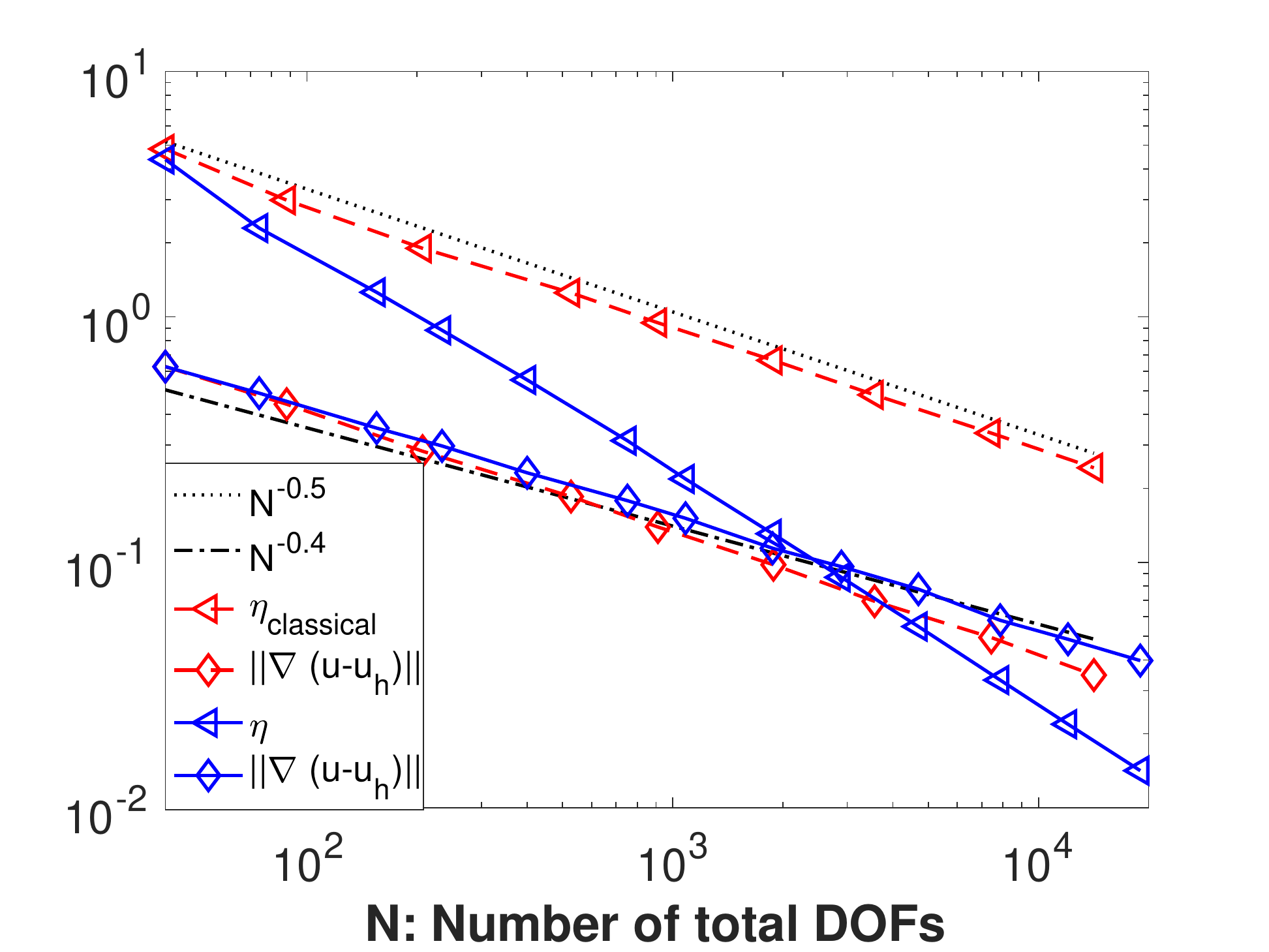}
&\includegraphics[width=0.35\textwidth]{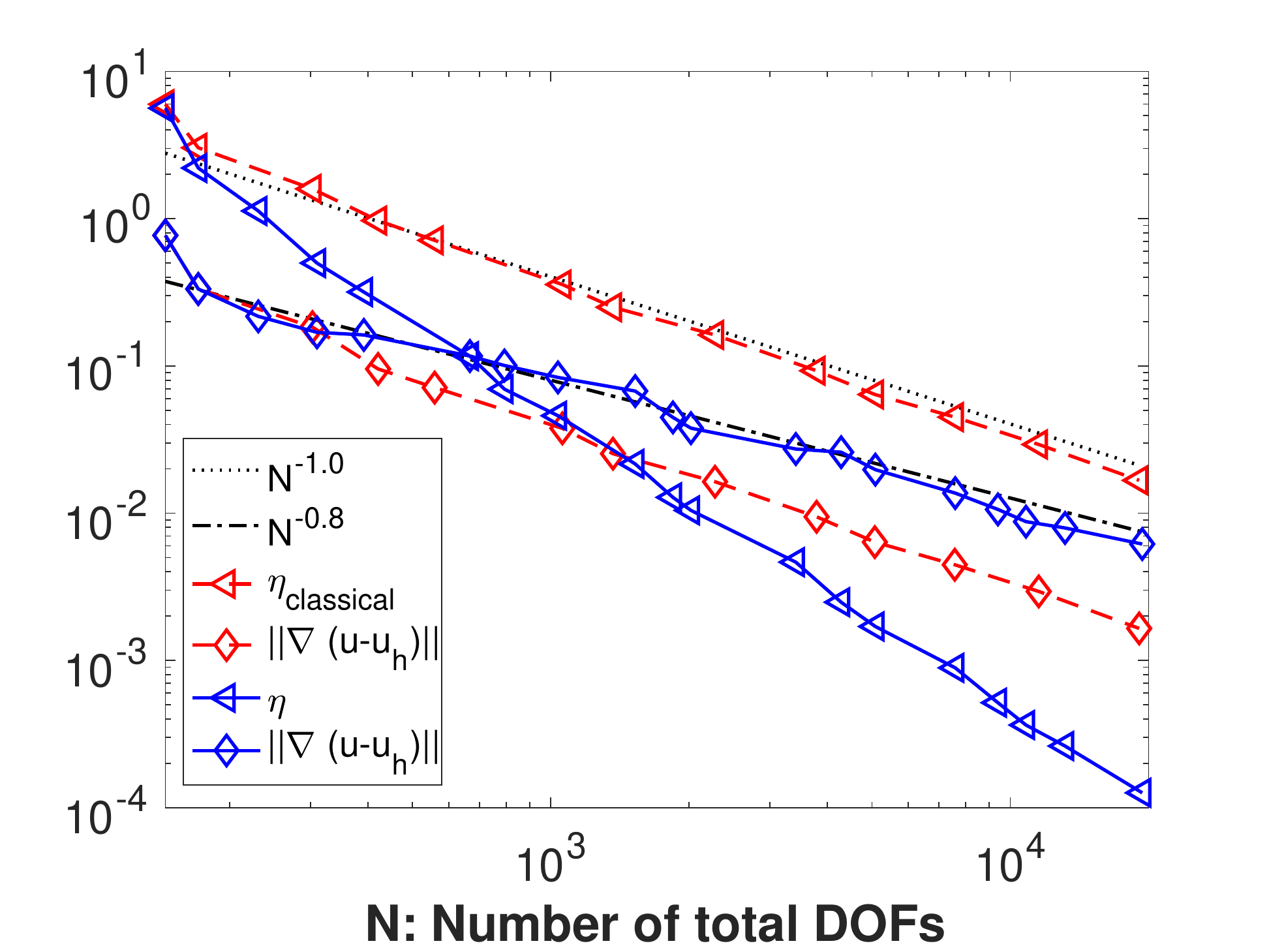}\\
(a)  Nitsche  $k=1$ &(b) Nitsche  $k=2$ \\
\includegraphics[width=0.35\textwidth]{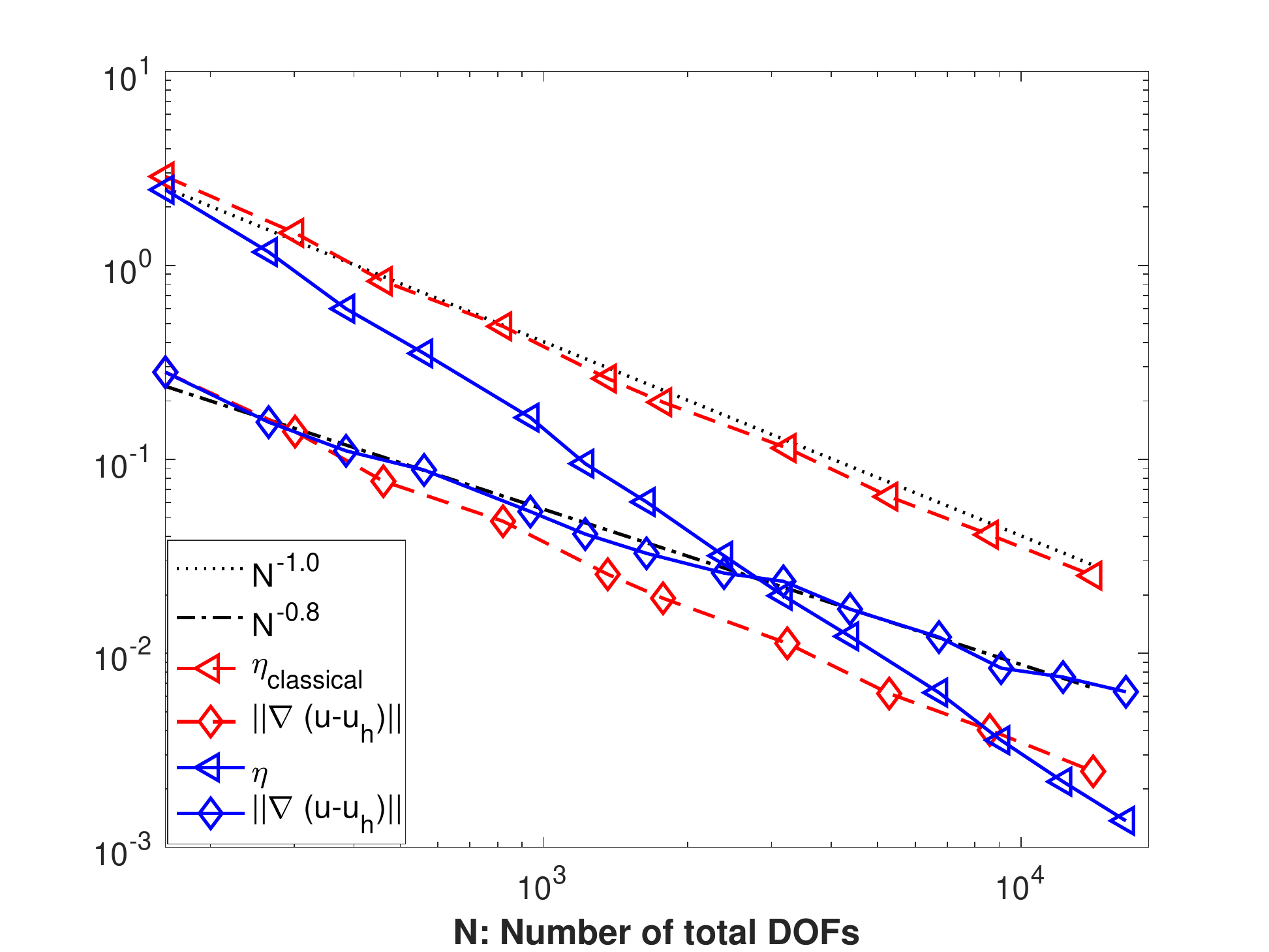}
&\includegraphics[width=0.35\textwidth]{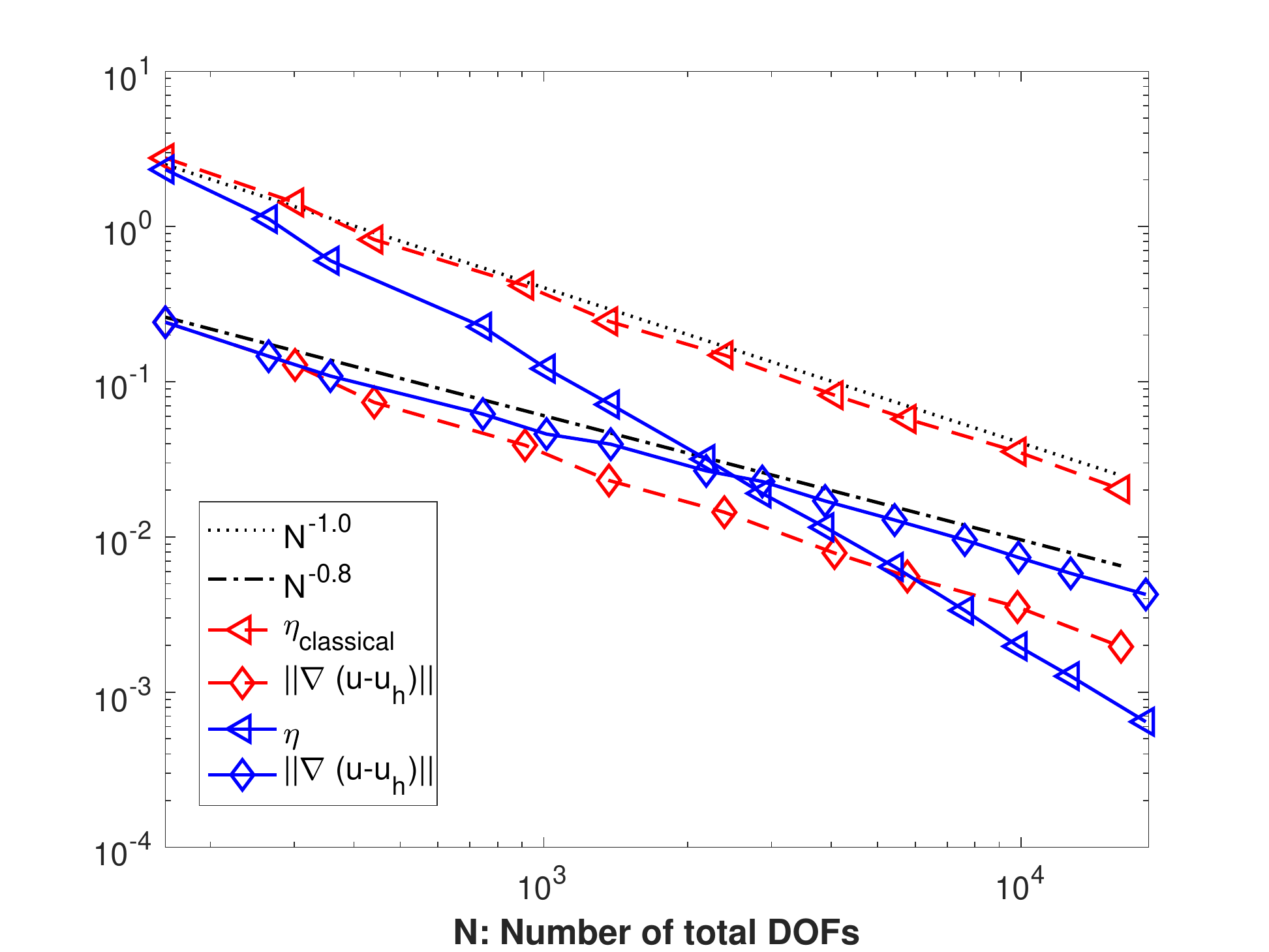}\\
(c) Lagrangian $k=2,k'=0$&(d) Lagrangian $k=2,k'=2$ \\
\end{tabular}
\caption{\cref{ex1}. Convergence comparison with energy error $\|\nabla (u - u_h)\|_{0,\O}$}
\label{fig:energy-error}
\end{figure}

\textcolor{blue}{To provide a more complete picture, in \cref{fig:energy-error}, we instead compare the performance of the error estimators $\eta_{classical}$ and $\eta$, and of the related adaptive mesh refinements, in terms of the convergences in the energy error $\|\nabla(u-u_h)\|_{0,\Omega}$. The results  confirm that $\eta_{classical}$ is optimal for the energy error, while, as it is to be expected, $\eta$ yields only a sub-optimal rate for the energy error in each of the tests.}

\textcolor{blue}{Remark that, despite the fact that both versions of the Lagrangian method that we tested are, for different reasons, suboptimal with respect to the order $k$ of the bulk discretization, our tests show that the proposed error estimator allows to obtain a satisfactory approximation of the normal flux also for such methods.}

\textcolor{blue}{For the remaining examples, to avoid a too large number of redundant tests, we then focus only on Nitsche's method, which is, instead, optimal and which, we recall, is equivalent to the Barbosa-Hughes method. 
Moreover, we observe that he results displayed before in \cref{tab:ex1-Nitsche-k=1} and \cref{Ex1-Nitsche-error} for Nitsche's method both confirm that $E_2$ can, after rescaling, serve as a good alternative to the {more expensive} $E_1$  in evaluating the true error. As for  Nitsche's method with $k=1$, the ratio $E_2/E_1$ is  stable around $0.25$, in the remaining examples, we will {use $E = 4E_2$} as the true error.}
 
\begin{example}\label{ex2}
\textcolor{blue}{
In this example, we test a diffusion problem with variable diffusion coefficient. 
The diffusion coefficient is defined as $a = 1.0+ \sin^2\left(\pi \sqrt{x^2+y^2}\right)$. And the functions $g$ and $f$ are defined such that
the true solution $u$ has the following representation:
\[
u(x,y) = \exp( - \alpha_p ((x - x_p)^2 + (y- y_p)^2)) \mbox{ with }\alpha_p =200, x_p = 0.2, y_p = 0.2.
\]
 Note that this function has a strong peak at the point $(x_p, y_p)$.}

\end{example}
\textcolor{blue}{In the adaptive procedure, the stopping criteria is again set  such that the total number of DOFs is less than $20,000$. 
 We test the Nitsche's method for both the first and second orders with $C_2 = 1.0$. 
For \cref{ex2} with variable coefficient, we observe similar numerical behavior as in \cref{ex1}, see \cref{fig:ex2-Nitsche}--\cref{Ex2-Nitsche-error}. From the left two sub-figures of \cref{Ex2-Nitsche-error}, we observe that in both cases the convergence rates for the true error using $\eta$ is almost double than that using $\eta_{classical}$. In the example, our adaptive algorithm slightly outperforms the PW method. In the case $k=2$ we however observe visible oscillations for the true error. This is not in contrast with the theory. Indeed, the Galerkin method minimizes a discrete energy norm of the error which controls the error on the normal flux only up to a constant. Therefore, refining the mesh does not automatically yield a reduction in the error on the normal flux, particularly if measured, as in our case, in a norm that does not depend on the diffusion coefficient $a$.
}

%To add later 09/17
\begin{figure}[ht]
\centering
\begin{tabular}{cc}
\includegraphics[width=.3\textwidth]{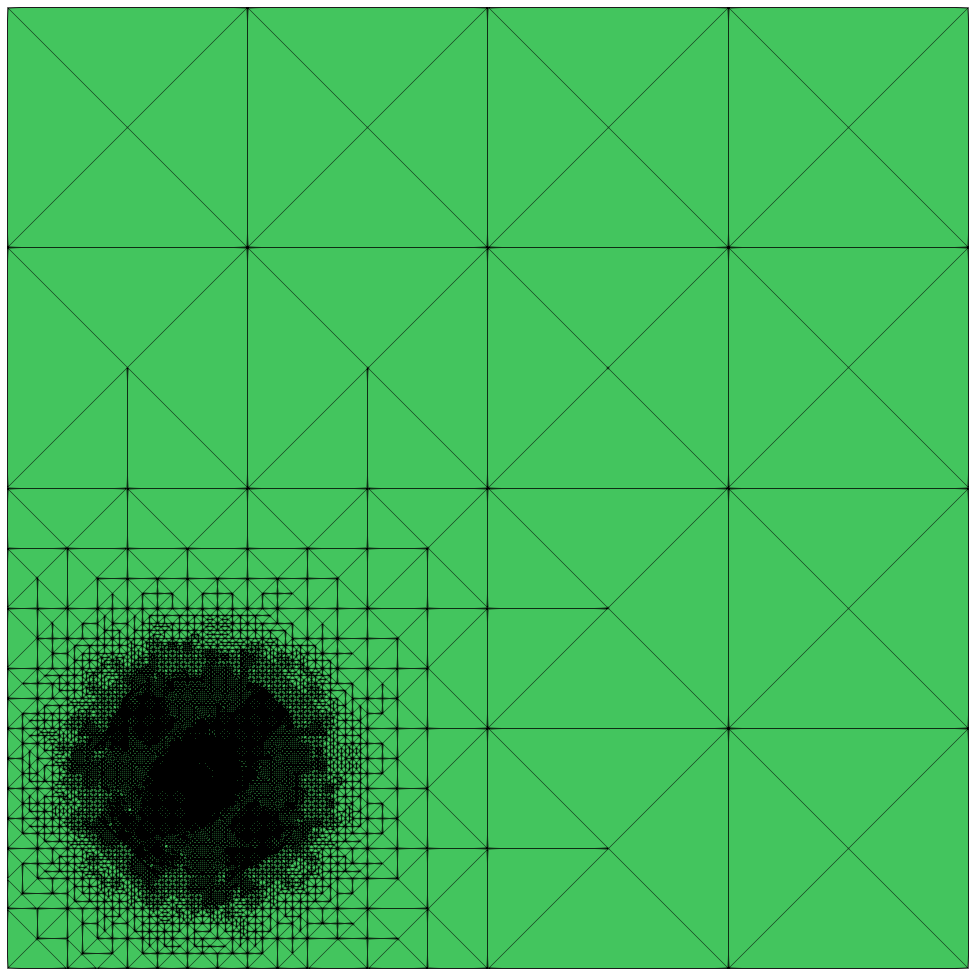} 
  &
 \includegraphics[width=.3\textwidth]{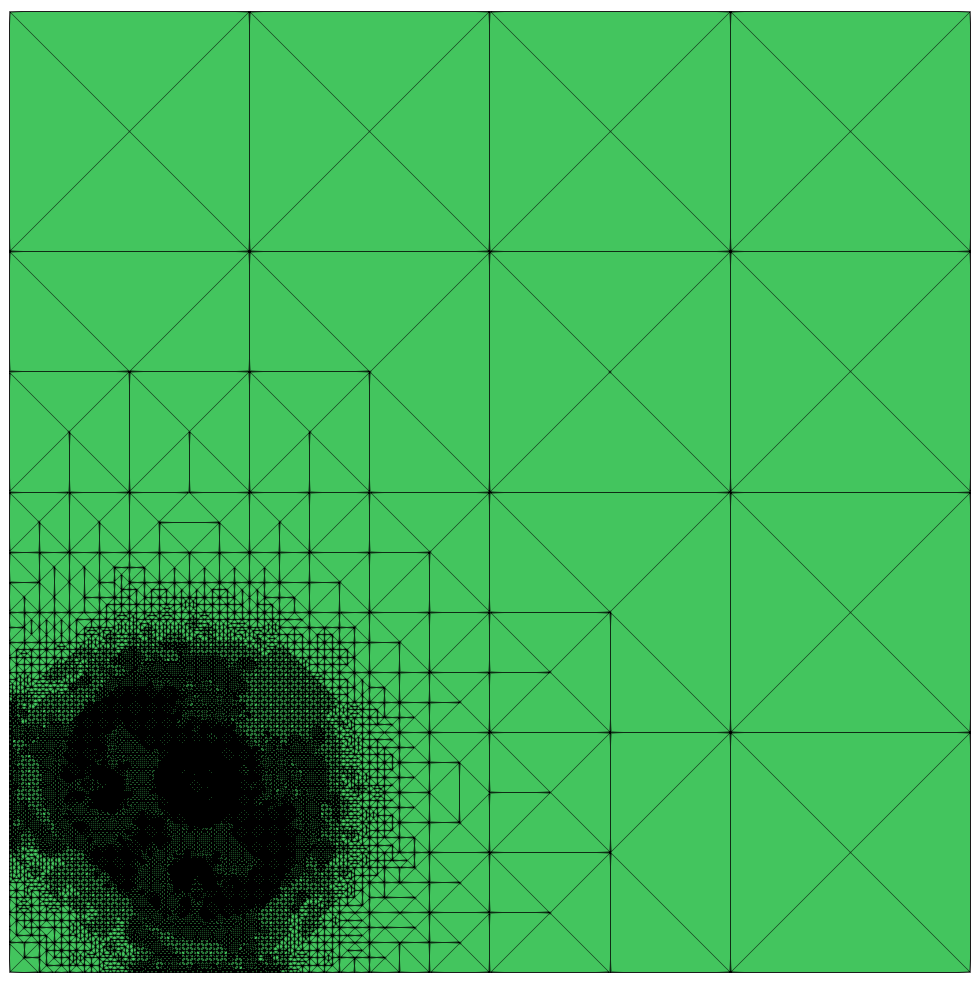}\\
(a) $k=1, \eta_{classical}$ &(b) $k=1, \eta$\\ 
\includegraphics[width=.3\textwidth]{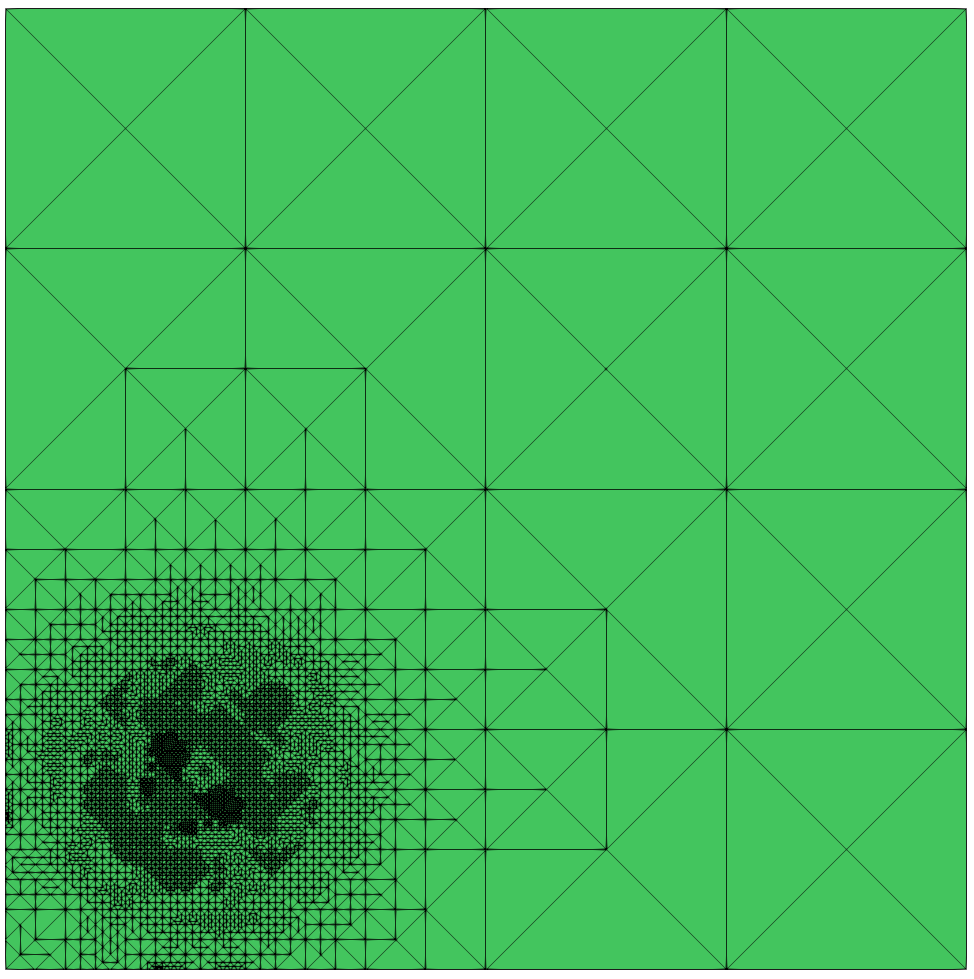} 
  &
\includegraphics[width=.3\textwidth]{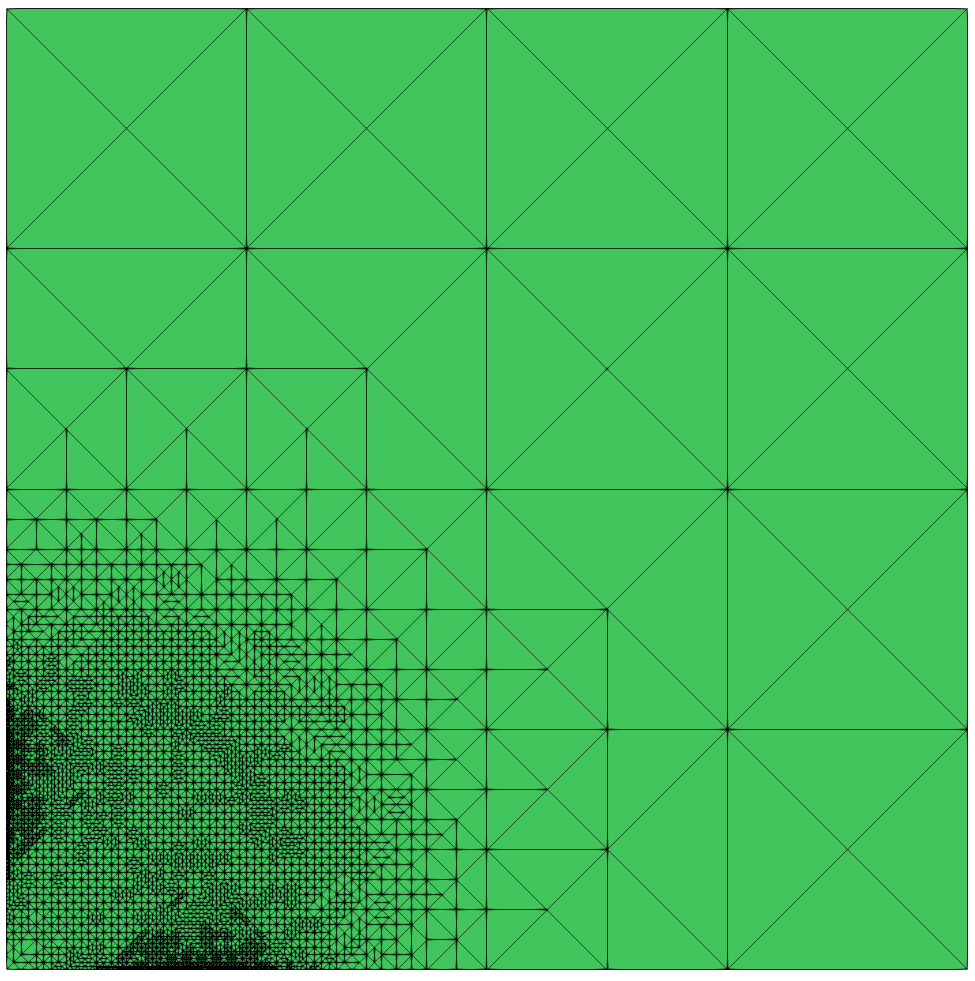}\\
(c)  $k=2, \eta_{classical}$ &(d)  $k=2, \eta$
\end{tabular}
\caption{\cref{ex2}. Final meshes for Nitsche's method ($C_2=1.0$).}
\label{fig:ex2-Nitsche}
\end{figure}

%To add later 09/17
 \begin{figure}[ht]
\centering
\begin{tabular}{ccc}
\includegraphics[width=0.30\textwidth]{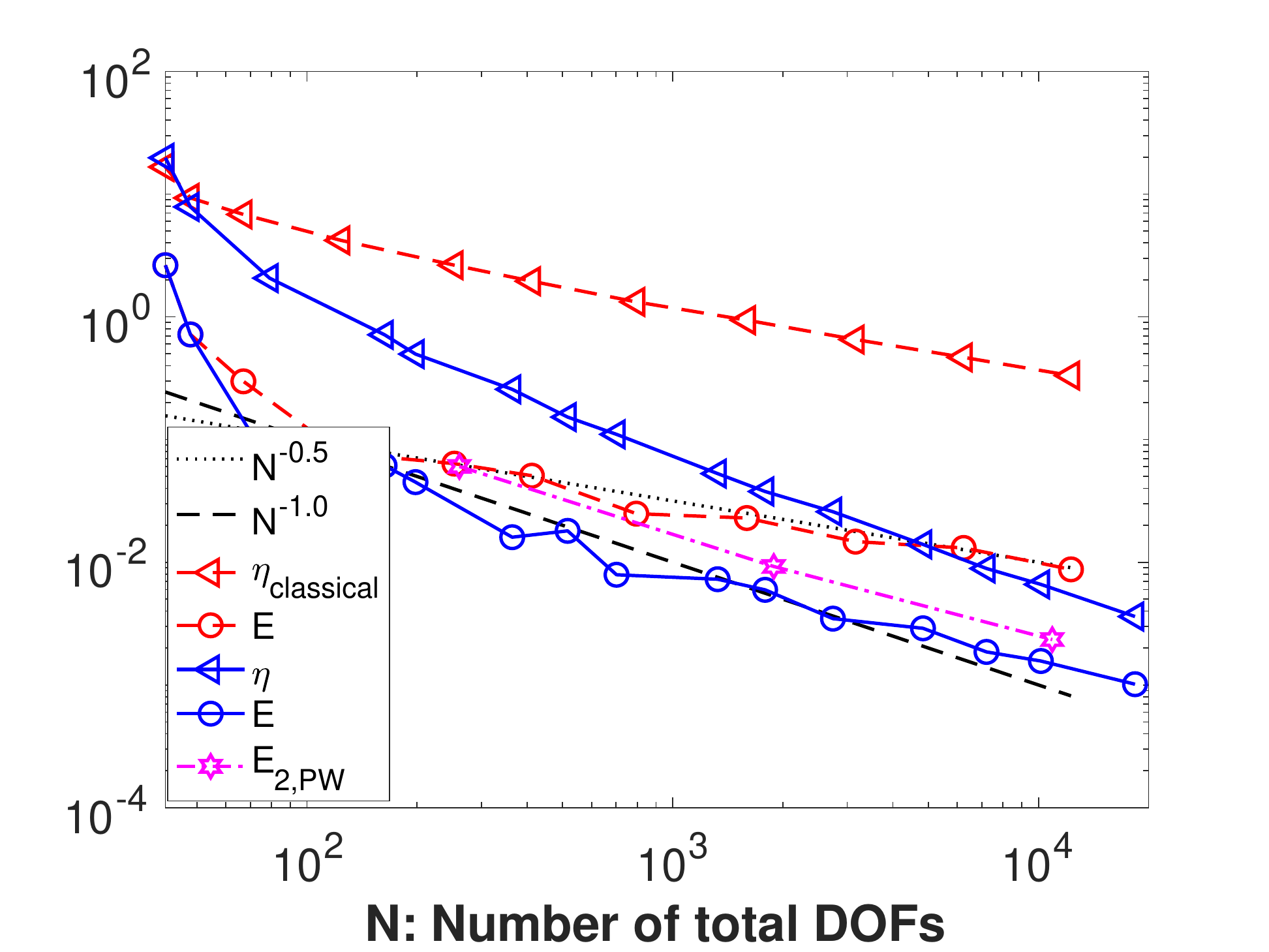}
&\includegraphics[width=0.30\textwidth]{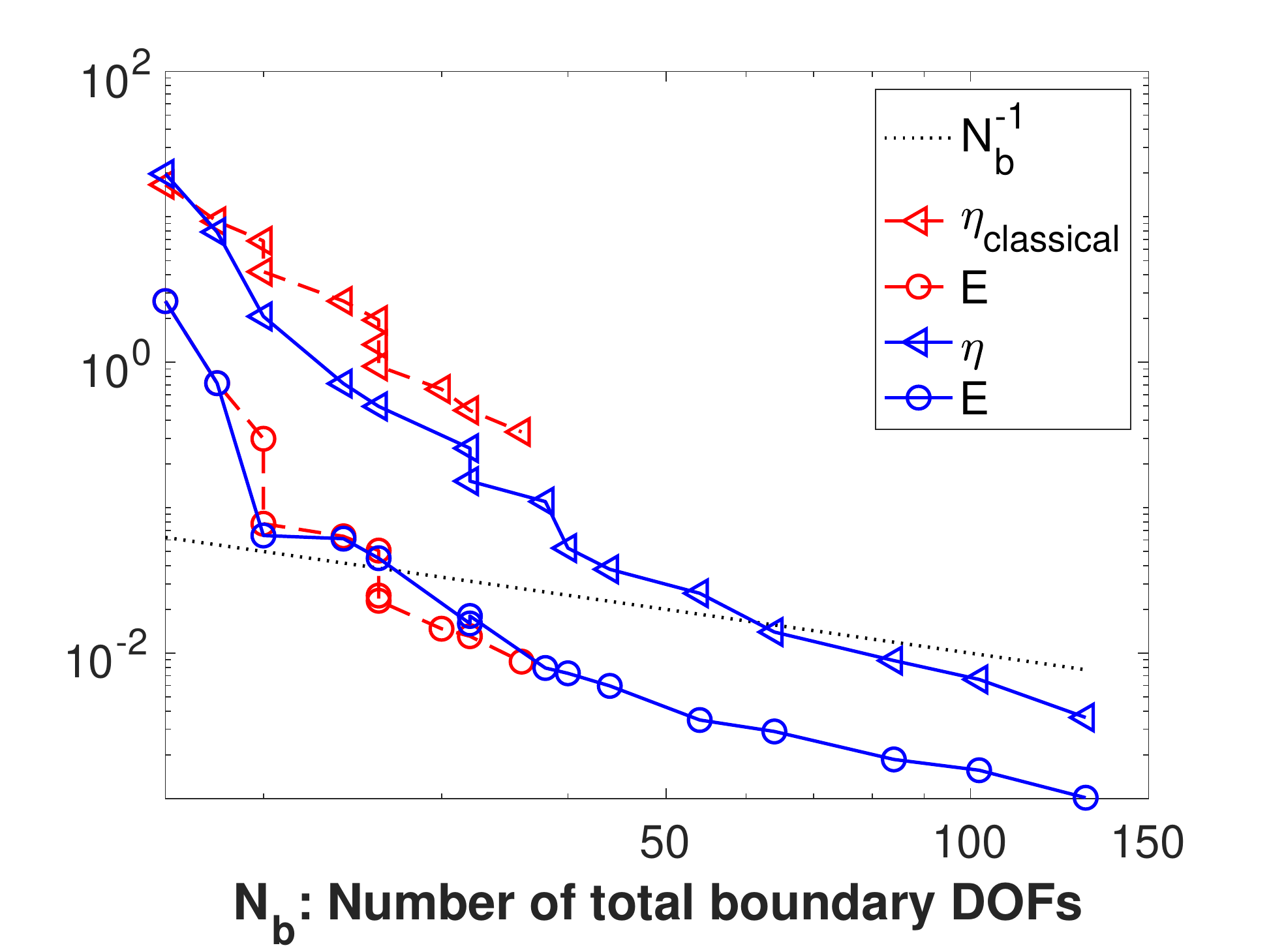}
&\includegraphics[width=0.30\textwidth]{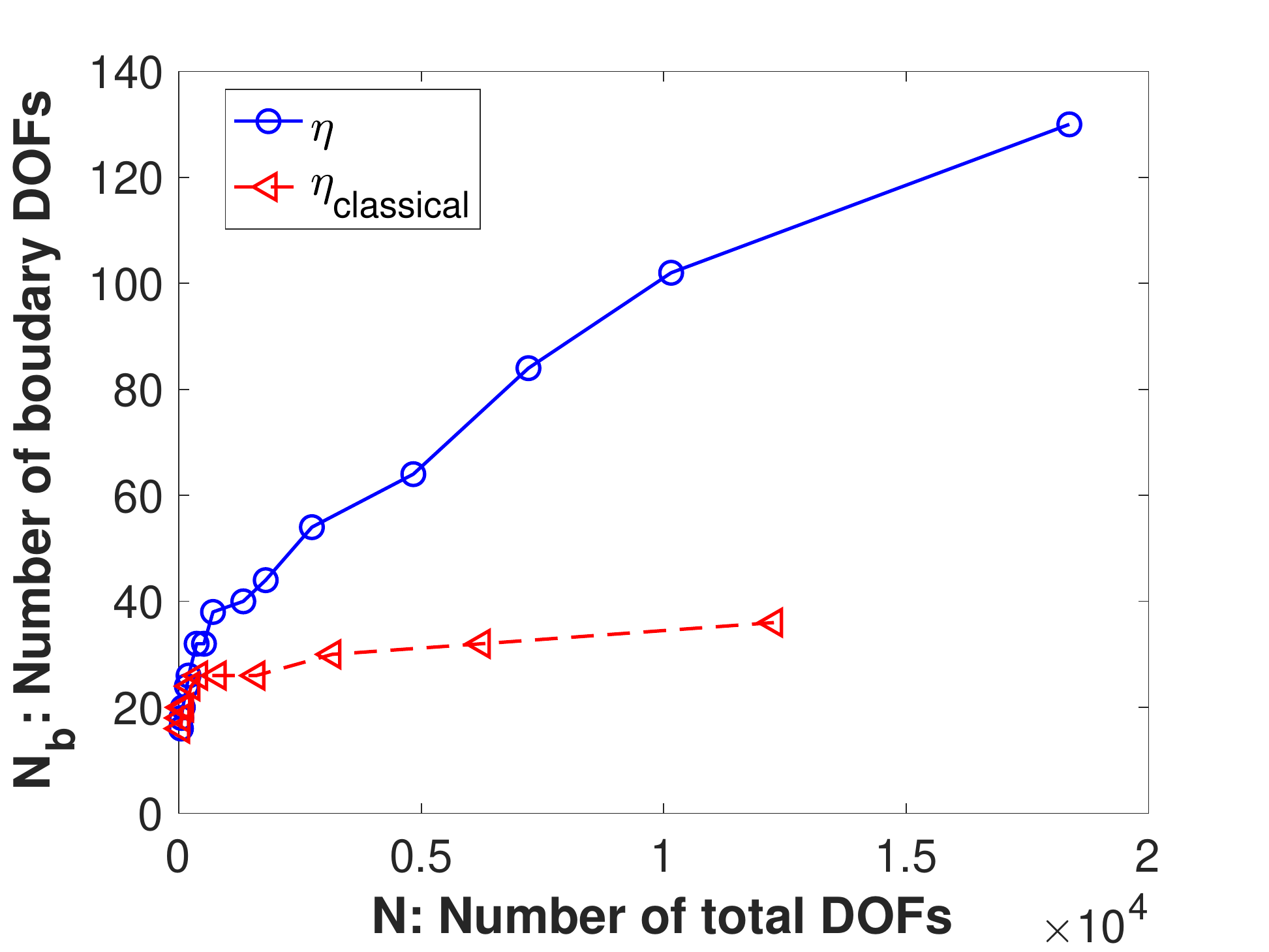}\\
& Nitsche $k=1$& \\
\includegraphics[width=0.30\textwidth]{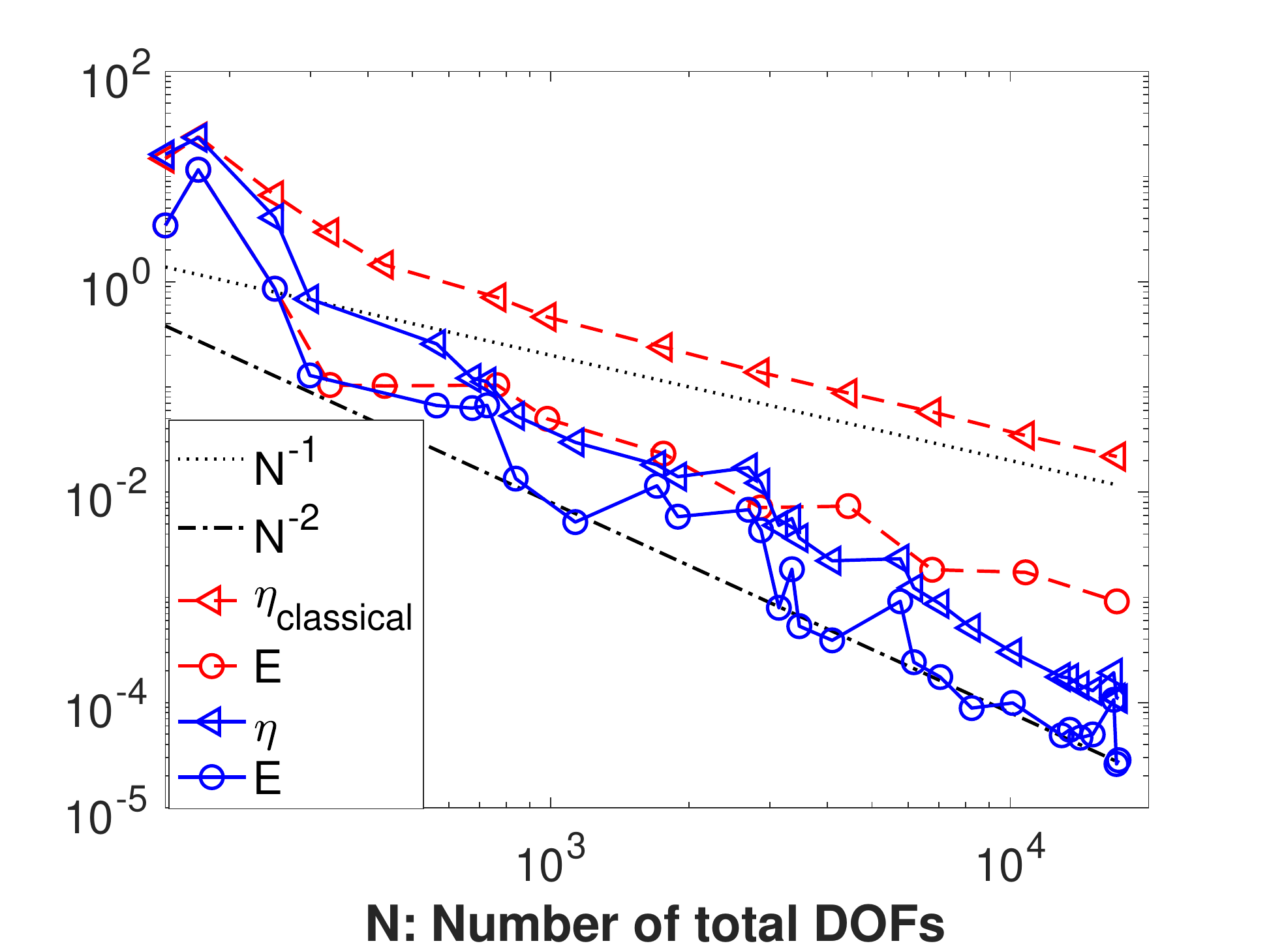}
&\includegraphics[width=0.30\textwidth]{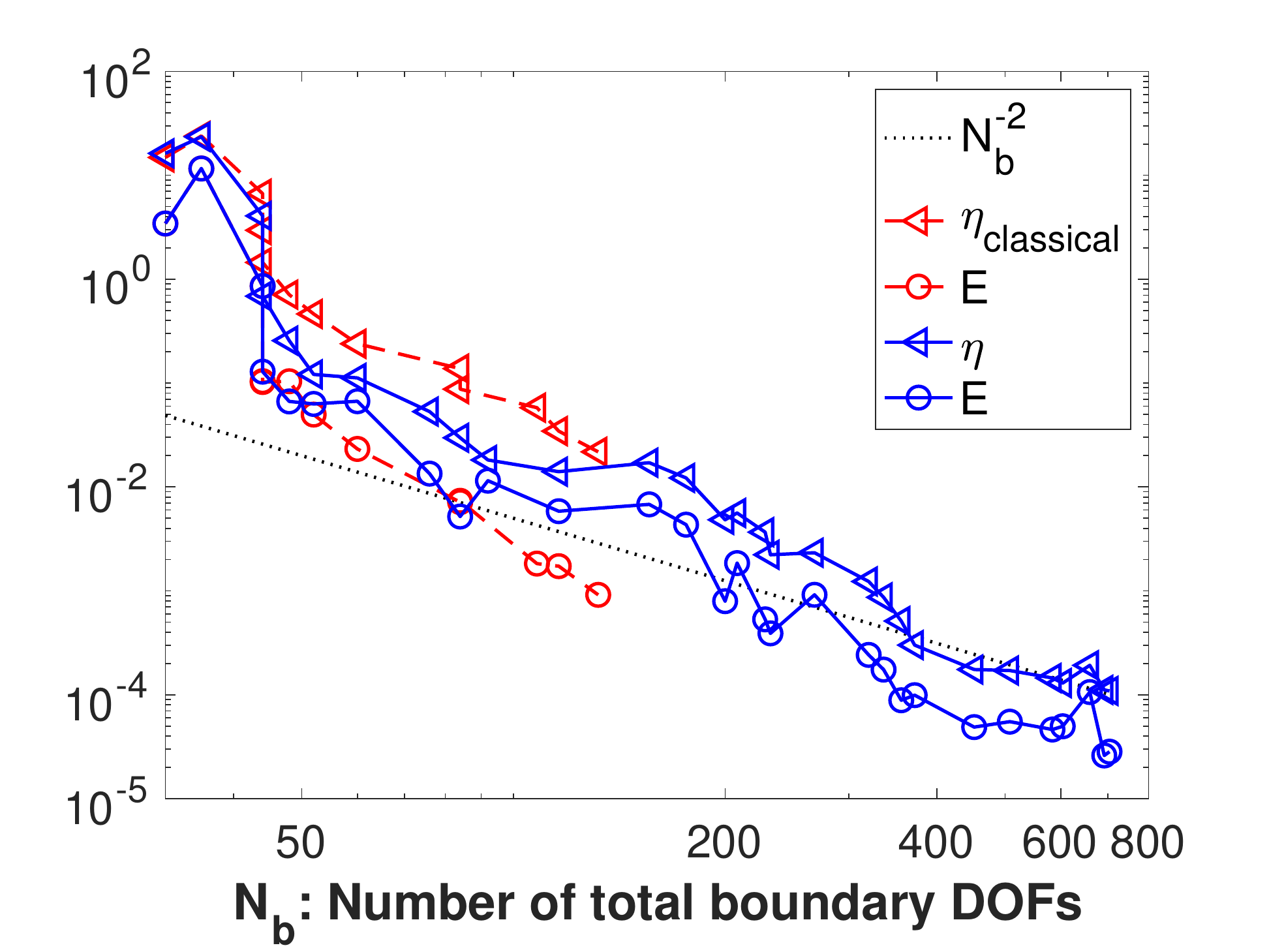}
&\includegraphics[width=0.30\textwidth]{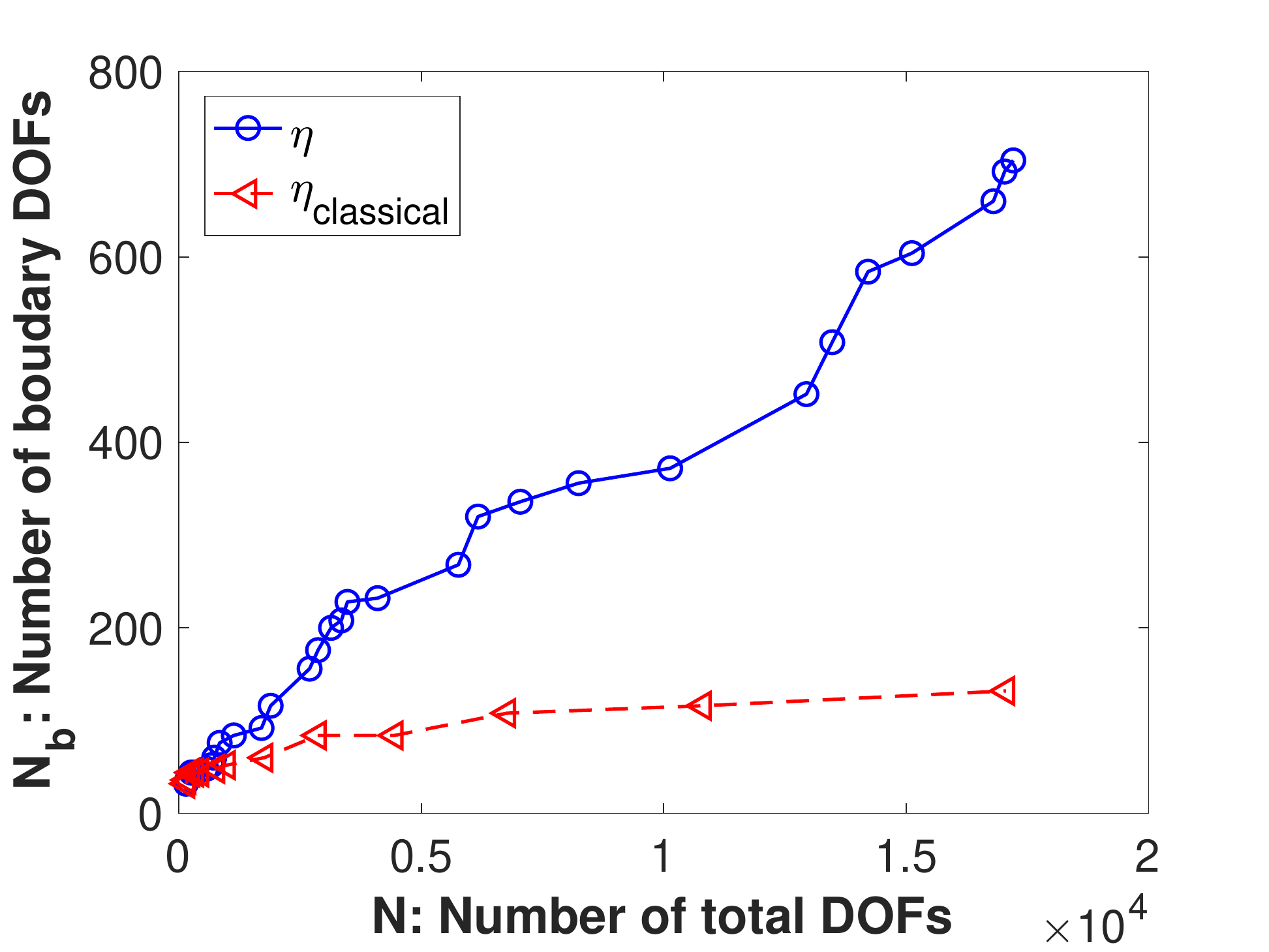}\\
& Nitsche $k=2$& \\
\end{tabular}
\caption{\cref{ex2}.  Convergence comparison for Nitsche's method $(C_2=1)$}
\label{Ex2-Nitsche-error}
\end{figure}

\begin{example}\label{ex3}
\textcolor{blue}{
In this example, we test the L-shaped domain Poisson problem ($a = 1$) with a corner singularity and with an addition interior peak. The true solution has the following representation in polar coordinates:
\[
u(r,\theta) = r^{\alpha} \sin(\alpha \theta) + \exp( - \alpha_p ((x - x_p)^2 + (y- y_p)^2)) \in H^{5/3}(\O) 
\]
where $\alpha = 2/3$, $(\alpha_p, x_p, y_p)$ is the same as in \cref{ex2}, and the $\Omega$ is the L-shaped domain, i.e., $\O = [-1,1]^2 \setminus (0,1)\times(-1,0)$. }
\end{example}
\textcolor{blue}{
In this test, we  set $C_2 = 1$ and $C_2=0.1$ for the first and second order Nitsche's method, respectively.
The convergence rate on uniform meshes is firstly verified  in \cref{tab:ex3}. }
We recall that, according to the  standard a priori error estimates for uniformly refined grids, the error for both the Lagrangian multiplier and the Nitsche's method  behaves like $h^{5/3-1} = h^{2/3} = N^{-1/3}$.}

\begin{table}[ht]
\caption{ \footnotesize{\cref{ex3}: Convergence rates on uniform meshes }}
\label{tab:ex3}
\begin{center}
{
	\begin{tabular}{||c| c  c| c c||}
	\hline
	& \multicolumn{2}{|c|}{Nitsche $k=1$}& 
	\multicolumn{2}{|c|}{Nitsche $k=2$}\\
	\hline
	h            & $E_2$ & rate &$E_2$ & rate \\
	\hline
	1.76E-1  & 1.10E-2&0.86 &6.88E-2&  2.53\\
	8.84E-2  & 1.28E-2&3.10  &1.62E-2& 2.08 \\
	4.42E-2  & 8.84E-3& 0.54 &7.93E-3&  1.03\\
	2.21E-2  & 6.04E-3 & 0.55&5.08E-3&  0.64\\
	1.10E-2  & 4.07E-3& 0.56 &3.36E-3&  0.59\\
	5.52E-3  & 2.72E-3& 0.57 &2.23E-3&  0.59\\
	\hline
	\end{tabular}
	}
\end{center}
\end{table}

\textcolor{blue}{The final meshes obtained for the Nitsche's method are given in \cref{fig:ex3-Nitsche}  and the corresponding convergence results are provided in  \cref{Ex3-Nitsche-error}.
We note that for this problem, even with the low regularity caused by boundary singularity, in both cases the true error $E$ driven by $\eta$ still doubles the convergence rates with respect to those of $\eta_{classical}$.
}

\textcolor{blue}{In this example, in the presence of the corner singularity on the boundary, the adaptive method based on our error estimator shows significantly better performance than the PW method, that has uniform refinement on the boundary. }

%To add later 09/17
\begin{figure}[ht]
\centering
\begin{tabular}{cc}
\includegraphics[width=.3\textwidth]{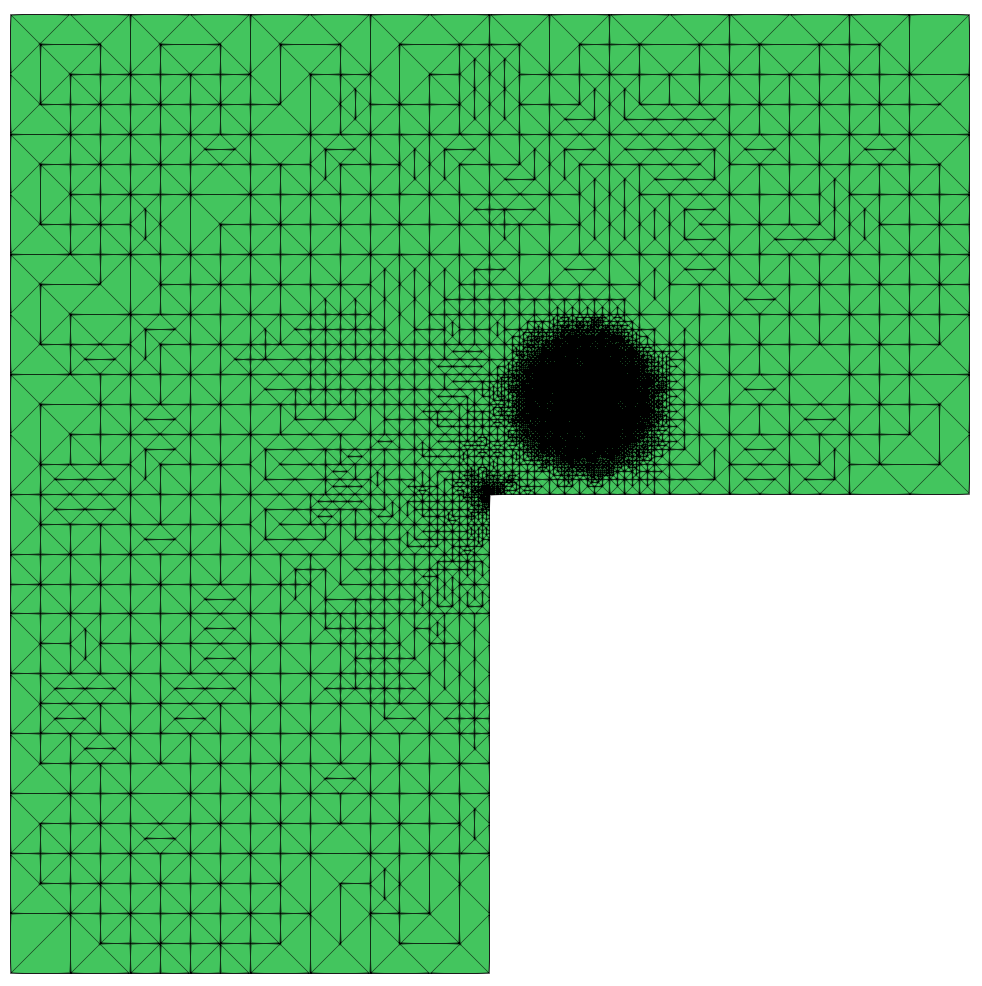} 
  &
 \includegraphics[width=.3\textwidth]{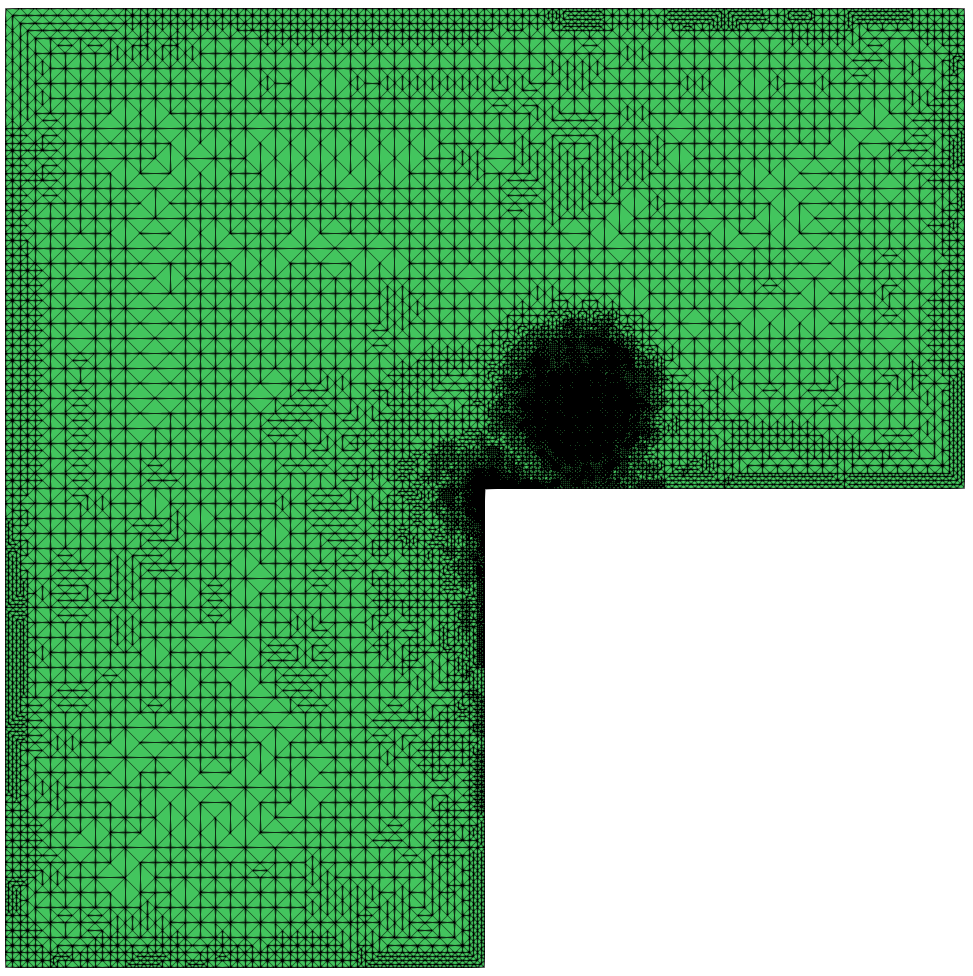}\\
(a) $k=1$, $\eta_{classical}$ &(b) $k=1, C_2=1$, $\eta$\\ 
\includegraphics[width=.3\textwidth]{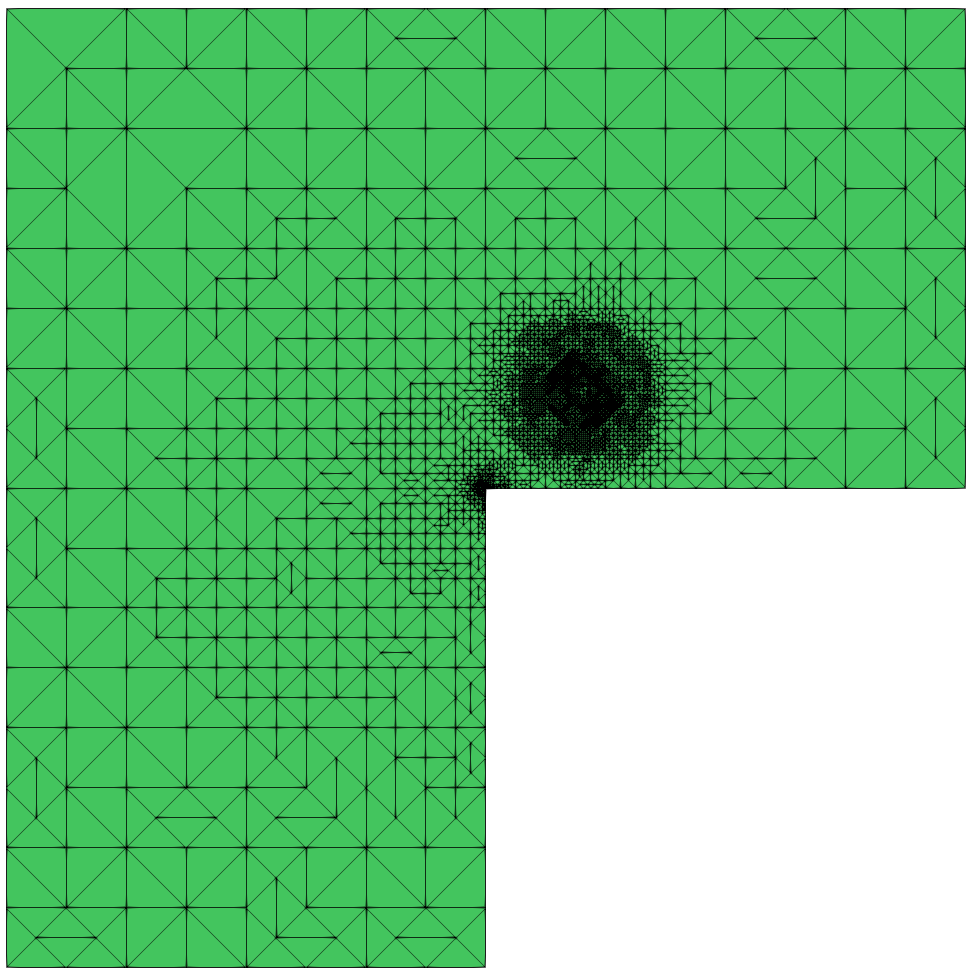} 
  &
\includegraphics[width=.3\textwidth]{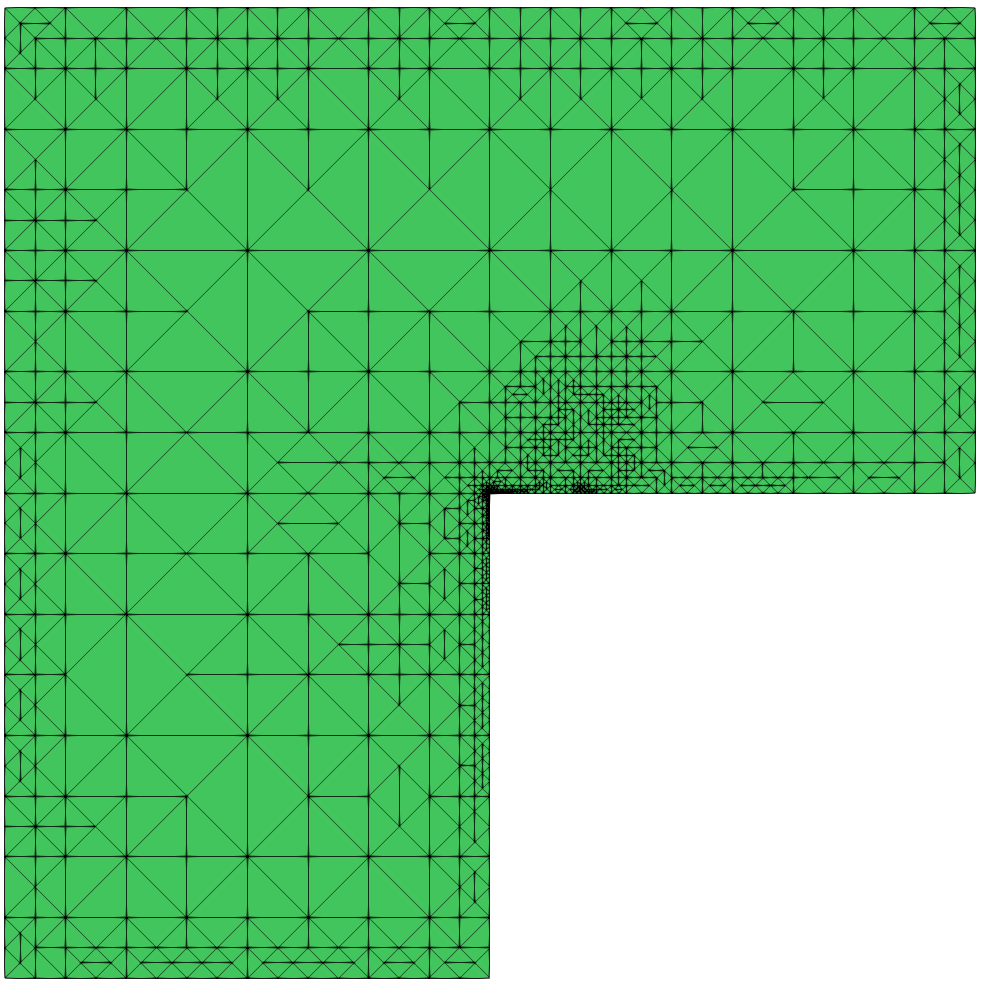}\\
(c)  $k=2$, $\eta_{classical}$ &(d)  $k=2, C_2=0.1$, $\eta$
\end{tabular}
\caption{\cref{ex3}. Final meshes for Nitsche's method.}
\label{fig:ex3-Nitsche}
\end{figure}

%To add later 09/17
 \begin{figure}[ht]
\centering
\begin{tabular}{ccc}
\includegraphics[width=0.30\textwidth]{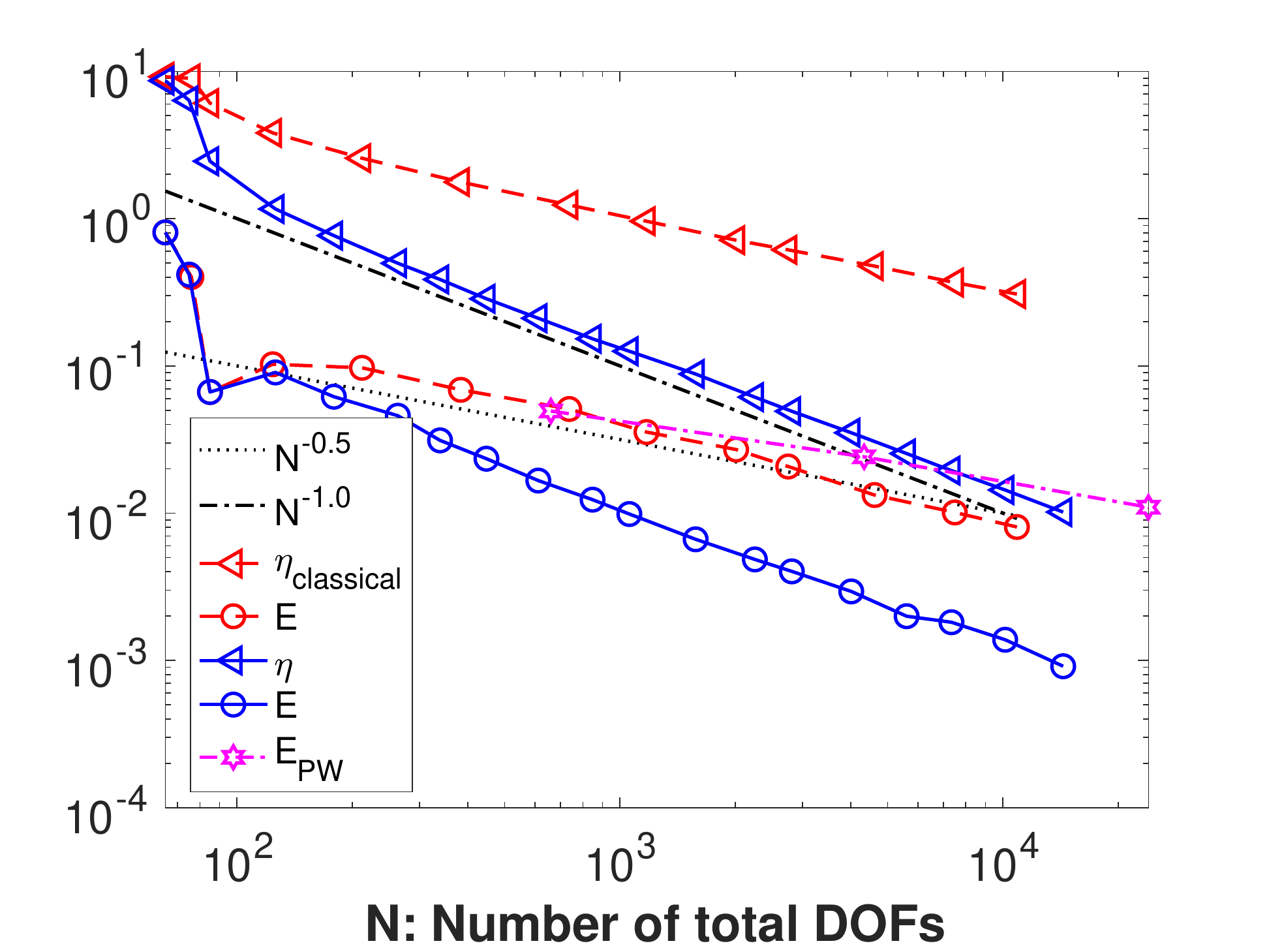}
&\includegraphics[width=0.30\textwidth]{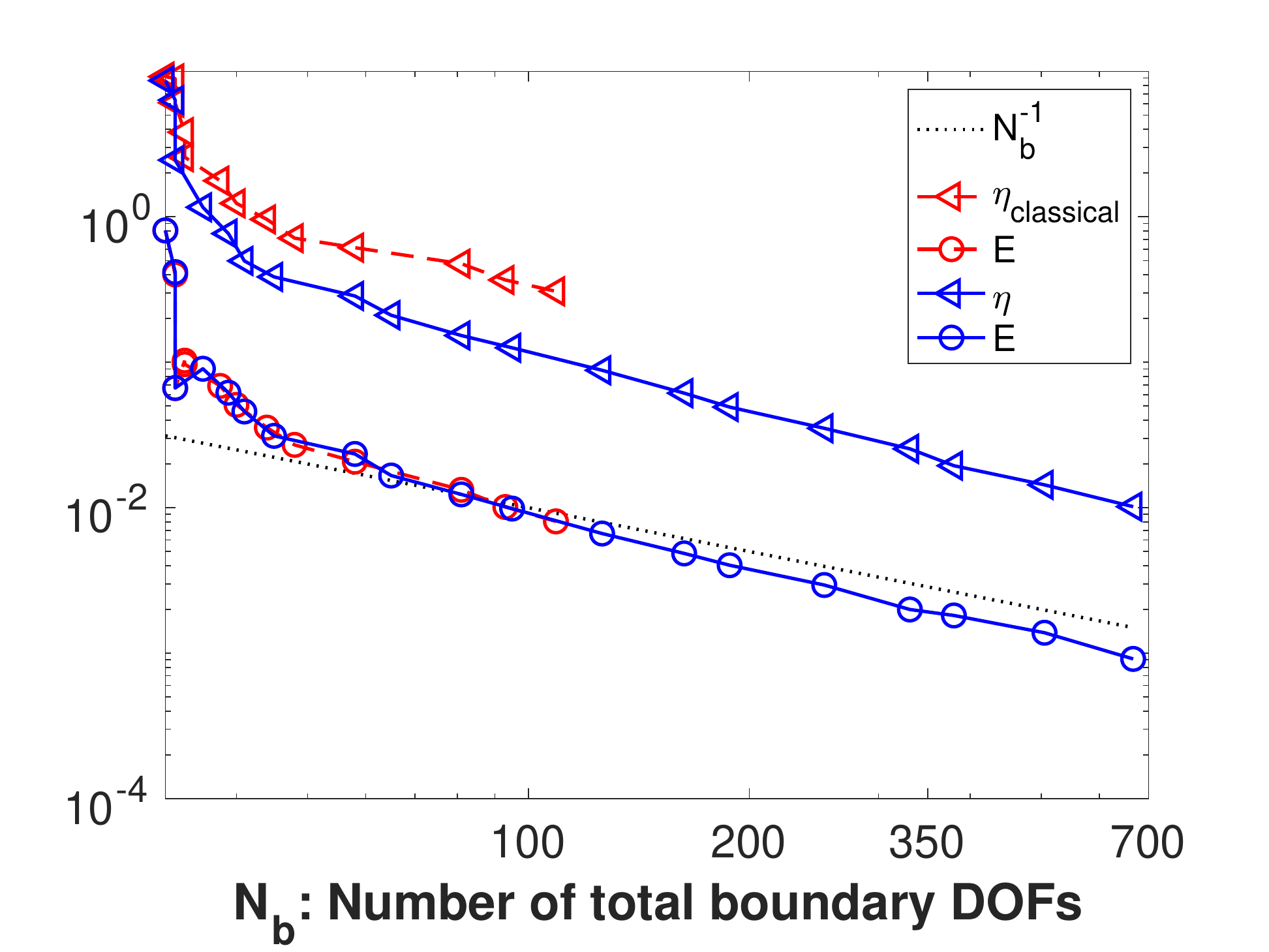}
&\includegraphics[width=0.30\textwidth]{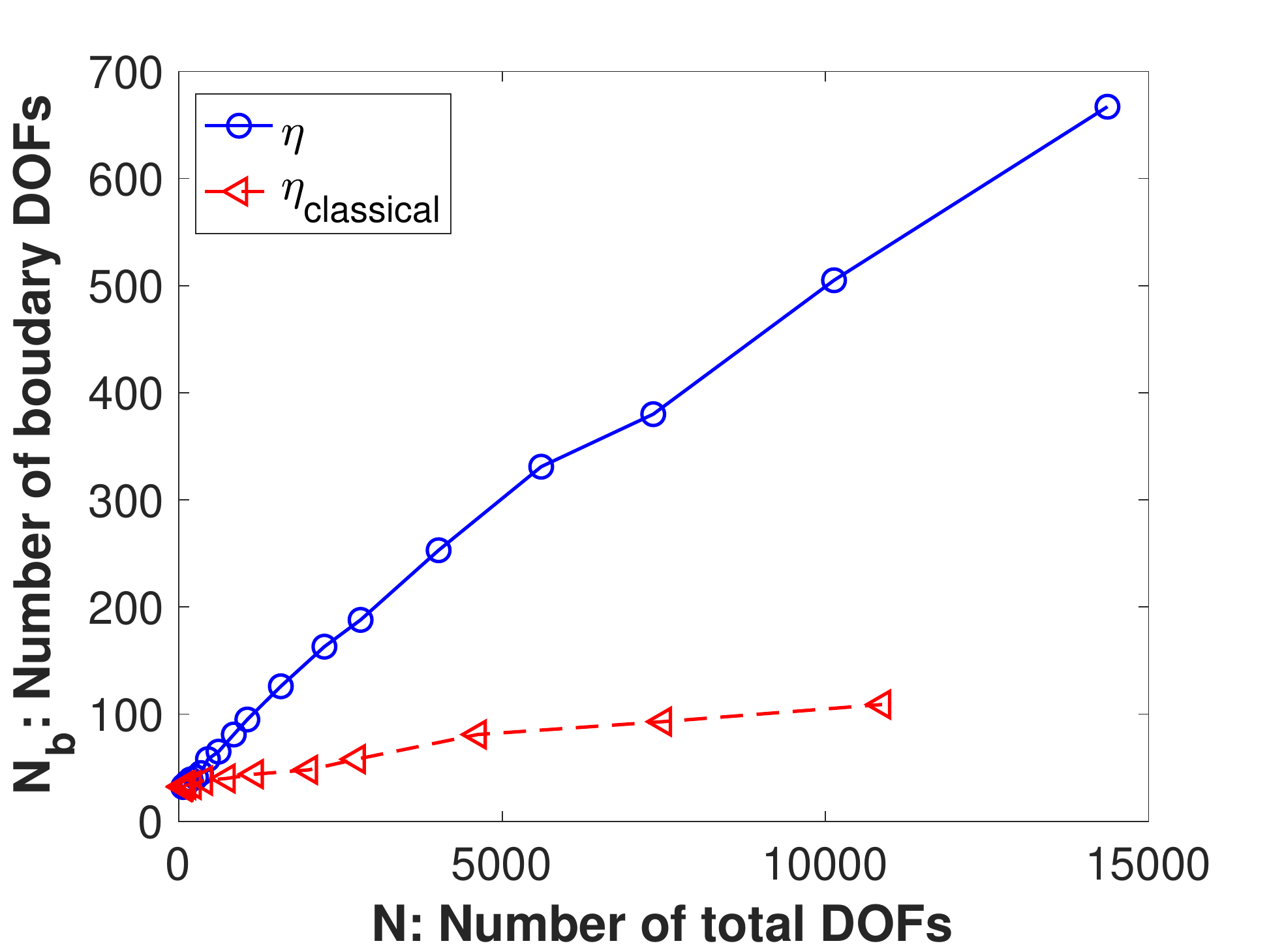}\\
& $k=1$, $C_2=1.0$& \\
\includegraphics[width=0.30\textwidth]{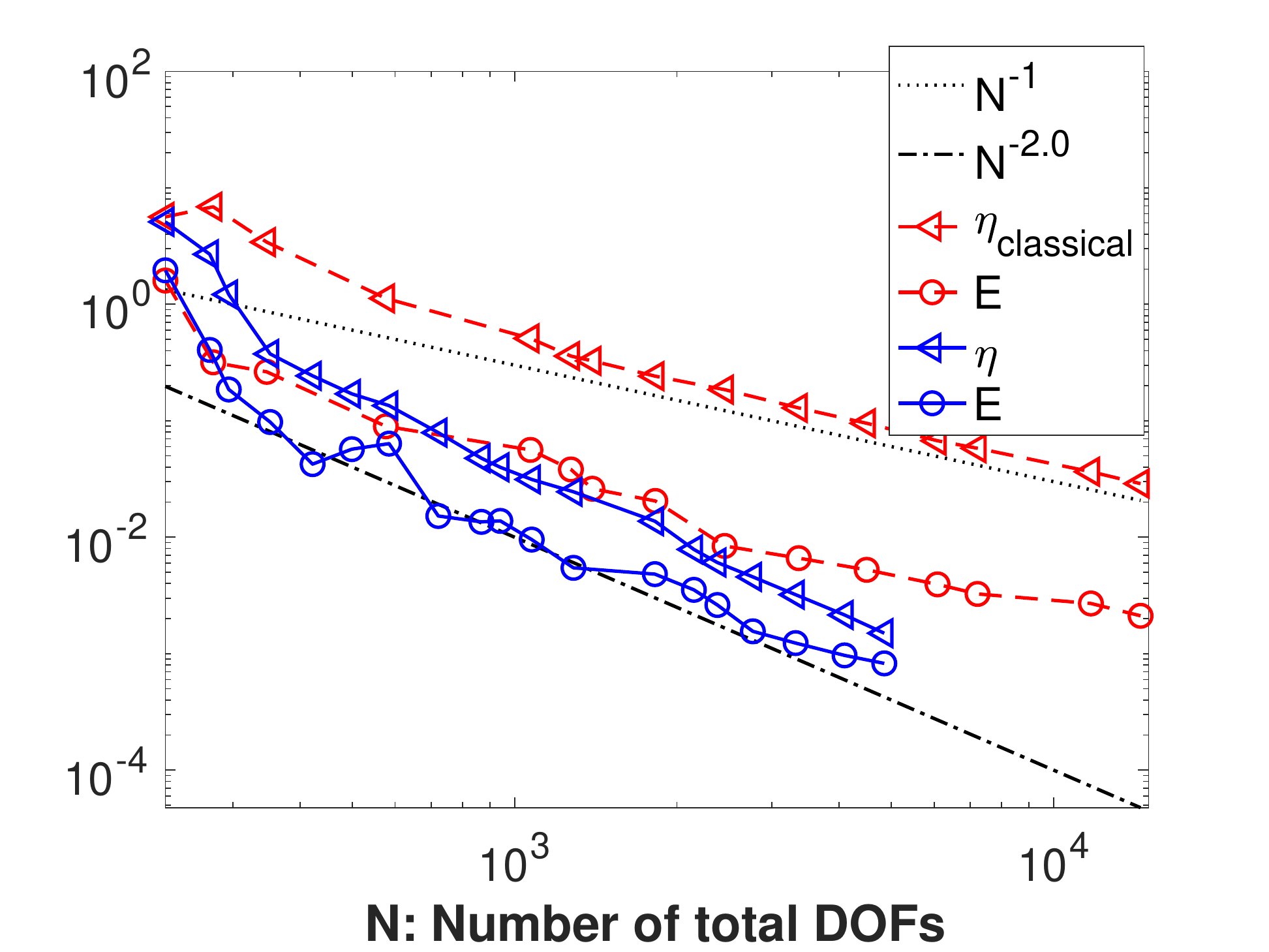}
&\includegraphics[width=0.30\textwidth]{Ex4-Nitsche-k=2-c2=0.1_error-a.pdf}
&\includegraphics[width=0.30\textwidth]{Ex4-Nitsche-k=2-c2=0.1_error-a.pdf}\\
& $k=2$, $C_2=0.1$& \\
\end{tabular}
\caption{\cref{ex3}.  Convergence comparison for Nitsche's method}
\label{Ex3-Nitsche-error}
\end{figure}

\

\textcolor{blue}{Comparing the performance of the estimator in the three examples we see that the adaptive procedure based on the dual wighted residual performs always better than the one based on the classical error estimator. If the solution is smooth, the results obtained by the AMR based on the new estimator are, in terms of error vs number of degrees of freedom, as good as the ones obtained by  using boundary concentrated meshes (of course, in this case, this last method is cheaper, as the mesh is designed a priori and the problem is solved only once). Our adaptive method is particularly advantageous when the solution presents singularities on or close to the boundary (which boundary concentrated meshes, based on a priori analysis, cannot tackle efficiently).}

\bibliographystyle{siamplain}
\bibliography{apost}
\end{document}